\documentclass[11pt, reqno]{amsart}

\usepackage[margin=3.5cm]{geometry}
\usepackage{textcmds} 
\usepackage{amsmath}
\usepackage{amstext}
\usepackage{amsthm}
\usepackage{amsfonts}
\usepackage{amssymb}
\usepackage{xspace}
\usepackage{microtype}
\usepackage[all]{xy}
\usepackage[modulo]{lineno}
\usepackage[usenames]{color}
\usepackage{aliascnt}
\usepackage{enumitem}
\usepackage[centertags]{amsmath}
\usepackage{rotating} 
\usepackage{dsfont}
\usepackage{bm}
\usepackage{subfigure}
\usepackage{array}
\usepackage[all]{xy}
\usepackage{euscript}
\usepackage[T1]{fontenc}
\usepackage{mathbbol}
\usepackage{nccmath}
\usepackage{marginnote}
\usepackage{stmaryrd}
\usepackage{tikz}
\usepackage{tikz-cd} 
\usepackage{graphics,graphicx}  
\usepackage[colorlinks=true,linkcolor=black,citecolor=blue,urlcolor=blue,citebordercolor={0 0 1},urlbordercolor={0 0 1},linkbordercolor={0 0 1}]{hyperref}

\let\oldtocsection=\tocsection
\let\oldtocsubsection=\tocsubsection

\renewcommand{\tocsection}[2]{\hspace{0em}\oldtocsection{#1}{#2}}
\renewcommand{\tocsubsection}[2]{\hspace{1em}\oldtocsubsection{#1}{#2}}

\newcommand{\veca}{{\vec{a}}}
\newcommand{\R}{\mathbb{R}}
\newcommand{\bdy}{\partial}
\newcommand{\hooksymp}{\overset{s}\hookrightarrow}
\newcommand{\eps}{\varepsilon}
\newcommand{\ham}{\op{Ham}}
\newcommand{\op}[1]{{\operatorname{#1}}}
\newcommand{\ovl}{\overline}
\newcommand{\setm}{\;\setminus\;}
\newcommand{\NI}{{\noindent}}
\newcommand{\wh}{\widehat}
\newcommand{\calJ}{\mathcal{J}}
\newcommand{\nil}{\varnothing}

\newcommand{\ga}{\gamma}
\newcommand{\calM}{\mathcal{M}}
\newcommand{\Ga}{\Gamma}
\newcommand{\cz}{\op{CZ}}
\newcommand{\ind}{\op{ind}}
\newcommand{\calA}{\mathcal{A}}
\newcommand{\vecb}{{\vec{b}}}
\newcommand{\stab}{\op{stab}}
\newcommand{\uvl}{\underline}
\newcommand{\sss}{\vspace{2.5 mm}}
\newcommand{\Z}{\mathbb{Z}}
\newcommand{\sht}{\mathfrak{s}}
\newcommand{\lc}{\lceil}
\newcommand{\rc}{\rceil}
\newcommand{\lng}{\mathfrak{l}}
\newcommand{\Q}{\mathbb{Q}}
\newcommand{\pqp}{(p/q)^+}
\newcommand{\pqm}{(p/q)^-}
\newcommand{\CP}{\mathbb{CP}}
\newcommand{\calB}{\mathcal{B}}
\newcommand{\te}{{t_e}}
\newcommand{\wt}{\widetilde}

\newcommand{\lf}{\lfloor}
\newcommand{\rf}{\rfloor}

\newcommand{\C}{\mathbb{C}}

\newcommand{\ra}{\rightarrow}

\newcommand{\blue}[1]{{\color{blue}#1}}

\newcommand{\Tcount}{\mathbf{T}}

\newcommand{\MS}{{\medskip}}

\newcommand{\la}{\lambda}
\newcommand{\orb}{\mathfrak{o}}
\newcommand{\std}{\op{std}}
\newcommand{\res}{\op{Res}}
\newcommand{\Ddiv}{\mathbf{D}}

\newcommand{\vecD}{{\vec{\Ddiv}}}
\newcommand{\T}{\mathcal{T}}
\newcommand{\vecm}{{\vec{m}}}
\newcommand{\pt}{\op{pt}}
\renewcommand{\lll}{\Langle}
\newcommand{\rrr}{\Rangle}

\newcommand{\vecv}{{\vec{v}}}
\newcommand{\sk}{\op{sk}}
\newcommand{\hor}{\op{hor}}
\newcommand{\ver}{\op{ver}}

\newcommand{\CC}{\mathcal{C}}
\newcommand{\po}{x_0}
\newcommand{\qo}{y_0}
\newcommand{\pp}{x}
\newcommand{\Da}{\Delta}
\newcommand{\perf}{\op{Perf}}
\newcommand{\perfbar}{\uvl{\perf}}
\newcommand{\bl}{{\op{Bl}}}

\newcommand{\mult}{{\op{mult}}}
\newcommand{\intE}{{\mathring{E}}}
\newcommand{\En}{\mathcal{E}}

\newcommand{\simp}{\nu}

\newcommand{\nn}{\nonumber}
\newcommand{\MM}{{\mathbb{M}}}
\newcommand{\D}{\mathbb{D}}

\newcommand{\Dpunc}{\D_\circ}
\newcommand{\mr}{\mathring}
\newcommand{\exc}{\mathbb{E}}
\newcommand{\F}{\mathbb{F}}
\newcommand{\bx}{\op{Box}}
\newcommand{\rect}{{\op{Rect}}}
\newcommand{\weight}{{\mathcal{W}}}

\tikzset{node distance=3cm, auto}

\makeatletter
\def\@secnumfont{\bfseries}
\def\section{\@startsection{section}{1}%
  \z@{.7\linespacing\@plus\linespacing}{.5\linespacing}%
  {\normalfont\Large\bfseries}}
  \def\acksection{\@startsection{subsubsection}{1}%
  \z@{.7\linespacing\@plus\linespacing}{.5\linespacing}%
  {\normalfont\Large\bfseries}}
\def\subsection{\@startsection{subsection}{2}%
  \z@{.75\linespacing\@plus.7\linespacing}{-.5em}%
  {\normalfont\large\bfseries}}
\def\subsubsection{\@startsection{subsubsection}{3}%
  \z@{.75\linespacing\@plus.7\linespacing}{-.5em}%
  {\normalfont\bfseries}}
\makeatother

\newtheorem{thm}{Theorem}[subsection]
\newtheorem{thmlet}{Theorem}

\newtheorem{lemma}[thm]{Lemma}
\newtheorem{prop}[thm]{Proposition}
\newtheorem{proplet}[thmlet]{Proposition}

\newtheorem{cor}[thm]{Corollary}
\newtheorem{corlet}[thmlet]{Corollary}

\newtheorem{notation}[thm]{Notation}
\newtheorem{claim}[thm]{Claim}

\newtheorem{assumptionlet}{Assumption}

\newtheorem*{assumptionstar}{Assumption (*)}

\newtheorem{conditionlet}{Condition}

\newtheorem{definition}[thm]{Definition}

\theoremstyle{remark}
\numberwithin{equation}{subsection} 

\newtheoremstyle{customremark}
{8pt}
{8pt}
{}
{}
{\bfseries}
{.}
{.5em}
{}
\theoremstyle{customremark}
\newtheorem{rmk_no_diamond}[thm]{Remark}
\newenvironment{rmk}{\begin{rmk_no_diamond} } {\hfill$\Diamond$ \end{rmk_no_diamond}}
\newtheorem{example_no_diamond}[thm]{Example}
\newenvironment{example}{\begin{example_no_diamond} } {\hfill$\Diamond$ \end{example_no_diamond}}

\begin{document}

\makeatother

\date{\today}

\title{Ellipsoidal superpotentials and singular curve counts}

\date{\today}

\begin{abstract}
Given a closed symplectic manifold, we construct invariants which count (a) closed rational pseudoholomorphic curves with prescribed cusp singularities and (b) punctured rational pseudoholomorphic curves with ellipsoidal negative ends.
We prove an explicit equivalence between these two frameworks, which in particular gives a new geometric interpretation of various counts in symplectic field theory.
We show that these invariants encode important information about singular symplectic curves and stable symplectic embedding obstructions. We also prove a correspondence theorem between rigid unicuspidal curves and perfect exceptional classes, which we illustrate by classifying rigid unicuspidal (symplectic or algebraic) curves in the first Hirzebruch surface.
\end{abstract}

\author{Dusa McDuff and Kyler Siegel}
\thanks{K.S. is partially supported by NSF grant DMS-2105578}

\maketitle

\tableofcontents

\section{Introduction}

\subsection{Motivation}

In algebraic and symplectic geometry one often considers curves which satisfy various types of geometric constraints. For instance, Gromov--Witten invariants count curves passing through one or more cycles representing chosen homology classes, and relative Gromov--Witten invariants further specify tangency conditions with a chosen divisor.
For certain applications, a distinguished role is played by constraints which are local near points, as these make no assumption on the ambient topology.
For example, curves with several generic point constraints are the subjects of the celebrated formulas of Kontsevich \cite{kontsevich1994gromov,KV} and Caporaso--Harris \cite{CaH}, and they form the basis of the $U$-map in embedded contact homology used to define the ECH capacities \cite{Hutchings_quantitative_ECH}.

More recently, curves with local tangency constraints have been used to put various obstructions on Lagrangain submanifolds and symplectic embeddings -- see e.g. \cite{CM1,CM2,tonk,McDuffSiegel_counting,mcduff2021symplectic}.
The local tangency constraint can be formulated in terms of a generic local divisor and it essentially amounts to specifying the $m$-jet of a curve at a marked point, for some $m \in \Z_{\geq 0}$.
In particular, for $m=0$ this reduces to an ordinary point constraint.

In this paper we introduce a family of local geometric constraints which naturally generalize local tangency constraints and are related to singularities of algebraic curves.
The resulting curve counts can be interpreted in several ways:
\begin{enumerate}[label=(\alph*)]
  \item\label{enum:multi} as closed curves satisfying local multidirectional tangency constraints
  \item\label{enum:cusp} as closed curves with prescribed cusp singularities
  \item\label{enum:punc} as punctured curves which are negatively asymptotic (in the sense of symplectic field theory) to Reeb orbits in ellipsoids
  \item\label{enum:relGW} as relative Gromov--Witten invariants with respect to certain nongeneric chains of divisors.
\end{enumerate}
Interpretation \ref{enum:multi} leads to a natural definition of these counts in the spirit of symplectic Gromov--Witten theory, while \ref{enum:cusp} will relate these to classical existence problems about singular algebraic curves, particularly those with one $(p,q)$ cusp singularity (this is locally modeled on $\{x^p+y^q = 0\} \subset \C^2$).
Interpretation \ref{enum:punc} embeds these counts into the framework of SFT and will allow us to use them to obstruct (stabilized) symplectic embeddings of ellipsoids.
Finally, \ref{enum:relGW} is related most directly to \ref{enum:cusp} via embedded resolution of curve singularities.

The main goals of this paper are:
\begin{itemize}
  \item to give rigorous self-contained definitions of these counts (under suitable assumptions)
  \item to formalize the above interpretations (a)-(d) into precise equivalences between invariants
  \item to discuss applications to symplectic embeddings and classifying singular (algebraic or symplectic) curves.
\end{itemize}
We will restrict to genus zero and primarily focus on constraints of an essentially (real) four-dimensional nature, although we allow target spaces of any dimension. 
For the remainder of this introduction we outline our main results in more detail.

\subsection{Ellipsoidal superpotentials}\label{subsec:ell_sup_intro}

We begin by discussing punctured curves with ellipsoidal negative ends. Symplectic field theory packages punctured curve counts in symplectizations and completed symplectic cobordisms into algebraic invariants which are known to provide powerful symplectic embedding obstructions.
Of particular importance for this paper are cobordisms of the form $M_\veca$, given by excising from a closed symplectic manifold $M$ a (rescaling of) the ellipsoid $E(\veca)$ with area factors $\veca = (a_1,\dots,a_n) \in \R_{>0}^n$ (see \S\ref{sec:robust}).
Given a homology class $A \in H_2(M)$, we denote by $\Tcount_{M,A}^\veca$ the SFT count of index zero planes in the symplectic completion $\wh{M}_\veca$ of $M_\veca$ with one negative end asymptotic to a Reeb orbit in $\bdy E(\veca)$.
In \cite{SDEP} this count is called the {\bf ellipsoidal superpotential}.

In general $\Tcount_{M,A}^\veca$ is a virtual count of pseudoholomorphic buildings in a compactified moduli space and takes values in $\Q$.
However, we will show that in favorable situations the count $\Tcount_{M,A}^\veca$ can be defined by classical pseudoholomorphic curve techniques and takes values in $\Z$. 
In order to give a precise formulation let us first introduce some relevant notation. For $\veca = (a_1,\dots,a_n) \in \R_{>0}^n$ rationally independent\footnote{That is, $a_1,\dots,a_n$ are linearly independent over $\Q$. We will frequently make this assumption out of convenience in order to ensure nondegenerate Reeb dynamics in $\bdy E(\veca)$.},
let $\orb^\veca_1,\orb^\veca_2,\orb^\veca_3,\dots$ denote the Reeb orbits in $\bdy E(\veca)$ in order of increasing action. In particular, the action of $\orb^\veca_k$ is given by 
$\calA(\orb^\veca_k) = \MM_k^\veca$,
where $\MM^\veca_k$ denotes the $k$th smallest positive integer multiple of one of $a_1,\dots,a_n$.\footnote{In general $\calA(\ga)$ will denote the action of a Reeb orbit $\ga$, and we use the notation $\MM_k^\veca$ when we wish to emphasize its combinatorial nature in the case of ellipsoids.}
Recall that each Reeb orbit in $\bdy E(\veca)$ is an iterate of one of the simple orbits $\simp_1,\dots,\simp_n$, where $\simp_i$ denotes the intersection of $\bdy E(\veca)$ with the $i$th complex axis.
We will denote the covering multiplicity of any Reeb orbit $\ga$ by $\mult(\ga)$.

Let $\calJ(M_\veca)$ denote the set of SFT admissible almost complex structures on $\wh{M}_\veca$ (see \S\ref{subsec:moduli_spaces}).
Given $J \in \calJ(M_\veca)$, let $\calM_{M_\veca,A}^J(\orb^\veca_{c_1(A)-1})$ denote the moduli space of $J$-holomorphic planes\footnote{These need not be embedded.} $u: \C \ra \wh{M}_\veca$
which lie in the homology class $A$ and are negatively asymptotic to $\orb^\veca_{c_1(A)-1}$.
Here $c_1(A) \in \Z$ denotes the first Chern number of $A$, and $\orb^\veca_{c_1(A)-1}$ is precisely the negative asymptotic Reeb orbit needed for this moduli space to have index zero.

The following theorem gives precise conditions under which the ellipsoidal superpotential is classically defined and integer-valued, in which case we get obstructions for stabilized symplectic embeddings of ellipsoids.
\begin{thmlet}[specialization of Theorem~\ref{thm:robust_counts} and Corollary~\ref{cor:stab_obs}]\label{thmlet:ell_sup}
Let $M^{2n}$ be a semipositive closed symplectic manifold, $A \in H_2(M)$ a homology class, and $\veca = (a_1,\dots,a_n)$ a rationally independent tuple satisfying Assumption (*), and such that $A$ and $\orb^\veca_{c_1(A)-1}$ have no common divisibility.\footnote{That is, there is no $\kappa \in \Z_{\geq 2}$ such that $\orb^\veca_{c_1(A)-1}$ is a $\kappa$-fold cover of another Reeb orbit and we have $A = \kappa B$ for some $B \in H_2(M)$. This holds for example if $c_1(A)$ and $\mult(\orb^\veca_{c_1(A)-1})$ are relatively prime.}
Then:
\begin{enumerate}[label=(\alph*)]
  \item For generic $J \in \calJ(M_\veca)$, the moduli space $\calM_{M_\veca,A}^J(\orb^\veca_{c_1(A)-1})$ is finite and regular, and its signed count is independent of the choice of generic $J$.
  \item Suppose that this count is nonzero, and further that $M \times \C^N$ is semipositive for some $N \in \Z_{\geq 0}$.
Then given any symplectic embedding $E(c \veca) \times \C^N \hooksymp M \times \C^N$ we must have 
$c \leq \frac{[\omega_M]\cdot A}{\MM^\veca_{c_1(A)-1}}$. 
\end{enumerate} 
\end{thmlet}
\NI Here $[\omega_M] \cdot A$ denotes the area of $A$ with respect to the symplectic form $\omega_M$ on $M$. 
Assumption (*) is our key numerical condition on $\veca$ and $c_1(A)$, formulated in \S\ref{subsec:main_rob_res}:
\begin{assumptionstar}[Assumptions \ref{assump:A} and \ref{assump:B}]\label{assump:star} 
Assume that for any $k \in \Z_{\geq 2}$ and $i_1,\dots,i_k \in \Z_{\geq 1}$ satisfying $\sum\limits_{s=1}^k i_s + k-1 \leq c_1(A)-1$ we have 
\begin{align*}
\MM^\veca_{i_1} + \cdots + \MM^\veca_{i_k} \leq \MM^\veca_{i_1+\cdots+i_k+k-1},
\end{align*}
with strict inequality if $\sum\limits_{s=1}^k i_s + k-1 = c_1(A)-1$.
\end{assumptionstar}
\begin{notation}
For $M,A,\veca$ as in Theorem~\ref{thmlet:ell_sup}, the ellipsoidal superpotential is
\begin{align*}
\Tcount_{M,A}^\veca := \# \calM_{M_\veca,A}^J(\orb^\veca_{c_1(A)-1}) \in \Z
\end{align*}
for generic $J \in \calJ(M_\veca)$.
\end{notation}

Special cases of Theorem~\ref{thmlet:ell_sup} in the case $M = \CP^2$ have appeared in e.g. \cite{HK,CGH,Mint,Ghost}.
Theorem~\ref{thmlet:ell_sup} is applied in \cite{CuspsStairs} to the monotone toric surfaces from \cite{cristofaro2020infinite}, and in \cite{magill2022staircase} to one point blowups of $\CP^2$.
For the stabilized embedding obstructions in (b) it is essential that we count planes, as curves with higher genus or several negative ends do not typically behave well under stabilization (see \S\ref{subsec:stab_inv_I}).
Note that $M \times \C^N$ is semipositive e.g. whenever $M$ is monotone or $\dim_\R M = 4$ and $N=1$. 

\sss

 While Assumption (*) is somewhat mysterious in general, we have the following important specialization for which Assumption (*) always holds (see \S\ref{subsec:main_rob_res}).
Fix $p,q \in \Z_{\geq 1}$ with $p>q$ 
relatively prime satisfying $p+q = c_1(A)$, and put $\veca = (q,p\pm \delta,a_3,\dots,a_n)$ with $a_3,\dots,a_n > pq$ and $\delta > 0$ sufficiently small. In this case the Reeb orbit $\orb^\veca_{c_1(A)-1}$ is either $\simp_1^p$ or $\simp_2^q$, and in either case its action $\MM^\veca_{c_1(A)-1}$ is approximately $pq$.
In dimension four Assumption (*) is   closely related to the ECH partition conditions -- see Remark~\ref{rmk:ECH_partitions}.

\begin{rmk}
The obstruction in Theorem~\ref{thmlet:ell_sup}(b) actually follows from a stronger result (see \S\ref{sec:robust}) which obstructs embeddings of the form
$E(c\veca,c\vecb) \hooksymp M \times Q^{2N}$ for any $\vecb = (b_1,\dots,b_N)$  with $b_1,\dots,b_N > \MM^\veca_{c_1(A)-1}$ and $Q^{2N}$ any closed symplectic manifold such that $M \times Q$ is semipositive.
\end{rmk}

\begin{rmk}[on the assumptions in Theorem~\ref{thmlet:ell_sup}]
Semipositivity is a standard assumption in symplectic Gromov--Witten theory which is needed to rule out multiple covers of negative index -- see \cite{JHOL} and \S\ref{subsec:main_rob_res} below.
The non-divisibility assumption on $A$ and $\orb^\veca_{c_1(A)-1}$ is also needed is rule out multiple covers in the moduli space $\calM_{M_\veca,A}^J(\orb^\veca_{c_1(A)-1})$.

It is tempting to relax the Assumption (*) in Theorem~\ref{thmlet:ell_sup}, but this likely requires a more sophisticated transversality setup (and results in values in $\Q$ rather than $\Z$), due to the possibility of 
branched covers of trivial cylinders in $\R \times \bdy E(\veca)$ with nonpositive index (see Remark~\ref{rmk:highd_fail}). 
In fact, \cite[\S7]{SDEP} gives examples where $\Tcount_{M,A}^\veca$ cannot be an integer, e.g. $\Tcount_{\CP^2,5[L]}^{(1,8^+)} = 113/13$. 
Similar considerations will apply to Theorem~\ref{thmlet:multidir_counts} below.
\end{rmk}

\subsection{Multidirectional tangency constraints}\label{subsec:multi}

We next discuss closed curve counts with local multidirectional tangency constraints, denoted by $N_{M,A}\lll \CC^\vecm \pt \rrr$.
Roughly speaking, for $\vecm = (m_1,\dots,m_n)$ this constraint requires curves to pass through a specified point in $M^{2n}$ with contact order $m_i$ to the $i$th complex direction.
The most relevant case for us will be when $m_3 = \cdots = m_n = 1$, i.e. we only impose constraints in the first two complex directions.
For simplicity we restrict to curves carrying a single multidirectional tangency constraint, although this could be readily generalized to multiple constraints.

Given a symplectic manifold $M^{2n}$, we will say that a collection of smooth local symplectic divisors $\vecD = (\Ddiv_1,\dots,\Ddiv_n)$ through a point $\po \in M$ is {\bf spanning} if their tangent spaces at $\po$ span $T_{\po} M$. For example, the complex hyperplanes $\{z_i = 0\}$ give a set of spanning local divisors through $\vec{0} \in \C^n$.
We denote by $\calJ(M,\vecD)$ the space of all tame almost complex structures on $M$ which are integrable near $\po$ and preserve each of $\Ddiv_1,\dots,\Ddiv_n$.

Given $J \in \calJ(M,\vecD)$, we define $\calM_{M,A}^J\lll \CC_\vecD^\vecm \po \rrr$ to be the moduli space of $J$-holomorphic maps $u: \CP^1 \ra M$ such that $[u] = A$, $u([0:0:1]) = \po$, and $u$ has contact order at least $m_i$ with $\Ddiv_i$ at $[0:0:1]$ for $i = 1,\dots,n$.

\begin{thmlet}[specialization of Theorem~\ref{thm:N_robust}]\label{thmlet:multidir_counts}
Let $M^{2n}$ be a semipositive closed symplectic manifold, $A \in H_2(M)$ a homology class, $\vecD = (\Ddiv_1,\dots\Ddiv_n)$ a collection of spanning local divisors at a point $\po \in M$, 
and $\vecm = (p,q,1,\dots,1) \in \Z_{\geq 1}$ a tuple such that
$p+q = c_1(A)$ and $\gcd(p,q)=1$.
Then for generic $J \in \calJ(M,\vecD)$, the moduli space $\calM_{M,A}^J\lll \CC^{\vecm}_{\vecD}\po\rrr$ is finite and regular, and its signed count is independent of the choices of $\po,\vecD$, and generic $J$.
\end{thmlet} 
\begin{notation}
For $M,A,\vecm$ as in Theorem~\ref{thmlet:multidir_counts}, we put 
\begin{align*}
N_{M,A}\lll \CC^\vecm \pt\rrr := \#\calM_{M,A}^J\lll \CC^{\vecm}_{\vecD}\po\rrr \in \Z
\end{align*}
for generic $J \in \calJ(M,\vecD)$.
\end{notation}

\begin{rmk}
In the special case $m_2 = \cdots = m_n = 1$, the constraint $\lll \CC_\vecD^\vecm \po \rrr$ reduces to the local tangency constraint denoted by $\lll \T_{\Ddiv_1}^{(m_1)}\po\rrr$ in \cite{McDuffSiegel_counting}.
\end{rmk}

\begin{rmk}\label{rmk:multi=cusp_intro}
Suppose $\dim_\R M = 4$, and we have $p+q =c_1(A)$ and $\gcd(p,q)=1$.
As we make precise in \S\ref{subsec:cusps_multi_sing}, the local multidirectional tangency constraint $\lll \CC^{(p,q)}\pt\rrr$ is akin to prescribing a $(p,q)$ cusp at $\po$ along with its maximal jet. 
In particular, for generic $J \in \calJ(M,\vecD)$, every curve $C \in \calM^J_{M,A}\lll \CC^{(p,q)}_{\vecD}\po\rrr$ has a $(p,q)$ cusp singularity. Conversely, if $C$ is a $J$-holomorphic curve with a $(p,q)$ cusp at $\po$, then we can find spanning local divisors $\vecD = (\Ddiv_1,\Ddiv_2)$ at $\po$ such that the constraint $\lll \CC_{\vecD}^{(p,q)}\po\rrr$ is satisfied. Incidentally, the choice of $\Ddiv_2$ is irrelevant if we assume $p > q$.
\end{rmk}

\begin{example}
 The constraint $\lll \CC^{(3,2)}_{\vecD}\po\rrr$ corresponds to having an ordinary cusp at a specified point $\po$ and with specified tangent line. Recall that there is indeed a well-defined tangent line at an ordinary cusp even though it is a singular point.
\end{example}

Let us briefly comment on the proof of Theorem~\ref{thmlet:multidir_counts}.
Generalizing \cite[Prop. 2.2.2]{McDuffSiegel_counting} for local tangency constraints, the basic strategy is to show that any bad degenerations (i.e. multiple covers or stable maps with more than one component) only appear in real codimension at least two.
In the case of local tangency constraints, the main difficulty comes from configurations involving ghost bubbles (i.e. the constraint is carried by a constant component), and these are handled by observing (via \cite[Lem. 7.2]{CM1}) that the nearby nonconstant components satisfy tangency constraints which ``remember'' the main constraint.
In the case of local multidirectional tangency constraints this approach apparently fails to produce enough codimension, but we show in \S\ref{subsec:hc} that such ghost degenerations carry additional ``hidden constraints'':
\begin{proplet}[specialization of Proposition~\ref{prop:hid_constr}]\label{proplet:hid_constr}
  Suppose that a curve carrying the constraint $\lll \CC_\vecD^{(p,q)}\po\rrr$ degenerates into curves carrying constraints $\lll \CC_\vecD^{(p_1,q_1)}\po\rrr,\dots,\lll \CC_\vecD^{(p_k,q_k)}\po\rrr$ respectively. 
Then we must have
\begin{align}\label{eq:hc_ineq_intro}
\sum_{i=1}^k \min(ap_i,bq_i) \geq \min(ap,bq)
\end{align}
for any choice of $a,b \in \R_{>0}$.
\end{proplet}
\NI A more precise statement which applies in all dimensions is given in \S\ref{subsec:hc}.
Assuming $\gcd(p,q) = 1$, \eqref{eq:hc_ineq_intro} readily implies the inequality $\sum_{i=1}^k (p_i+q_i) \geq p+q+1$, which is the main ingredient needed to show that bad degenerations have codimension at least two.

\subsection{Correspondence theorem}\label{subsec:corresp}

We now formulate a precise relationship between negative ellipsoidal ends and multidirectional tangency constraints.
Recall that $\wh{M}_\veca$ has a negative end modeled on $\R_{\leq 0} \times \bdy E(\eps \veca)$, and that a curve $C \in \calM_{M_\veca,A}^J(\orb^\veca_{c_1(A)-1})$ has a negative end asymptotic to the Reeb orbit $\orb^\veca_{c_1(A)-1}$ in (a rescaling of) $\bdy E(\veca)$.
We can view the image of a small loop around the puncture as an embedded loop in $\bdy E(\veca)$ which has a well-defined linking number $m_i \in \Z$ with the $(2n-3)$-sphere $\bdy E(\veca) \cap \{z_i = 0\}$ for $i = 1,\dots,n$. 
Using results from \cite[\S4]{SDEP}, for a careful choice of $J$ we can directly transform $C$ into a $J'$-holomorphic sphere in $M$ satisfying the constraint $\lll \CC_{\vecD}^\vecm \po\rrr$ for suitable $J',\vecD$.

Furthermore, the linking numbers $\vecm = (m_1,\dots,m_n)$ associated to the puncture of $C$ are generically given by $\Da^\veca_{c_1(A)-1}$, which is defined combinatorially as follows.
\begin{definition}\label{def:Da}
For rationally independent $\veca \in \R_{>0}^n$  and $k \in \Z_{\geq 1}$, let $\Da_k^\veca$ denote the tuple $(i_1,\dots,i_n) \in \Z_{\geq 1}$ which maximizes $\min\limits_{1 \leq s \leq n} a_s i_s$ subject to $\sum\limits_{s=1}^n i_s = n+k-1$.  
\end{definition}
\begin{example}
For $\veca = (2,3^+)$, we have 
\begin{center}
\begin{tabular}{c|c|c|c|c|c|c|c|c} 
 $k$  & 1 & 2 & 3 & 4 & 5 & 6 & 7 & 8\\ 
 \hline
 $\Da^\veca_k$ & (1,1) & (2,1) & (2,2) & (3,2) & (4,2) & (4,3) & (5,3) & (5,4)\\
\end{tabular}.
\end{center}
\end{example}

The lattice path $\Da^\veca = (\Da^\veca_1,\Da^\veca_2,\Da^\veca_3,\dots)$ in $\Z_{\geq 1}^n$ will play a central role in relating negative ellipsoidal ends with multidirectional tangencies. It is the natural analogue of the lattice path $\Ga^\veca$ from \cite[\S1.1]{SDEP} which instead relates to positive ends (in fact we have $\Da^\veca_k = \Ga^\veca_{k-1} + (1,\dots,1)$).

Note that for $a_3,\dots,a_n \gg a_1,a_2$ we have $\Da^\veca_k = (p,q,1,\dots,1)$ for some $p,q \in \Z_{\geq 1}$. In this case it is easy to check that $\orb^\veca_k$ is either $\simp_1^p$ or $\simp_2^q$, depending on whether $a_1 p$ or $a_2 q$ is smaller.

\begin{thmlet}[specialization of Theorem~\ref{thm:equivalence}]\label{thmlet:cusp=ell}
Let $M^{2n}$ be a semipositive symplectic manifold, $A \in H_2(M)$ a homology class, and $\veca = (a_1,\dots,a_n) \in \R_{>0}^n$ a rationally independent tuple with $a_3,\dots,a_n > pq$.
Put $(p,q,1\dots,1) := \Delta^\veca_{c_1(A)-1} =: \vecm$, and assume that $\gcd(p,q) = 1$.
Then we have $\Tcount_{M,A}^{\veca} = N_{M,A}\lll \CC^\vecm\pt \rrr$. 
\end{thmlet}


\begin{rmk}
Let $E_\sk^4 = E(a_1,a_2)$ with $a_2 \gg a_1$ denote the four-dimensional ``skinny ellipsoid'' (up to rescaling).
It is shown in \cite{McDuffSiegel_counting} that the local tangency constraint $\lll \T^{(m)}\pt\rrr$ is interchangeable with a negative end asymptotic to the Reeb orbit $\orb_k = \simp_1^k$ in $\bdy E_\sk^4$. 
Theorem~\ref{thmlet:cusp=ell} generalizes this to multidirectional tangency constraints and arbitrary ellipsoids.
\end{rmk}

\subsection{Singular symplectic and algebraic curves}

It is important to understand when $\Tcount_{M,A}^\veca$ is nonzero, as by Theorem~\ref{thmlet:ell_sup} this obstructs symplectic embeddings of the form $E(c\veca) \times \C^N \hooksymp M \times \C^N$. In fact, for $N \geq 1$, all such obstructions known to us are of this form, and at least in the case $M = \CP^2$ it is expected that these provide a complete set of obstructions.
The recent paper \cite{SDEP} provides tools to compute $\Tcount_{M,A}^\veca$ 
(and hence also $N_{M,A}\lll\CC^\vecm \pt \rrr$) 
by purely combinatorial methods, but many mysteries remain.

Using Theorem~\ref{thmlet:cusp=ell}, we recast this problem in \S\ref{subsec:cusps_multi_sing} in much more geometric terms, opening new avenues to prove nonvanishing results.
A {\bf $\mathbf{(p,q)}$-sesquicuspidal
 symplectic curve} in a symplectic four-manifold $M^4$ is a subset which has one singularity modeled on the $(p,q)$ cusp, and which is otherwise a positively immersed symplectic submanifold (see Definition~\ref{def:sesqsui_symp}).
One can show using the adjunction formula that the number of double points must be $\tfrac{1}{2}(2 + A\cdot A - c_1(A) - (p-1)(q-1))$.
In \S\ref{subsec:cusps_multi_sing} we prove:
\begin{thmlet}\label{thmlet:sing_symp_curve}
Fix $M^4$ a four-dimensional closed symplectic manifold, $A \in H_2(M)$ a homology class, and $p,q \in \Z_{\geq 1}$ with $\gcd(p,q) = 1$ and $p+q = c_1(A)$. Then we have $N_{M,A}\lll \CC^{(p,q)}\pt\rrr \in \Z_{\geq 0}$, with $N_{M,A}\lll \CC^{(p,q)}\pt\rrr > 0$ if and only if there exists a rational $(p,q)$-sesquicuspidal symplectic curve in $M$ lying in homology class $A$.
\end{thmlet}
Since any singularity of an algebraic curve can be symplectically perturbed into a finite set of positive double points, we have:
\begin{corlet}\label{corlet:alg_curv}
If $M^4$ is a smooth projective surface containing 
an index zero irreducible rational algebraic curve in homology class $A$ with a $(p,q)$ cusp (and possibly other singularities), then we have 
  $\Tcount_{M,A}^{(p,q)} > 0$.
\end{corlet}

This allows us to study $\Tcount_{M,A}^\veca$
 (or equivalently $N_{M,A}\lll\CC^\vecm \pt \rrr$)
by constructing algebraic curves and potentially importing techniques from algebraic geometry.
As an illustration, Orevkov \cite{orevkov2002rational} constructed an infinite sequence of index zero unicuspidal rational algebraic curves in $\CP^2$, so by Corollary~\ref{corlet:alg_curv} there is a corresponding infinite sequence of nonvanishing values for $\Tcount_{\CP^2,d[L]}^{(p,q)}$, and these turn out to give precisely 
the
symplectic embedding obstructions at the outer corners of the Fibonacci staircase \cite{McDuff-Schlenk_embedding} via Theorem~\ref{thmlet:ell_sup}(b).
In the forthcoming work \cite{CuspsStairs} we generalize Orevkov's construction in two independent directions, by constructing new families of index zero rational sesquicuspidal algebraic plane curves and by replacing $\CP^2$ with various other toric surfaces.

\begin{rmk}
 Under the assumptions of Theorem~\ref{thmlet:sing_symp_curve}, one can similarly show that we have $\Tcount_{M,A}^{(p,q)} > 0$ if and only if the transverse $(p,q)$ torus knot $\mathbb{T}_{p,q} \subset S^3$ with maximal self-linking number has a genus zero positively immersed symplectic hat (in the sense of \cite{etnyre2020symplectic}) in $M$ in homology class $A$.
\end{rmk}

\subsection{Unicuspidal curves and perfect exceptional classes}

Lastly, we discuss applications to existence questions for singular curves.
While the previous subsection highlights the relevance of index zero sesquicuspidal curves in symplectic four-manifolds, we focus here on the special case of curves which are {\bf unicuspidal} curves, i.e. having a $(p,q)$ cusp and no other singularities.

According to \cite{fernandez2006classification}, the aforementioned curves constructed by Orevkov are the {\em only} index zero unicuspidal rational algebraic plane curves.
Before generalizing this result we must recall a bit more terminology.
Firstly, given $p,q \in \Z_{\geq 1}$ with $\gcd(p,q) = 1$ and $p > q$, let $\weight(p,q) = (m_1,\dots,m_L) \in \Z_{\geq 1}^L$
 denote the corresponding {\bf weight sequence}. 
As we recall in \S\ref{subsec:res_sing}, this is related to the continued fraction expansion of $p/q$ and controls (among other things) the resolution of the $(p,q)$ cusp singularity.

Given a closed symplectic four-manifold $M^4$, a homology class $B \in H_2(M)$ is {\bf exceptional} if we have $B \cdot B = -1$, $c_1(B) = 1$, and $B$ is represented by a symplectically embedded two-sphere.
We will say that $A \in H_2(M)$ is \textbf{$\mathbf{(p,q)}$-perfect exceptional} if $\wt{A} := A - \sum_{i=1}^Lm_1e_1 - \cdots - m_Le_L \in H_2(\wt{M})$ is an exceptional class, where $\weight(p,q) = (m_1,\dots,m_L)$.
Here $\wt{M}$ is the $L$-point blowup of $M$, and we have the identification $H_2(\wt{M}) \cong H_2(M) \oplus \langle e_1,\dots,e_L\rangle$, where $e_1,\dots,e_L$ are exceptional classes.

In \S\ref{subsec:deg_to_nongen_bl} we prove:
\begin{thmlet}\label{thmlet:per_exc}
Fix a $M^4$ a symplectic four-manifold and $A \in H_2(M)$ a homology class.
There is an index zero rational $(p,q)$-unicuspidal symplectic curve in $M$ in homology class $A$ if and only if $A$ is $(p,q)$-perfect exceptional.
Moreover, in this case we have $N_{M,A}\lll \CC^{(p,q)}\pt\rrr = 1$.
\end{thmlet}

As an illustration of Theorem~\ref{thmlet:per_exc}, we note that the perfect exceptional homology classes in the first Hirzebruch surface $F_1$ were studied comprehensively in \cite{magill2022staircase} as part of their study of infinite staircases in the symplectic ellipsoid embedding functions of one-point blowups of $\CP^2$. 
Put $H_2(F_1) = \langle \ell,e\rangle$, where $\ell$ is the line class and $e$ is the exceptional divisor class.
Let $a_1,a_2,a_3,\dots$ be the sequence defined by the recursion $a_{j+6} = 6a_{j+3}-a_{j}$ with initial values $1,1,1,1,2,4$, and put
\begin{itemize}
  \item $t_j := \sqrt{a_{j+3}^2 - 6a_{j+3}a_j + a_j^2 + 8}$
  \item $d_j := \tfrac{1}{8}(3a_{j+3} + 3a_j + (-1)^{j+1}t_j)$
  \item $m_j := \tfrac{1}{8}(a_{j+3} + a_j + (-1)^{j+1}3t_j)$
\end{itemize}
for $j \in \Z_{\geq 1}$.
\begin{corlet}\label{corlet:F_1}
There is an index zero $(p,q)$-unicuspidal rational symplectic curve in $F_1$ in homology class $A = d\ell - me$ \text{and satisfying} $p/q < 3 + 2\sqrt{2}$ if and only if $(p,q,d,m) = (a_{j+3},a_j,d_j,m_j)$ for some $j \in \Z_{\geq 1}$.
\end{corlet}
\NI In fact, the constructions in \cite{CuspsStairs} show that each of these is realized by an algebraic curve.
This could be viewed as a version of symplectic isotopy problem for singular symplectic curves in the spirit of \cite{golla2019symplectic}.

Corollary~\ref{corlet:F_1} should be compared with (the index zero part of) \cite[Thm 1.1]{fernandez2006classification}, but with $\CP^2$ replaced by $F_1$.
The case $p/q > 3 + 2\sqrt{2}$ (sans algebraicity) can also be deduced from the results in \cite{magill2022staircase}, but it is considerably more complicated; we discuss this briefly in \S\ref{subsec:F_1}.

\acksection*{Acknowledgements}
K.S. benefited from many helpful discussions with Grisha Mikhalkin.



\section{Ellipsoidal superpotentials}\label{sec:robust}




The main technical result in this section is Theorem~\ref{thm:robust_counts}, which establishes conditions under which the ellipsoidal superpotential is {\em robust} (see Definition~\ref{def:robust}), which roughly means well-defined by classical transversality techniques and independent of any auxiliary choices.
Together with Corollary~\ref{cor:stab_obs}, this implies Theorem~\ref{thmlet:ell_sup}.

After some setting up some geometric preliminaries in \S\ref{subsec:small_ell} and \S\ref{subsec:moduli_spaces}, we state our main result and some of its consequences in \S\ref{subsec:main_rob_res}.
In \S\ref{subsec:pf_main_rob_res} we give the proof of Theorem~\ref{thm:robust_counts} assuming several lemmas about curves and their branched covers, which are proved in \S\ref{subsec:proofs_of_lemmas}.
In \S\ref{subsec:stab_inv_I} we give criteria under which moduli spaces are stabilization invariant.
Finally, in \S\ref{subsec:from_rob_to_emb}, we illustrate how (stable) symplectic embedding obstructions are naturally seen from this framework.


\subsection{Symplectic embeddings of small ellipsoids}\label{subsec:small_ell}

Let $M^{2n}$ be a closed symplectic manifold, and let $E(\veca)$ be the symplectic ellipsoid with area factors $\veca = (a_1,\dots,a_n) \in \R_{> 0}^n$.
For $\eps > 0$ sufficiently small, there exists a symplectic embedding $\iota$ of the scaled ellipsoid $E(\eps \veca) = E(\eps a_1,\dots,\eps a_n)$ into $M$.
In fact, up to further shrinking $\eps$, this embedding is unique up to Hamiltonian isotopy:
\begin{lemma}\label{lem:no_local_knotting}
Let $\iota_0,\iota_1: E(\eps\veca) \hooksymp M$ be symplectic embeddings.
Then for any $\eps' > 0$ sufficiently small there is a smooth family of symplectic embeddings $\{\iota'_t: E(\eps'\veca) \hooksymp M\;|\;t \in [0,1]\}$ such that $\iota'_i = \iota_i|_{E(\eps'\veca)}$ for $i =0,1$. 
\end{lemma}
\begin{proof}
 After post-composing with a Hamiltonian isotopy of $M$ which sends $\iota_1(0)$ to $\iota_0(0)$, we can assume $\iota_0(0) = \iota_1(0) =: p$. 
We can then find $\eps' > 0$ sufficiently small so that the images of both 
$\iota_0' := \iota_0|_{E(\eps'\veca)}$ and $\iota_1' := \iota_1|_{E(\eps'\veca)}$
lie in a Darboux ball in $M$ centered at $p$.
We thereby view $\iota_0',\iota_1'$ as symplectic embeddings $E(\eps'\veca) \hooksymp B \subset \R^{2n}$, where $B$ is a $2n$-dimensional ball of some radius centered at the origin.
After further shrinking $\eps'$, we can assume that $B$ also contains the standard $E(\eps'\veca) \subset \R^{2n}$.
It suffices to find a family of symplectic embeddings $E(\eps'\veca) \hooksymp B$ interpolating between $\iota_0'$ and $\iota_1'$.

By the ``extension after restriction'' principle (see e.g. \cite[\S4.4]{Schlenk_old_and_new}), $\iota_0': E(\eps'\veca) \hooksymp B^{2n}$ extends to a Hamiltonian diffeomorphism of $\R^{2n}$.
In particular, there is a Hamiltonian isotopy $\{\phi_t \in \ham(\R^{2n})\;|\; t \in [0,1]\}$ with $\phi_0 = \mathbb{1}$ such that $\phi_1|_{E(\eps'\veca)} = \iota_0'$.
Note that $\phi_t|_{E(\eps'\veca)}$ has image in $B$ for $t = 0,1$, but not necessarily for all $t \in (0,1)$.
However, an inspection of the construction of the Hamiltonian isotopy via the Alexander trick as (c.f. the explanation in \cite[\S4.4]{Schlenk_old_and_new}) shows that we can further assume $\phi_t(0) = 0$ for all $t \in [0,1]$. 
Therefore, after further shrinking $\eps'$, we can arrange that $\phi_t|_{E(\eps'\veca)}$ has image in $B$ for all $t \in [0,1]$.
By concatening this family with (the time reversal of) the analogous one given by applying the same considerations to $\iota'_1$, this concludes the proof.
\end{proof}

\begin{rmk}
The proof above actually applies equally if we replace the ellipsoid $E(\veca)$ with any star-shaped domain in $\R^{2n}$.  
\end{rmk}

\begin{rmk}\label{rmk:symp_iso_implies_ham_iso}
Since $H^1(E(\veca);\R) = 0$, given any family of symplectic embeddings $\{\iota_t': E(\eps'\veca) \hooksymp M\}$ 
  there is a Hamiltonian isotopy $\{\phi_t \in \ham(M)\}$ such that $\phi_0 = \mathbb{1}$ and $\iota_t' = \phi_t \circ \iota'_0$.
  In particular, there is a symplectomorphism between $M \setminus \iota_0'(\intE(\eps'\veca))$ and $M \setminus \iota_1'(\intE(\eps'\veca))$. Here $\intE(\veca)$ denotes the open ellipsoid $\{\pi \sum_{i=1}^n |z_i|^2/a_i < 1\}$.
\end{rmk}

\subsection{Moduli spaces of punctured curves}\label{subsec:moduli_spaces} 
Here we briefly discuss some geometric preliminaries and relevant notions about pseudoholomorphic curves, taking the occasion to set notation and terminology.

\subsubsection{Spaces with negative ellipsoidal ends}
Fix a closed symplectic manifold $M^{2n}$, a homology class $A \in H_2(M)$, and a vector $\veca = (a_1,\dots,a_n) \in \R_{> 0}^n$ whose components are rationally independent.
\begin{notation}
We put $M_\veca := M \,\setminus\, \iota(\intE(\eps\veca))$, where $\iota:  E(\eps\veca) \hooksymp M$ is a symplectic embedding for some $\eps > 0$.
\end{notation}
\NI Such an embedding $\iota$ always exists for $\eps > 0$ sufficiently small, e.g. with image contained in a Darboux chart.
We denote the symplectic completion of $M_\veca$ by 
\begin{align*}
\wh{M}_\veca := M_\veca \cup (\R_{\leq 0} \times \bdy M_\veca).
\end{align*}

\subsubsection{Almost complex structures}

Given a contact manifold $Y$ (typically $Y = \bdy E(\veca)$), we denote by $\calJ(Y)$ the space of admissible almost complex structures on the symplectization $\R \times Y$ (see e.g. \cite[\S2.1.2]{mcduff2021symplectic}).
In particular, any $J \in \calJ(Y)$ is invariant under translations in the first factor of $\R \times Y$.

Similarly, given a compact symplectic cobordism $X$ with positive contact boundary $\bdy^+X$ and negative contact boundary $\bdy^-X$, we denote by $\calJ(X)$ the space of admissible almost complex structures on the symplectic completion $\wh{X}$ of $X$. 
It will be more convenient to require $J \in \calJ(X)$ to be only tame rather than compatible on the compact part $X \subset \wh{X}$, which suffices for all of our purposes.

Our main example is $X = M_\veca$, which has empty positive boundary and negative boundary $\bdy M_\veca \cong \bdy E(\eps \veca)$ for some small $\eps > 0$.
By definition $J \in \calJ(M_\veca)$ is an almost complex structure on $\wh{M}_\veca$ which is tame on the compact piece $M_\veca$ and whose restriction to the negative end $\R_{\leq 0} \times \bdy M_\veca$ agrees with the restriction of some admissible almost complex structure $J|_{M_\veca} \in \calJ(\bdy M_\veca)$ on $\R \times \bdy M_\veca$.

\subsubsection{Moduli spaces}
Given a collection of Reeb orbits $\Gamma = (\ga_1,\dots,\ga_k)$ in $\bdy E(\veca)$ and an almost complex structure $J \in \calJ(M_\veca)$, we denote by
$\calM^J_{M_\veca}(\Gamma)$ the moduli space of $J$-holomorphic maps $u$ from a $k$-punctured Riemann sphere
to $\wh{M}_\veca$, such that $u$ is asymptotic at the $i$th puncture to $\ga_i$ for $i = 1,\dots,k$.
Here the conformal structure on the domain (or equivalently the locations of the punctures) is arbitrary, and elements of $\calM^J_{M_\veca}(\Ga)$ (usually referred to as simply ``curves'') are taken modulo biholomorphic reparametrization.


Since $H^1(E(\veca);\R) = H^2(E(\veca);\R) = 0$, any curve $C \in \calM^J_{M_\veca}(\Ga)$ can be uniquely filled so as to give a well-defined homology class $[C] \in H_2(M)$.
Given $A \in H_2(M)$, we put
$\calM^J_{M_\veca,A}(\Ga) := \{C \in \calM^J_{M_\veca}(\Ga)\;|\; [C] = A\}.$
We will sometimes suppress the almost complex structure from the notation and write simply $\calM_{M_\veca,A}(\Ga)$ when the choice of $J$ is implicit or immaterial.
We write $\calM_{M_\veca,A}(\ga)$ when $\Ga$ consists of a single Reeb orbit $\ga$.

Similarly, given collections of Reeb orbits $\Gamma^+ = (\ga_1^+,\dots,\ga^+_{k^+})$ and $\Gamma^- = (\ga^-_1,\dots,\ga^-_{k^-})$ in $\bdy E(\veca)$ and an almost complex structure $J \in \calJ(\bdy E(\veca))$, we denote by $\calM^J_{\bdy E(\veca)}(\Gamma^+;\Gamma^-)$ the moduli space of maps $u$ from a Riemann sphere having $k^+$ positive punctures and $k^-$ negative punctures to the symplectization $\R \times \bdy E(\veca)$, such that $u$ is asymptotic to the respective orbits $\Gamma^+$ at the positive punctures and to $\Gamma^-$ at the negative punctures.
Since $J \in \bdy E(\veca)$ is translation invariant, this moduli space inherits an $\R$ action which translates in the first factor of the target space, and we denote the quotient space by $\calM^J_{\bdy E(\veca)}(\Gamma^+;\Gamma^-)/\R$.

\subsubsection{Index formulas}

Given a homology class $A \in H_2(M)$ and Reeb orbits $\Gamma = (\ga_1,\dots,\ga_k)$ in $\bdy E(\veca)$, recall that the expected dimension of the moduli space $\calM_{M_\veca,A}(\Gamma)$ is given by 
\begin{align*}
\ind\, \calM_{M_\veca,A}(\Gamma)= (n-3)(2-k) + 2c_1(A) - \sum_{i=1}^k\cz(\ga_i).
\end{align*}
Here $c_1(A)$ is the first Chern number of $A$,
and $\cz(\ga_i)$ is the Conley--Zehnder index of $\ga_i$ with respect to the (unique up to homotopy) global trivialization of the contact hyperplane distribution over $\bdy E(\veca)$.
In particular, for the Reeb orbit $\orb_j^\veca$ in $\bdy E(\veca)$ we have
 $\cz(\orb_j^\veca) = n-1+2j$ (see e.g. equation~\eqref{eq:CZind} below and \cite[\S2.1]{Gutt-Hu}).
As usual we put $\dim M = 2n$.

By definition the ellipsoidal superpotential $\Tcount_{M,A}^\veca$ counts planes, i.e. $k=1$, and we have $\ind\, \calM_{M_\veca,A}(\ga) = 0$ if and only if $\cz(\ga) = n-3+2c_1(A)$, i.e. $\ga = \orb^\veca_{c_1(A)-1}$.
Standard transversality results (see e.g. \cite[\S8]{wendl_SFT_notes}) show that for generic $J \in \calJ(M_\veca)$ the moduli space $\calM^J_{M_\veca,A}(\ga)$ is smooth and of the expected dimension near any simple curve $C$. 

Similarly, given Reeb orbits $\Gamma^+ = (\ga_1^+,\dots,\ga^+_{k^+})$ and $\Gamma^- = (\ga_1^-,\dots,\ga^-_{k^-})$, the index of the moduli space $\calM_{\bdy E(\veca)}(\Gamma^+;\Gamma^-)/\R$ is given by
\begin{align*}
\ind\, \calM_{\bdy E(\veca)}(\Gamma^+;\Gamma^-) - 1 = (n-3)(2-k^+ - k^-) + \sum_{i=1}^{k^+}\cz(\ga^+_i) - \sum_{j=1}^{k^-}\cz(\ga^-_j) - 1.
\end{align*}
Standard transversality results show that, for generic $J \in \calJ(\bdy E(\veca))$, the moduli space $\calM_{\bdy E(\veca)}(\Gamma^+;\Gamma^-)/\R$ is smooth and of the expected dimension near any simple curve $C$ which is not a trivial cylinder.

We will also need to consider transversality in generic one-parameter families.
Given a one-parameter family of almost complex structures $\{J_t \in \calJ(M_\veca)\;|\; t \in [0,1]\}$, we consider the associated the parametrized moduli space
\begin{align*}
\calM_{M_\veca,A}^{\{J_t\}}(\Gamma) = \{(u,t)\;|\; t \in [0,1],\; u \in \calM^{J_t}_{M_\veca,A}(\Gamma)\}.
\end{align*}
If the family $\{J_t\}$ is generic, this is a smooth manifold of dimension
$2n-5 + 2c_1(A) - \sum_{i=1}^k\cz(\ga_i)$ near any simple curve (i.e. a pair $(u,t)$ with $u$ simple).
Given a one-parameter family $\{J_t \in \calJ(\bdy E(\veca))\;|\; t \in [0,1]\}$, we define the parametrized moduli space $\calM_{\bdy E(\veca)}^{\{J_t\}}(\Gamma^+,\Gamma^-)$ in a similar manner.
We denote the quotient by target translations by $\calM_{\bdy E(\veca)}^{\{J_t\}}(\Gamma^+,\Gamma^-) / \R$. For a generic one-parameter family $\{J_t\}$, this is a smooth manifold of dimension 
$2n-6 + \sum_{i=1}^{k^+}\cz(\ga_i^+) - \sum_{k=1}^{k^-}\cz(\ga_j^-)$ near any simple curve.

\subsubsection{Formal curves}



Following the usage in \cite{mcduff2021symplectic}, the language of formal curves provides a convenient bookkeeping tool for making index arguments.
The basic idea is that, for certain combinatorial purposes, we can formally glue together several curve components in a pseudoholomorphic building in order to treat them at a single formal curve. 
In such contexts, it is usually irrelevant whether or not an actual analytic gluing exists. 
  
Here we give definitions specialized to the situations most relevant for us. 
\begin{definition}
A genus zero {\bf formal curve component $C$ in $M_\veca$} is a pair $(\Gamma,A)$, where
  \begin{itemize}
     \item $\Gamma = (\ga_1,\dots,\ga_k)$ is a tuple of Reeb orbits in $\bdy E(\veca)$
     \item $A \in H_2(M)$ is a homology class
     \item we require the energy 
     \begin{align*}
    \En(C) := \int_A\omega - \sum_{i=1}^k\calA(\ga_i)
     \end{align*}
      to be nonnegative. 
   \end{itemize}  
\end{definition}
\NI Recall that $\calA(\ga)$ denotes the action of $\ga$, and for the Reeb orbit $\ga = \orb^\veca_k$ in $\bdy E(\veca)$ we have $\calA(\orb^\veca_k) = \MM^\veca_k$, where $\MM^\veca_k$ is the $k$th smallest element of the multiset $\{ia_j\;|\; i \in \Z_{\geq 1}, 1 \leq j \leq n\}$, or equivalently
$$
\MM^\veca_k = \min\limits_{\substack{i_1,\dots,i_n \in \Z_{\geq 0},\\i_1+\cdots+i_n = k}}\;\max\limits_{1 \leq s \leq n} a_si_s
$$
 (see \cite[\S1.2]{Gutt-Hu}).
 The index of $C$ is defined to be
\begin{align}\label{eq:indC}
\ind(C) = (n-3)(2-k) + 2c_1(A) - \sum_{i=1}^k \cz(\ga_i),
\end{align}
with $ \cz(\ga_i)$ calculated as in \eqref{eq:CZind}.

\begin{definition}
Similarly, a genus zero {\bf formal curve component $C$ in $\bdy E(\veca)$} is a pair $(\Gamma^+,\Gamma^-)$, where
\begin{itemize}
  \item $\Gamma^+ = (\ga_1^+,\dots,\ga^+_{k^+})$ and $\Gamma^- = (\ga^-_1,\dots,\ga^-_{k^-})$ are tuples of Reeb orbits in $\bdy E(\veca)$
  \item we require the energy 
  \begin{align}\label{eq:formal_symp_energy}
  \En(C) := \sum_{i=1}^{k^+}\calA(\ga_i^+) - \sum_{j=1}^{k^-}\calA(\ga_j^-)
  \end{align}
   to be nonnegative.
\end{itemize}  
\end{definition}
The index of $C$ is given by
\begin{align}\label{eq:indCcob}
\ind(C) = (n-3)(2-k^+-k^-) + \sum_{i=1}^{k^+}\cz(\ga_i^+) - \sum_{j=1}^{k^-}\cz(\ga_j^-).
\end{align}
\NI A formal cylinder in $\bdy E(\veca)$ is a formal curve with one positive end and one negative end (i.e. $k^+ = k^- = 1$), and it is {\bf trivial} if $\Gamma^+ = \Gamma^-$.

\subsection{Main robustness result}\label{subsec:main_rob_res}


We now discuss notions of robustness for moduli spaces.
In the following, $M^{2n}$ is a closed symplectic manifold, $A \in H_2(M)$ is a homology class, $\veca = (a_1,\dots,a_n) \in \R_{>0}^n$ is a rationally independent tuple, and we put
$M_\veca = M \setminus \iota(\intE(\eps \veca))$, which implicitly depends on a choice of $\eps > 0$ and $\iota: E(\eps\veca) \hooksymp M$.
\begin{definition}\label{def:robust}
We will say that $\calM_{M_\veca,A}(\ga)$ is \textbf{robust} if $\calM_{M_\veca,A}^J(\ga)$ is finite and regular for any choice of $\eps,\iota$, and generic $J \in \calJ(M_\veca)$, and moreover the (signed) count $\# \calM_{M_\veca,A}^J(\ga)$ is independent of these choices.

We will say that $\calM_{M_\veca,A}(\ga)$ is \textbf{strongly robust} if it is robust and moreover we have $\ovl{\calM}_{M_\veca,A}^J(\ga) = \calM_{M_\veca,A}^J(\ga)$ for any choice of $\eps,\iota$, and generic $J \in \calJ(M_\veca)$.
\end{definition}
\NI Here $\ovl{\calM}_{M_\veca,A}^J(\ga)$ denotes the SFT compactification of $\calM_{M_\veca,A}^J(\ga)$ via pseudoholomorphic buildings \`a la \cite{BEHWZ}. Note that $\ovl{\calM}_{M_\veca,A}^J(\ga) = \calM_{M_\veca,A}^J(\ga)$ implies that $\calM_{M_\veca,A}^J(\ga)$ is already compact, and hence finite if it is regular with index zero.
If $\calM_{M_\veca,A}(\ga)$ is robust, we will write $\# \calM_{M_\veca,A}(\ga)$ to refer to the (signed) count $\# \calM_{M_\veca,A}^J(\ga)$ for any choice of generic $J \in \calJ(M_\veca)$.
We will say that a robust moduli space $\calM_{M_\veca,A}(\ga)$ is {\bf deformation invariant} if it remains robust under deformations of the symplectic form on $M$, and moreover the count $\# \calM_{M_\veca,A}(\ga)$ is unchanged under such deformations.
\begin{rmk}
Although robust moduli spaces are already well-suited for enumerative purposes, the added benefit of strong robustness is that it guarantees that the counts agree with their SFT counterparts, which a priori depend on the full SFT compactification (c.f. \cite[\S3.4]{mcduff2021symplectic}).
Strong robustness is also closely related to the notion of ``formal perturbation invariance'' utilized in \cite{mcduff2021symplectic}, but due to some technical differences in the setup we use a different term here to avoid confusion.
\end{rmk}

  Recall that a closed symplectic manifold $M^{2n}$ is {\bf semipositive} if any $A \in \pi_2(M)$ with positive symplectic area and $c_1(A) \geq 3-n$ satisfies $c_1(A) \geq 0$ (this is automatic if $\dim M \leq 6$).
The upshot is that, for generic $J$, all $J$-holomorphic spheres $C$ have $\ind(C) \geq 0$. Indeed, by standard transversality the underlying simple curve $\uvl{C}$ must have $\ind(\uvl{C}) \geq 0$, and hence $c_1(\uvl{C}) \geq 0$ by semipositivity, whence $\ind(C) \geq \ind(\uvl{C}) \geq 0$.

Before stating our main result on robust moduli spaces, we will need to make some numerical assumptions on the data $\veca$ and $c_1(A)$.

\begin{assumptionlet}\label{assump:A}
We have 
\begin{align*}
\MM^\veca_{i_1} + \cdots + \MM^\veca_{i_k} \leq \MM^\veca_{i_1+\cdots+i_k + k-1}
\end{align*}
for any $i_1,\dots,i_k \in \Z_{\geq 1}$ with $k \geq 2$ satisfying $\sum\limits_{s=1}^k i_s +  k-1 \leq c_1(A)-1$.
\end{assumptionlet}
\begin{assumptionlet}\label{assump:B}
We have 
\begin{align*}
\MM^\veca_{i_1} + \cdots + \MM^\veca_{i_k} < \MM^\veca_{c_1(A)-1}
\end{align*}
for any $i_1,\dots,i_k \in \Z_{\geq 1}$ with $k \geq 2$ satisfying $\sum\limits_{s=1}^k i_s +  k-1 = c_1(A)-1$.
\end{assumptionlet}
\begin{assumptionlet}\label{assump:C}
 The homology class $A$ and Reeb orbit $\orb^\veca_{c_1(A)-1}$ have no common divisibility. 
\end{assumptionlet} 

\begin{rmk}
 Assumption~\ref{assump:A} equivalently states that any index zero (rational) formal curve in $\bdy E(\veca)$ with one negative end $\orb^\veca_m$ must have nonpositive energy, provided that $m \leq c_1(A)-1$.
Incidentally, a formal curve with nonpositive energy necessarily has zero energy since by definition formal curves cannot have negative energy.
 Assumption~\ref{assump:B} equivalently states that there are no nontrivial index zero (rational) formal curves in $\bdy E(\veca)$ with one negative end $\orb^\veca_{c_1(A)-1}$.
 Note that for $\veca$ rationally independent a formal curve with one negative end can only have zero energy if all Reeb orbits involved are covers of the same underlying simple orbit.
\end{rmk}
\begin{rmk}\label{rmk:ECH_partitions}
Note that for $\veca$ rationally independent, 
 When $n=2$, Assumption~\ref{assump:B} is equivalent to the statement that $\orb^\veca_{c_1(A)-1}$ is maximal with respect to the Hutchings--Taubes partial order \cite[Def. 1.8]{HuT}, i.e. a negative end asymptotic to $\orb^\veca_{c_1(A)-1}$ satisfies the negative ECH partition condition (c.f. \cite[Ex. 3.14]{Hlect}). 
\end{rmk}

We can now state our main robustness result:
\begin{thm}\label{thm:robust_counts}
Let $M^{2n}$ be a semipositive closed symplectic manifold, $A \in H_2(M)$ a homology class, and $\veca = (a_1,\dots,a_n) \in \R_{>0}^n$ a rationally independent tuple such that Assumptions \ref{assump:A}, \ref{assump:B}, and \ref{assump:C} hold.
Then the moduli space $\calM_{M_\veca,A}(\orb^\veca_{c_1(A)-1})$ is strongly robust and deformation invariant.
\end{thm}
\NI Recall that we put $\Tcount_{M,A}^\veca := \# \calM_{M_\veca,A}(\orb^\veca_{c_1(A)-1})$.



\sss

Although we do not a priori have a uniform way to verify Assumptions~\ref{assump:A} and \ref{assump:B} in general, the next two lemmas give useful sufficient conditions. 
\begin{lemma}\label{lem:suff_for_assump_A}
If $a_3,\dots,a_n > \MM^{(a_1,a_2)}_{c_1(A)-1}$, then Assumption~\ref{assump:A} holds.  
\end{lemma}

\begin{lemma}\label{lem:suff_for_assump_B}
Fix $p,q \in \Z_{\geq 1}$ relatively prime such that $p+q = c_1(A)$, and suppose we have $\veca = (q,p\pm\delta,a_3,\dots,a_n)$ with $a_3,\dots,a_n > pq$ and $\delta \in \R_{>0}$ sufficiently small. Then Assumptions \ref{assump:A} and \ref{assump:B} hold.
\end{lemma}
\NI Note that Assumptions \ref{assump:A},\ref{assump:B},\ref{assump:C} depend only on $\veca$ up to scalar multiplication, i.e. if they hold for $\veca$ then they also hold for $\la  \veca$ for any $\la \in \R_{>0}$.

Deferring the proofs of these lemmas for the moment, let us consider the four-dimensional case $n=2$, or more generally $a_3,\dots,a_n \gg a_1,a_2$.
In this situation it is customary 
to take without loss of generality $(a_1,a_2) = (1,a)$ with $a > 1$,
and we refer to $\simp_1$ as the ``short orbit'' $\sht$ and $\simp_2$ as the ``long orbit'' $\lng$.
In the following, for $p/q \in \Q$ with $p>q$ and $\gcd(p,q)=1$ we put $\pqp := p/q+\delta$ and $\pqm := p/q -\delta$ for $\delta > 0$ arbitrarily small.
We will sometimes write $M_a$ in place of $M_{(1,a)}$, and $\orb^a_k$ in place of $\orb^{(1,a)}_k$, and so on.
In particular, assuming $p+q = c_1(A)$, note that if $a = \pqp$ then we have $\orb^a_{c_1(A)-1} = \sht^p$, while if $a = \pqm$ then we have $\orb^a_{c_1(A)-1} = \lng^q$.
We thus have:

\begin{cor}\label{cor:sht_and_lng_counts} Let $M^4$ be a closed symplectic four-manifold, and let $A \in H_2(M)$ be a homology class such that $c_1(A) = p+q$ for some $p> q \in \Z_{\geq 1}$ relatively prime.
\begin{enumerate}
  \item For $a = \pqp$, the moduli space $\calM_{M_a,A}(\sht^p)$ is strongly robust.
  \item For $a = \pqm$, the moduli space $\calM_{M_a,A}(\lng^q)$ is strongly robust.
\end{enumerate}
In either case, the signed count is a nonnegative integer.

More generally, the same is true if $M^{2n}$ is any semipositive closed symplectic manifold and we replace $M_a$ with $M_\veca$ for $\veca = (1,a,a_3,\dots,a_n)$ with $a_3,\dots,a_n > p$.
\end{cor}
\NI The last sentence of the first paragraph is a standard consequence of automatic transversality as in \cite{Wendl_aut} -- see \cite[\S5.2]{mcduff2021symplectic}.

\begin{rmk}\label{rmk:orbit_family}
If we put $a = p/q$, $\bdy E(1,a)$ has a two-parameter family of Reeb orbits which becomes the two orbits $\sht^p,\lng^q$ after slightly perturbing $a$,
and the moduli spaces described in cases (1) and (2) of Corollary~\ref{cor:sht_and_lng_counts} can be viewed as two different ways of Morsifying this family.
\end{rmk}

We note that the above discussion gives sufficient but not necessary conditions for Assumptions \ref{assump:A} and \ref{assump:B} to hold, as the following simple example illustrates.

\begin{example}
Consider the case with $\veca = (8,13,22)$ (or a small perturbation thereof) and $c_1(A) = 5$.
Then we have $\orb^\veca_{c_1(A)-1} = \simp_3$ and $\calA(\orb^\veca_{c_1(A)-1}) = 22$, and one readily checks that Assumptions \ref{assump:A} and \ref{assump:B} hold, even though the hypotheses of Lemma~\ref{lem:suff_for_assump_A} and Lemma~\ref{lem:suff_for_assump_B} do not hold (even after permuting $\veca$).
\end{example}

\sss

We end this subsection by proving Lemmas \ref{lem:suff_for_assump_A} and \ref{lem:suff_for_assump_B}.

\begin{proof}[Proof of Lemma~\ref{lem:suff_for_assump_A}]
Fix $i_1,\dots,i_k \in \Z_{\geq 1}$ satisfying $m := \sum\limits_{s=1}^k i_s + k-1 \leq c_1(A)-1$.
Considering the contrapositive of Assumption~\ref{assump:A}, it suffices to show that any formal curve $C$ in $\bdy E(\veca)$ with strictly positive energy and with positive ends $\orb_{i_1}^\veca,\dots,\orb^\veca_{i_k}$ and negative end $\orb^\veca_m$ must satisfy $\ind(C) > 0$.

Note that for any $i \leq c_1(A)-1$, the Reeb orbit 
$\orb_i^\veca$ is a multiple of either $\simp_1$ or $\simp_2$.
Suppose that we have $\orb_i^\veca = \nu_1^j$ for some $j \in \Z_{\geq 1}$.
Then we have
\begin{align*}
ja_1 = \calA(\nu_1^j) = \calA(\orb^\veca_i) \leq \calA(\orb^\veca_{c_1(A)-1}) < a_3,\dots,a_n
\end{align*}
and hence
\begin{align}\label{eq:CZind}
\cz(\orb_i^\veca) = \cz(\simp_1^j) &= (n-1) + 2j + 2\lf a_1 j/a_2\rf + \cdots + \lf a_1j/a_n\rf \\ \notag
&=(n-1) + 2j + 2\lf a_1 j/a_2\rf,
\end{align}
i.e. $a_3,\dots,a_n$ are too large to contribute to the Conley--Zehnder index of $\orb_i^\veca$.
Similarly, if $\orb_i^\veca = \simp_2^j$ for some $j \in \Z_{\geq 1}$ then we have
\begin{align*}
\cz(\orb_i^\veca) = \cz(\simp_2^j) = (n-1) + 2j + 2\lf a_2 j / a_1\rf. 
\end{align*}
It follows that for index purposes we can assume $n = 2$, and since we are assuming $\En(C) > 0$ we then have $\ind(C) \geq 2$ by Lemma~\ref{lem:4d_formal_symp_curves} below.
\end{proof}

\begin{lemma}\label{lem:4d_formal_symp_curves}
Fix $a_1,a_2 \in \R_{>0}$ rationally independent, and let $C$ be a (rational) formal curve in $\bdy E(a_1,a_2)$. Then we have $\ind(C) \geq 0$, and in fact $\ind(C) \geq 2$ unless $\En(C) = 0$.
\end{lemma}
\begin{proof}
Up to permutation we can write the  positive asymptotics of $C$ as $\nu_1^{i_1^+},\dots,\nu_1^{i_{k^+}^+},\nu_2^{j_1^+},\dots,\nu_2^{j_{l^+}^+}$ and the negative asymptotics as $\nu_1^{i_1^-},\dots,\nu_1^{i_{k^-}^-},\nu_2^{j_1^-},\dots,\nu_2^{j_{l^-}^-}$.
We then have
\begin{align}\label{eq:energy_4d_symp}
\En(C) = \sum_{s=1}^{k^+} i_s^+a_1 + \sum_{s=1}^{l^+} j_s^+a_2 - \sum_{s=1}^{k^-} i_s^-a_1 - \sum_{s=1}^{l^-} j_s^-a_2 \geq 0. 
\end{align}

  Slightly rearranging the terms in the index formula, we have
\begin{align*}
\ind(C) = -2 + 2\left( \sum_{s=1}^{k^+}i_s^+ + \sum_{s=1}^{l^+}\lceil j_s^+a_2/a_1\rceil - \sum_{s=1}^{k^-}i_s^- - \sum_{s=1}^{l^-}\lfloor j_s^-a_2/a_1\rfloor \right) \nn\\+ 2\left( \sum_{s=1}^{k^+}\lceil i_s^+a_1/a_2\rceil + \sum_{s=1}^{l^+}j_s^+ - \sum_{s=1}^{k^-}\lfloor i_s^- a_1/a_2\rfloor - \sum_{s=1}^{l^-}j_s^- \right).
\end{align*}
By ~\ref{eq:energy_4d_symp}, each of the terms in parentheses must be nonnegative, and at least one of them must be strictly positive (and hence at least $1$ by integrality) since $a_2/a_1$ is irrational.
If $\En(C)$ is strictly positive then both of the terms in parentheses must be strictly positive.

\end{proof}

\begin{proof}[Proof of Lemma~\ref{lem:suff_for_assump_B}]
To verify Assumption~\ref{assump:A}, by Lemma~\ref{lem:suff_for_assump_A} it suffices to establish $\MM_{p+q-1}^{(q,p\pm \delta)} < a_3,\dots,a_n$.
Note that for $\delta > 0$ sufficiently small we have $\MM_{p+q-1}^{(q,p + \delta)} = pq$ and $\MM_{p+q-1}^{(q,p-\delta)} = q(p-\delta)$.

We will verify Assumption~\ref{assump:B} in the case $\veca = (q,p+\delta,a_3,\dots,a_n)$, the case $\veca = (q,p-\delta,a_3,\dots,a_n)$ being nearly identical.
It suffices to show that any nontrivial formal (rational) curve $C$ in $\bdy E(\veca)$ with negative end $\orb^\veca_{p+q-1}$ and $\En(C) = 0$ satisfies $\ind(C) \geq 2$.

As in the proof of Lemma~\ref{lem:suff_for_assump_A}, it suffices to consider the case $n=2$.
Since $\orb^\veca_{p+q-1} = \simp_1^p$ and $\En(C) = 0$, we can assume that the positive ends of $C$ are 
$\simp_1^{p_1},\dots,\simp_1^{p_m}$ for some $p_1,\dots,p_m \in \Z_{\geq 1}$ with $m \geq 2$ and $p_1 + \cdots + p_m = p$.
Then we have
\begin{align*}
\ind(C) = 2\sum_{i=1}^m \lc \tfrac{p_ia_1}{a_2} \rc - 2\lc \tfrac{pa_1}{a_2}\rc
 = 2\sum_{i=1}^m \lc \tfrac{p_i q}{p+\delta}\rc - 2q.
\end{align*}
Note that $p_iq/p$ is not an integer for $i = 1,\dots,m$, since we have $\gcd(p,q)=1$ and $1 \leq p_i < p$. Therefore, for $\delta > 0$ sufficiently small we have 
\begin{align*}
\ind(C) = 2\sum_{i=1}^m \lc \tfrac{p_i q}{p}\rc -2q > 2\sum_{i=1}^m \tfrac{p_i q}{p} - 2q = 0.
\end{align*}
This completes the proof.
\end{proof}

\subsection{Proof modulo lemmas}\label{subsec:pf_main_rob_res}
In this subsection we give the proof of Theorem~\ref{thm:robust_counts}, modulo several lemmas whose proofs are deferred to \S\ref{subsec:proofs_of_lemmas}.
Our goal is to establish the following:

\begin{enumerate}[label=(\Roman*)]
  \item for any choice of $\eps > 0$, $\iota: E(\eps\veca) \hooksymp M$, and generic $J \in \calJ(M_\veca)$, the moduli space $\calM^{J}_{M_\veca,A}(\orb^\veca_{c_1(A)-1})$ is regular and equal to its SFT compactification
  \item the count $\# \calM^{J}_{M_\veca,A}(\orb^\veca_{c_1(A)-1})$ is independent of $\iota,\eps$ and generic $J$, and is invariant under deformations of the symplectic form of $M$.
\end{enumerate}

By the following claim, we can work with fixed $\eps,\iota$ throughout the argument.
\begin{claim}
Independence of $\#\calM^J_{M_\veca,A}(\ga_0)$ of the choice of $\eps$ and $\iota$ follows from independence of the choice of generic $J \in \calJ(M_\veca)$.
\end{claim}
\begin{proof}
For $i = 0,1$, consider symplectic embeddings $\iota_i:  E(\eps_i\veca) \hooksymp M$ for some $\eps_i > 0$, along with generic $J_i \in \calJ(M \setminus \iota_i(\intE(\eps_i \veca)))$.
By Lemma~\ref{lem:no_local_knotting}, we can find $\eps' > 0$ sufficiently small such that the restrictions $\iota_0' := \iota_0|_{E(\eps'\veca)}$ and $\iota_1' := \iota_1|_{E(\eps' \veca)}$ are isotopic through symplectic embeddings. 
Note that, for $i = 0,1$, the symplectic completions of $M \setminus\iota_i( \intE(\eps_i\veca))$ and $M \setminus \iota_i( \intE(\eps'\veca))$ are naturally identified such that we have an inclusion
$\calJ(M \setminus \iota_i( \intE(\eps_i\veca))) \subset \calJ(M \setminus \iota_i( \intE(\eps'\veca)))$.
Moreover, as in Remark~\ref{rmk:symp_iso_implies_ham_iso}, there is a symplectomorphism 
$M \setminus \iota_0( \intE(\eps'\veca)) \cong M \setminus \iota_1( \intE(\eps'\veca))$, and this sets up a bijection between $\calJ(M \setminus \iota_0( \intE(\eps'\veca)))$ and $\calJ(M \setminus \iota_1( \intE(\eps'\veca)))$.
Under these identifications, we can thus view $J_0,J_1$ as two generic almost complex structures on the symplectic completion of a fixed symplectic manifold $M_\veca$.
\end{proof}

Now fix a generic one-parameter family of almost complex structures $$\{J_t \in \calJ(M_\veca) \;|\; t \in [0,1]\},$$ and let
$$\{J_t^- \in \calJ(E(\veca))\;|\; t \in [0,1]\}$$
be the induced family given by restricting each $J_t$ to the negative end.
Here we could also allow a one-parameter family of symplectic forms on $M$, but to keep the exposition simpler we will assume the symplectic form on $M$ is fixed.

\MS

\begin{proof}[Proof of Theorem~\ref{thm:robust_counts}]

By Assumption~\ref{assump:C}, every curve in the parametrized moduli space
$\calM^{\{J_t\}}_{M_\veca,A}(\orb^\veca_{c_1(A)-1})$ is necessarily simple.
Therefore by genericity of the family $\{J_t\}$ we can assume that the moduli space $\calM^{\{J_t\}}_{M_\veca,A}(\orb^\veca_{c_1(A)-1})$ is regular and hence a smooth $1$-dimensional manifold.
Similarly, $\calM^{J_i}_{M_\veca,A}(\orb^\veca_{c_1(A)-1})$ is regular for $i = 0,1$.

Now consider the SFT compactification
\begin{align*}
\ovl{\calM}^{\{J_t\}}_{M_\veca,A}(\orb^\veca_{c_1(A)-1}) = \{(u,t)\;|\; t \in [0,1], u \in \ovl{\calM}^{J_t}_{M_\veca,A}(\orb^\veca_{c_1(A)-1})\}.
\end{align*}
To prove Theorem~\ref{thm:robust_counts}, it will suffice to show that 
$\ovl{\calM}^{\{J_t\}}_{M_\veca,A}(\orb^\veca_{c_1(A)-1}) = \calM^{\{J_t\}}_{M_\veca,A}(\orb^\veca_{c_1(A)-1})$, which implies that $\calM^{\{J_t\}}_{M_\veca,A}(\orb^\veca_{c_1(A)-1})$ is already compact.
In particular, this implies that $\ovl{\calM}^{J_i}_{M_\veca,A}(\orb^\veca_{c_1(A)-1})$ is equal to its SFT compactification for $i = 0,1$, which confirms (I).
Then $\calM^{\{J_t\}}_{M_\veca,A}(\orb^\veca_{c_1(A)-1})$ is a smooth $1$-dimensional compact cobordism between finite zero-dimensional manifolds $\calM^{J_0}_{M_\veca,A}(\orb^\veca_{c_1(A)-1})$ and $\calM^{J_1}_{M_\veca,A}(\orb^\veca_{c_1(A)-1})$, which confirms (II).

A priori, an element of 
$\ovl{\calM}^{\{J_t\}}_{M_\veca,A}(\orb^\veca_{c_1(A)-1})$ is a stable pseudoholomorphic building $\calB$ consisting of, for some $t_e \in [0,1]$,
\begin{itemize}
  \item a level in $\wh{M}_\veca$ consisting of one or more $J_\te$-holomorphic curve components
  whose total homology class is $A$
  \item some number (possibly zero) of symplectization levels $\R \times \bdy E(\veca)$, each consisting of one or more $J^-_\te$-holomorphic curve components, not all of which are trivial cylinders.
\end{itemize}
Our goal is to show that in fact there cannot be any symplectization levels, and that the level in $\wh{M}_{\veca}$ has a single component.
See Figure~\ref{fig:bldg_formal_gluing} for a cartoon.
\begin{figure}
    \centering
   \includegraphics[width=1\textwidth]{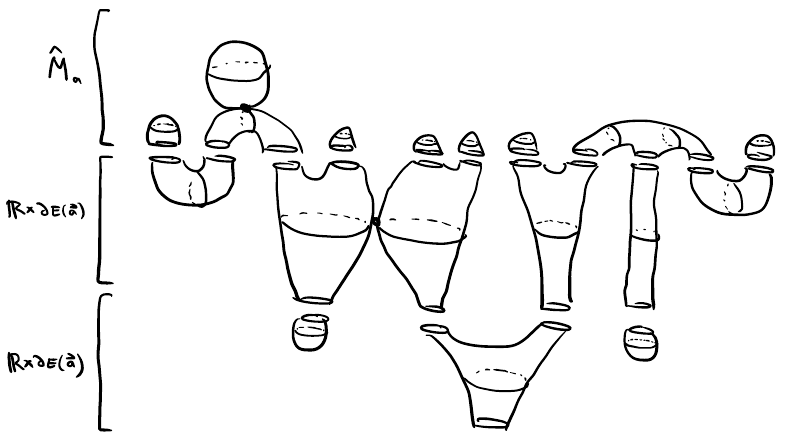}
    \caption{A potential building $\calB$ in $\ovl{\calM}^{\{J_t\}}_{M_\veca,A}(\orb^\veca_{c_1(A)-1})$.
After formally gluing as described above, $C_0$ has $4$ positive ends,
and there are $4$ planes $C_1,C_2,C_3,C_4$ and $2$  spheres $Q_1,Q_2$. Here one plane and
one sphere have components in all three levels, while the others
are entirely contained in $M_{\veca}$.
    }
    \label{fig:bldg_formal_gluing}
\end{figure}

After formally gluing various pairs of curve components in adjacent levels along shared Reeb orbits, we arrive at the following picture:
\begin{itemize}
  \item a formal curve component $C_0$ in $\bdy E(\veca)$ with $k$ positive ends and a single negative end $\orb^\veca_{c_1(A)-1}$, for some $k \in \Z_{\geq 1}$
  \item formal planes $C_1,\dots,C_k$ in $M_\veca$
  \item formal spheres $Q_1,\dots,Q_m$ in $M_\veca$, for some $m \in \Z_{\geq 0}$.
\end{itemize}
Here the positive ends of $C_0$ match up with the respective negative ends of $C_1,\dots,C_k$.
Note that each of $C_0,\dots,C_k,Q_1,\dots,Q_m$ is composed of some number of curve components of $\calB$.

\MS

In the next subsection we will prove that Assumptions \ref{assump:A} and \ref{assump:B} imply the following conditions:
\begin{conditionlet}\label{cond:A}
Let $J \in \calJ(M_\veca)$ be part of a generic one-parameter family.
 Let $C$ be a genus zero $J$-holomorphic curve component in $\wh{M}_\veca$ with $m$ negative ends, all but one of which bounds a formal plane in $E(\veca)$. Assume also that we have $c_1([C]) \leq c_1(A)$. Then, denoting the formal planes by $P_1,\dots,P_{m-1}$, we have
\begin{align*}
\ind(C) + \sum_{i=1}^{m-1}\ind(P_i) \geq 0,
\end{align*}
with equality only if $m=1$.
\end{conditionlet}
\begin{conditionlet}\label{cond:B}
Let $C$ be a genus zero formal curve in $\bdy E(\veca)$ with one negative end $\orb^\veca_{c_1(A)-1}$. 
Then we have $\ind(C) \geq 0$, with equality only if $\En(C) = 0$.
\end{conditionlet}
\begin{conditionlet}\label{cond:C}
Let $C$ be a genus zero formal curve in $\bdy E(\veca)$ with one negative end $\orb^\veca_{c_1(A)-1}$, and suppose that $\ind(C) = 0$ and $\En(C) = 0$. Then $C$ is a trivial cylinder.
\end{conditionlet}
\NI We also have the following elementary lemma, whose proof is deferred to the next subsection:
\begin{lemma}\label{lem:pos_c1}
Let $J \in \calJ(M_\veca)$ be part of a generic one-parameter family, and let $C$ be a genus zero $J$-holomorphic curve component in $\wh{M}_\veca$. Then we have $c_1(C) \geq 0$, and in fact $c_1(C) \geq 2$ if $C$ has any negative ends.
\end{lemma}

Taking these for granted for the moment, we complete the proof of Theorem~\ref{thm:robust_counts} as follows.
\MS

For $i = 1,\dots,k$, we claim that $\ind(C_i) \geq 0$, with $\ind(C_i) \geq 2$ unless $C_i$ is composed of a single $J_{t_e}$-holomorphic plane in $\wh{M}_\veca$.
To see this, note that a priori $C_i$ is composed of some number $m_i$ of $J_{t_e}$-holomorphic curve components in $\wh{M}_\veca$ and some number of $J_{t_e}^-$-holomorphic curve components in $\R \times \bdy E(\veca)$.
If $m_i = 1$, then we can view $C_i$ as a $J_{t_e}$-holomorphic curve component in $\wh{M}_\veca$ such that all but one of its ends bounds a formal plane in $\bdy E(\veca)$, whence the claim follows directly by Condition~\ref{cond:A}.
A similar argument holds in the case $m_i \geq 2$, noting that by Lemma~\ref{lem:pos_c1} any extra components in $\wh{M}_\veca$ will only push up the total index of $C_i$.

Now let $\calB_0$ denote the sub-building of $\calB$ given by throwing away all components corresponding to $Q_1,\dots,Q_m$.
Since $c_1(Q_1),\dots,c_1(Q_m) \geq 0$ by Lemma~\ref{lem:pos_c1},
the total index of all curve components satisfies 
\begin{align*}
\ind(\calB_0) \leq \ind(\calB) = 0.
\end{align*}
Since $\ind(C_0) \geq 0$ by \ref{cond:B} and $\ind(C_1),\dots,\ind(C_k) \geq 0$ by the above discussion, we must have $\ind(C_0) = \cdots = \ind(C_k)$, and hence
\begin{itemize}
  \item $\En(C_0) = 0$
  \item $C_i$ is a $J_{t_e}$-holomorphic plane in $\wh{M}_\veca$ for $i = 1,\dots,k$.
\end{itemize}
Moreover, by Condition~\ref{cond:C} $C_0$ must be a trivial formal cylinder, and in particular $k = 1$.
It also follows that the total index of $\calB_0$ is zero, and hence $c_1(Q_1) = \cdots = c_1(Q_m) = 0$, so by Lemma~\ref{lem:pos_c1} each $Q_i$ is actually a $J_{t_e}$-holomorphic sphere in $\wh{M}_\veca$.

In principle $C_0$ could be composed of several $J_{t_e}^-$-holomorphic curve components in $\R \times \bdy E(\veca)$ which formally glue together to give a cylinder. However, since the total energy must be zero, $C_0$ can only consist of trivial cylinders over $\orb^\veca_{c_1(A)-1}$.
This contradicts stability of the pseudoholomorphic building $\calB$.

It remains that $\calB$ does not have any symplectization levels, and consists of a $J_{t_e}$-holomorphic plane $C_0$ and $J_{t_e}$-holomorphic spheres $Q_1,\dots,Q_m$.
By a standard argument using semipositivity, we must have $m = 0$, which completes the proof.
Indeed, letting $\uvl{Q}_i$ denote the underlying simple curve of $Q_i$ 
for $i = 1,\dots,m$, we have
\begin{align*}
(\ind(\uvl{Q}_i)+2) + (\ind(C_0) + 2) = 2n-2,
\end{align*}
so by genericity of $\{J_t\}$ and standard tranversality techniques
$\uvl{Q}_i$ and $C_0$ cannot intersect for codimension reasons.
\end{proof}

\begin{rmk}\label{rmk:highd_fail}
If it easy to find counterexamples to Condition~\ref{cond:B} if we relax the Assumption~\ref{assump:A}.
For instance, put $\veca = (1,1+\delta_1,1+\delta_2)$ with $0 < \delta_1 < \delta_2 \ll 1$, and let $C$ be the branched cover of the trivial cylinder over $\simp_1$ in $\R \times \bdy E(\veca)$ having negative end $\simp_1^k$ and $k$ positive ends on $\simp_1$, for some $k \in \Z_{\geq 2}$.
We have $\cz(\simp_1^k) = 6k-2$ and hence $\ind(C) = 2-2k < 0$.
\end{rmk}

\subsection{Proof of lemmas}\label{subsec:proofs_of_lemmas}

In this subsection we now prove that Assumptions ~\ref{assump:A} and \ref{assump:B} imply Conditions \ref{cond:A},\ref{cond:B},\ref{cond:C}, and we also give the proof of Lemma~\ref{lem:pos_c1}.

\begin{lemma}
  Condition \ref{cond:A} holds under Assumption~\ref{assump:A}.
\end{lemma}
\begin{proof}
Let $\ga$ be the negative end of $C$ which does not bound a formal plane, and put $B := [C] \in H_2(M)$. Our goal is to establish
\begin{align}\label{eq:capped_cob_curve}
n - 3 + 2c_1(B) - \cz(\ga) \geq 0,
\end{align}
with equality only if $C$ has precisely one negative end.
Note that the left hand side of \eqref{eq:capped_cob_curve} can be viewed as the index of $C$ after capping off all of its other negative ends by formal planes.

Let $\kappa$ be the covering index of $C$ over its underlying simple curve $\uvl{C}$, with $\uvl{B} := [\uvl{C}] \in H_2(M)$, and let $\uvl{\ga}$ denote the negative end of $\uvl{C}$ which is covered by $\ga$. 
Since $\uvl{C}$ is simple, by the genericity assumption on $J$ we have $\ind(\uvl{C}) \geq -1$ and hence $\ind(\uvl{C}) \geq 0$ by parity considerations.
Since all formal planes in $E(\veca)$ have positive index, we have
\begin{align}\label{eq:uvlC_ind}
n - 3 + 2c_1(\uvl{B}) - \cz(\uvl{\ga}) \geq \ind(\uvl{C}) \geq 0,
\end{align}
where the first inequality is strict unless $\uvl{C}$ has precisely one negative end (i.e. $m=1$).
We will assume $\kappa \geq 2$, since otherwise $C = \uvl{C}$ and we are done.

Put $\ga = \orb^\veca_i$ and $\uvl{\ga} = \orb^\veca_{\uvl{i}}$ for some $i,\uvl{i} \in \Z_{\geq 1}$.
Using \eqref{eq:uvlC_ind} we have
\begin{align*}
n-1 + 2\uvl{i} = \cz(\uvl{\ga}) \leq n-3 + 2c_1(\uvl{B}),
\end{align*}
i.e. $\uvl{i} \leq c_1(\uvl{B}) - 1$, and hence $\kappa \uvl{i} + \kappa -1 \leq \kappa c_1(\uvl{B}) - 1$, with strict inequality unless $m=1$.

We can therefore apply Assumption~\ref{assump:A} in the case $k = \kappa$ and $i_1 = \cdots = i_k = \uvl{i}$ to get
\begin{align*}
\kappa \MM_{\uvl{i}}^\veca \leq \MM^\veca_{\kappa \uvl{i}+\kappa-1} \leq \MM^\veca_{\kappa c_1(\uvl{B})-1}.
\end{align*}
Since $\ga$ is at most a $\kappa$-fold cover of $\uvl{\ga}$, we also have 
\begin{align}\label{eq:calA_ga}
\calA(\ga) = \calA(\orb^\veca_i) \leq \kappa \calA(\orb^\veca_{\uvl{i}}) \leq \calA(\orb^\veca_{\kappa c_1(\uvl{B})-1}),
\end{align}
with the first inequality strict unless $\ga$ is precisely a $\kappa$-fold cover of $\uvl{\ga}$ and the second inequality strict unless $m = 1$.
Note that $\calA(\ga) \leq \calA(\orb^\veca_{\kappa c_1(\uvl{B})-1})$ is equivalent to $i \leq \kappa c_1(\uvl{B})-1$.

Finally, note that \eqref{eq:capped_cob_curve} is equivalent to $i \leq c_1(B) - 1$, i.e. $i \leq \kappa c_1(\uvl{B}) - 1$.
This holds by the above, with strictly inequality only if $m=1$ and $\ga$ is a $\kappa$-fold cover of $\uvl{\ga}$, in which case $C$ also has precisely one negative end.
\end{proof}

\begin{lemma}
  Condition \ref{cond:B} holds under Assumption~\ref{assump:A}.
\end{lemma}
\begin{proof}
Since $C$ has positive ends $\orb^\veca_{i_1},\dots,\orb^\veca_{i_k}$ and negative end $\orb^\veca_{c_1(A)-1}$, we have
\begin{align}\label{eq:half_index_C}
\tfrac{1}{2}\ind(C) = \sum_{s=1}^k i_s + k -1 - (c_1(A)-1).
\end{align}
If $\ind(C) = 0$, then by Assumption~\ref{assump:A} we have $\En(C) \leq 0$, and hence $\En(C) = 0$ since the energy must be nonnegative.

Otherwise, suppose by contradiction that we have $\ind(C) < 0$. Then $\sum\limits_{s=1}^k i_s + k -1 < c_1(A)-1$, so by Assumption~\ref{assump:A} we have 
\begin{align*}
\MM_{i_1}^\veca + \cdots + \MM_{i_k}^\veca \leq \MM^\veca_{i_1+\cdots+ i_k + k -1} < \MM^\veca_{c_1(A)-1},
\end{align*}
and hence $\En(C) < 0$, which is impossible.
\end{proof}
\begin{lemma}
 Condition \ref{cond:C} holds under Assumption~\ref{assump:B}. 
\end{lemma}
\begin{proof}
If $C$ has index zero is not a trivial cylinder, then \eqref{eq:half_index_C} together with Assumption~\ref{assump:B} implies that $\En(C) < 0$, which is a contradiction since formal curves by definition have nonnegative energy.
\end{proof}

\begin{proof}[Proof of Lemma~\ref{lem:pos_c1}]
  
Since the first Chern number is multiplicative under taking covers, we can assume that $C$ is simple.  
Let $k \in \Z_{\geq 0}$ denote the number of negative ends of $C$, and let $\ga_1,\dots,\ga_k$ denote its asymptotic Reeb orbits.
By genericity we have $\ind(C) \geq -1$, and hence $\ind(C) \geq 0$ since the index
\begin{align*}
\ind(C) = (n-3)(2-k) + 2c_1(C) - \sum_{i=1}^k \cz(\ga_i)
\end{align*}
is necessarily even.
If $C$ is a closed curve, i.e. $k = 0$,
then we have $c_1(A) \geq 3-n$, and hence $c_1(A) \geq 0$ by semipositivity. Otherwise, we have $k \geq 1$ and
\begin{align*}
2c_1(A) &\geq \sum_{i=1}^k \cz(\ga_i) + (k-2)(n-3) 
\\&\geq k(n+1) + (k-2)(n-3) \\
&= 2n(k-1) -2k + 6 \\
&\geq 2(k-1) - 2k + 6 = 4.
\end{align*}
\end{proof}

\subsection{Stabilization invariance I}\label{subsec:stab_inv_I}

We now discuss the effect on the ellipsoidal superpotential of stabilizing the target space, i.e. multiplying it with a given closed symplectic manifold $Q^{2N}$. 
This will be used in the next subsection to obstruct {\em stabilized} symplectic embeddings 
as in Theorem~\ref{thmlet:ell_sup}(b).

\begin{thm}\label{thm:ell_sup_stab}
Let $M^{2n}$ be a closed symplectic manifold, $A \in H_2(M)$ a homology class, and $\veca = (a_1,\dots,a_n)$ a rationally independent tuple.
Let $Q^{2N}$ be another closed symplectic manifold and $\vecb = (b_1,\dots,b_N)$ another tuple satisfying $b_1,\dots,b_N > \MM^\veca_{c_1(A)-1}$.
Then we have
\begin{align}\label{eq:Tcount_stab}
\Tcount_{M,A}^\veca = \Tcount_{M \times Q,A \times [\pt]}^{(\veca,\vecb)},
\end{align}
provided that the moduli spaces underlying both sides of \eqref{eq:Tcount_stab} are robust and deformation invariant. 
\end{thm}
\NI Here $(\veca,\vecb)$ denotes the concatenated tuple $(a_1,\dots,a_n,b_1,\dots,b_N)$.

In the special case of Corollary~\ref{cor:sht_and_lng_counts} we have:
\begin{cor}
Let $M^4$ be a closed symplectic four-manifold, and let $A \in H_2(M)$ be a homology class such that $c_1(A) = p+q$ for some $p,q \in \Z_{\geq 1}$ relatively prime. 
Let $Q^{2N}$ be another closed symplectic manifold such that $M \times Q$ is semipositive.
Then for any $b_1,\dots,b_N > pq$ and $\delta > 0$ sufficiently small we have
\begin{align*}
\Tcount_{M \times Q,A\times[\pt]}^{(q,p\pm\delta,b_1,\dots,b_N)} = \Tcount_{M,A}^{(q,p\pm \delta)}.
\end{align*}
\end{cor}
\begin{rmk}\label{rmk:low_dim_or_monotone}
  Note that $M \times Q$ is semipositive if e.g. $\dim(M) = 4$ and $\dim(Q) = 2$, or if $M$ is monotone and $\pi_2(Q) = 0$.
\end{rmk}

We naturally view any Reeb orbit $\ga$ in $\bdy E(\veca)$ as lying also in $\bdy E(\veca,\vecb)$ via the inclusion $\bdy E(\veca) \subset \bdy E(\veca,\vecb)$.
For $b_1,\dots,b_N \gg a_1,\dots,a_n$, the Conley--Zehnder index of $\ga$ as an orbit of $\bdy E(\veca,\vecb)$ is $N$ greater than its Conley--Zehnder index as an orbit of $\bdy E(\veca)$. Using the shorthand $M^\stab_\veca := (M \times Q)_{(\veca,\vecb)}$, we have by
\eqref{eq:indC} and \eqref{eq:CZind} that
\begin{align}\label{eq:stab_ind}
\ind\,\calM_{M_\veca^\stab,A \times [\pt]}(\ga) = \ind\,\calM_{M_\veca,A}(\ga).
\end{align}
This simple observation goes back to \cite{HK}, and is the key starting point for obstructing stabilized symplectic embeddings. 
Here is it crucial that we count planes. Indeed, by constrast note that for $\Gamma = (\ga_1,\dots,\ga_k)$ with $k \geq 2$ we have
  \begin{align*}
  \ind\,\calM_{M_\veca^\stab,A \times [\pt]}(\Ga) = \ind\,\calM_{M_\veca,A}(\Ga) + 2N(1-k) < \ind\,\calM_{M_\veca,A}(\Ga).
  \end{align*}

\sss

If $Q$ is a symplectic surface, Theorem~\ref{thm:ell_sup_stab} can be proved by a straightforward adaptation of the technique in \cite[\S3.6]{mcduff2021symplectic}.
In brief, we can arrange that $\wh{M}_\veca$ sits inside $\wh{M}^\stab_\veca$ as a symplectic divisor, and we work with $J^\stab \in \calJ(M_\veca^\stab)$ which preserves this divisor and restricts to some $J \in \calJ(M_\veca)$.
Then we have a natural inclusion
\begin{align}\label{eq:stab_I_inclusion}
\calM^J_{M_\veca,A}(\orb^\veca_{c_1(A)-1}) \subset \calM^{J^\stab}_{M^\stab_\veca,A \times [\pt]}.
\end{align}
In fact, any curve $C \in \calM^{J^\stab}_{M^\stab_\veca,A \times [\pt]}$ must be entirely contained in $\wh{M}_\veca$, since otherwise its intersection number with the divisor $\wh{M}_\veca$ must be zero for homological reasons, but also positive due to positivity of intersections and winding number estimates.
Furthermore, by \cite[\S A]{mcduff2021symplectic} (see also \cite{pereira2022equivariant}), the inclusion \eqref{eq:stab_I_inclusion} preserves regularity, i.e. any regular curve $C \in \calM^J_{M_\veca,A}(\orb^\veca_{c_1(A)-1})$ is also regular in $\calM^{J^\stab}_{M^\stab_\veca,A \times [\pt]}$.
Therefore we have 
\begin{align*}
\Tcount_{M\times Q,A \times [\pt]}^{(\veca,\vecb)} = \# \calM_{M_\veca^\stab,A\times[\pt]}^{J^\stab}(\orb^\veca_{c_1(A)-1}) = \# \calM_{M_\veca,A}^{J}(\orb^\veca_{c_1(A)-1}) = \Tcount_{M,A}^\veca,
\end{align*}
as desired.

This argument extends to the case of general $Q^{2N}$ by first deforming the symplectic form to make it integral, and then applying Donaldson's theorem to find a full flag of smooth symplectic divisors 
\begin{align*}
\{\qo\}  = Q_0 \subset Q_1 \subset \cdots \subset Q_N = Q,
\end{align*}
with $\dim Q_i = 2i$ for $i = 0,\dots,N$.
For suitable $J^\stab \in \calJ(M_\veca^\stab)$, the above argument can be applied iteratively to show that that inclusion map \eqref{eq:stab_I_inclusion} is again a regularity-preserving bijection, whence $\Tcount_{M\times Q,A \times [\pt]}^{(\veca,\vecb)} = \Tcount_{M,A}^\veca$.

We will explain the closely analogous argument for closed curves in detail in \S\ref{subsec:stab_inv_II}.
The case of closed curves is slightly cleaner since it does not require intersection theory for punctured curves, and it also implies Theorem~\ref{thm:ell_sup_stab} via the equivalence discussed in \S\ref{subsec:global_equiv}.

\subsection{Symplectic embedding obstructions}\label{subsec:from_rob_to_emb}

We end this section by discussing the relationship between robust moduli spaces and (stable) symplectic embedding obstructions.
\begin{prop}\label{prop:from_Tcount_to_obs}
Let $M^{2n}$ be a closed symplectic manifold, $A \in H_2(M)$ a homology class, and $\veca = (a_1,\dots,a_n) \in \R_{>0}^n$ a rationally independent tuple.
Assume that the moduli space $\calM_{M_\veca,A}(\orb^\veca_{c_1(A)-1})$ is robust and we have $\Tcount_{M,A}^\veca \neq 0$.
Then given any symplectic embedding $E(c\veca) \hooksymp M$ we must have
$c \leq \frac{[\omega_M] \cdot A}{\calA(\orb^\veca_{c_1(A)-1})}$.
\end{prop}
\begin{proof}
Let $\iota: E(c\veca) \hooksymp M$ be a symplectic embedding. By robustness of $\calM_{M_\veca,A}(\orb^\veca_{c_1(A)-1})$, for any generic $J \in \calJ(M \setminus \iota(\intE(c\veca))$ we have
\begin{align*}
\# \calM^J_{M\, \setminus\, \iota(\intE(c\veca)),A}(\orb^\veca_{c_1(A)-1}) = \#\calM_{M_\veca,A}(\orb^\veca_{c_1(A)-1}) \neq 0.
\end{align*}
Any curve $C \in \calM^J_{M \,\setminus\, \iota(\intE(c\veca)),A}(\Gamma)$ must have nonnegative energy, i.e. we have
\begin{align*}
0 \leq \En(C) = [\omega_M] \cdot A - c\calA(\orb^\veca_{c_1(A)-1}).
\end{align*}
\end{proof}

\sss

As for stable obstructions, we have:
\begin{cor}\label{cor:stab_obs}
Let $M^{2n}$ be a closed symplectic manifold, $A \in H_2(M)$ a homology class, and $\veca = (a_1,\dots,a_n) \in \R_{>0}^n$ a rationally independent tuple.
Let $Q^{2N}$ be another closed symplectic manifold and $\vecb = (b_1,\dots,b_N)$ another tuple satisfying $b_1,\dots,b_N > \calA(\orb^\veca_{c_1(A)-1})$.
Assume that the moduli spaces underlying $\Tcount_{M,A}^\veca$ and $\Tcount_{M \times Q, A \times [\pt]}^{(\veca,\vecb)}$ are both robustly defined and deformation invariant, and we have $\Tcount_{M,A}^\veca \neq 0$.
Then any symplectic embedding $E(c\veca,c\vecb) \hooksymp M \times Q$ must satisfy 
$c \leq \frac{[\omega_M] \cdot A}{\calA(\orb^\veca_{c_1(A)-1})}$.
\end{cor}
\begin{proof}
By Theorem~\ref{thm:ell_sup_stab} we have $\Tcount_{M \times Q,A \times [\pt]}^{(\veca,\vecb)} = \Tcount_{M,A}^\veca \neq 0$.
The result then follows by Proposition~\ref{prop:from_Tcount_to_obs}.
\end{proof}

Note that under the hypotheses of Corollary~\ref{cor:stab_obs}, any symplectic embedding of the form $\iota: E(c\veca) \times \C^N \hooksymp M \times \C^N$ must also satisfy $c \leq \frac{[\omega_M] \cdot A}{\calA(\orb^\veca_{c_1(A)-1})}$.
Indeed, we can restrict $\iota$ to an embedding of the form $E(c\veca,c\vecb) \hooksymp M \times \C^N$, and by compactifying the target space we get a symplectic embedding $E(c\veca,c\vecb) \hooksymp M \times Q$, after suitably scaling up the symplectic form on $Q$.
The claim then follows by applying Corollary~\ref{cor:stab_obs} to this embedding.

\section{Multidirectional tangency constraints}\label{sec:multi}

In this section we construct Gromov--Witten type invariants, denoted by  $N_{M,A}\lll \CC^\vecm \pt \rrr$, which count closed curves with local multidirectional tangency constraints.
Here $M^{2n}$ is a closed symplectic manifold, $A \in H_2(M)$ is a homology class, and $\vecm = (m_1,\dots,m_n) \in \Z_{\geq 1}$ is a tuple satisfying 
\begin{align}\label{eq:indmulti}
\sum_{s=1}^n m_s = c_1(A) + n - 2
\end{align}
 (so that we expect a finite count).
Theorem~\ref{thm:N_robust} establishes robustness of these counts under suitable assumptions, and Theorem~\ref{thm:equivalence} proves equivalence with the counts of the previous section.
In \S\ref{subsec:local_equiv} we import a technical tool from \cite[\S4]{SDEP} which will be used to relate multidirectional tangencies and ellipsoidal ends. In \S\ref{subsec:hc} we discuss ``hidden constraints'', essentially showing that multidirectional tangency constraints degenerate in the same way as ellipsoidal negative ends.
Using this, the proof of Theorem~\ref{thm:N_robust} is nearly identical mutatis mutandis to the proof of Theorem~\ref{thm:robust_counts}, after which Theorem~\ref{thm:equivalence} follows directly by making a special choice of almost complex structure.
Finally, in \S\ref{subsec:cusps_multi_sing} we discuss relations between multidirectional tangencies and singular curves, and in \S\ref{subsec:stab_inv_II} we establish invariance under stabilization.

\subsection{Multidirectional tangencies and ellipsoidal ends: local equivalence}\label{subsec:local_equiv}

We first discuss the local relationship between multidirectional tangency constraints and ellipsoidal negative ends.
This will be used in the next subsection to compare degenerations of multidirectional tangency constraints with degenerations of ellipsoidal ends.
In the following, we consider $\C^n$ with its standard integrable almost complex structure $J_\std$.
We take $\po = \vec{0} \in \C^n$, and the standard set of spanning local divisors at $\vec{0}$ given by $\vecD^\std = (\Ddiv_1^\std,\dots,\Ddiv_n^\std)$, with $\Ddiv_i^\std = \{z_i = 0\} \subset \C^n$ for $i = 1,\dots,n$. 
 \begin{notation}[\cite{SDEP}]
   For $\vecv = (v_1,\dots,v_n) \in \Z_{\geq 1}^n$, let $\orb^\veca_{-\vecv}$ denote the Reeb orbit in $\bdy E(\veca)$ given by $\simp_{i_m}^{v_{i_m}}$, where $i_m$ is the index $1 \leq i \leq n$ for which $a_iv_i$ is minimal.
 \end{notation}

The following proposition shows that $J_\std$ becomes SFT admissible under a suitable diffeomorphism, and for a $J_\std$-holomorphic curve we can explicitly describe the resulting Reeb orbit asymptotics.
Let $\D := \{|z| < 1\}$ denote the open unit disk in $\C$, and let $\Dpunc := \D \setminus \{0\}$ denote the result after puncturing the origin.

\begin{prop}[{\cite[\S4]{SDEP}}]\label{prop:Phi}
For each $\veca \in \R_{>0}^n$ rationally independent, there is an explicit diffeomorphism
\begin{align*}
\Phi_\veca: \R \times \bdy E(\veca) \ra \C^n \setminus \{\vec{0}\}
\end{align*}
such that $J_{\bdy E(\veca)} := (\Phi_\veca)^* J_\std$ lies in $\calJ(\bdy E(\veca))$.

Moreover, suppose that $u: \D \ra \C^n$ is a $J_\std$-holomorphic map which strictly satisfies $\lll \CC_{\vecD^\std}^\vecm \vec{0} \rrr$ at $0 \in \D$ for some $\vecm \in \Z_{\geq 1}^n$, and is otherwise disjoint from $\vec{0}$.
Then the $J_{\bdy E(\veca)}$-holomorphic map
\begin{align*}
\Phi_\veca^{-1} \circ u|_{\Dpunc}: \Dpunc \ra \R \times \bdy E(\veca)
\end{align*}
is negatively asymptotic to the Reeb orbit $\orb^\veca_{-\vecm}$ in $\bdy E(\veca)$ at the puncture $0$ .
\end{prop}
\NI Explicitly, we have
\begin{align}\label{eq:Phiveca}
\Phi_\veca(r,z_1,\dots,z_n) = \frac{(e^{2\pi r/a_1}\sqrt{a_1}z_1,\dots,e^{2\pi r/a_n}\sqrt{a_n}z_n)}{\sqrt{\pi \sum_{i=1}^n |z_i|^2}},
\end{align}
where $r$ is the coordinate on the first factor of $\R \times \bdy E(\veca)$ and $(z_1,\dots,z_n)$ are coordinates on $\C^n$.
The second statement in Proposition~\ref{prop:Phi} can be seen by analyzing the asymptotics as $|z| \ra 0$ 
of $(\Phi_\veca^{-1} \circ u)(z)$, with $u$ of the form $u(z) = (C_1z^{m_1},\dots,C_nz^{m_n})$.

For completeness let us elaborate on how the Reeb orbit $\orb^\veca_{-\vecm}$ arises in Proposition~\ref{prop:Phi}, refering the reader to \cite[\S4]{SDEP} for full details.\footnote{Here the seemingly superfluous minus sign in the notation $\orb^\veca_{-\vecm}$ is inherited from \cite{SDEP}, where the discussion applies uniformly to both positive and negative ellipsoidal ends, whereas in this paper we only consider negative ellipsoidal ends.}
 If $\pi:\R\times \bdy E(\veca)\to
\bdy E(\veca)$ is the projection, one
 can show that the map
$$
\Dpunc \ra  \bdy E(\veca),\quad 
z\mapsto (\pi\circ \Phi_\veca^{-1} \circ u)(z)
$$ 
is a composite  $z\mapsto G\circ {\rm fl}^{\rho_{\veca}(u(z))}(u(z))$, where
$G: \bdy E(\veca)\to  \bdy E(\veca)$ is a diffeomorphism, ${\rm fl}^t: \C^n\to \C^n, t\in \R,$ is the flow
$\vec z \mapsto 
(e^{2\pi t/a_1} z_1,\dots,  e^{2\pi t/a_n} z_n)$, and $\rho_{\veca}(w)$ is the unique real number such that ${\rm fl}^{\rho_{\veca}(w)}(w) \in \bdy E(\veca)$. It turns out that if we are only interested in the asymptotics of
$(\Phi_\veca^{-1} \circ u)(z)$ as $|z| \ra 0$ then we can 
 ignore $G$ and assume that $u(z) = (z^{m_1},\dots, z^{m_n})$, so that
   $\rho_{\veca}(u(s e^{i\theta}))$ depends only on the absolute values $s^{m_1}, \dots, s^{m_n}$ as $s\to 0$.  When $n=1$ we have $\rho_{\veca}(u(s e^{i\theta})) = r_1$ where $e^{2\pi r_1/a_1}s^{m_1} = \sqrt{a_1/\pi}$; that is, $2\pi r_1/a_1 = m_1|\log (s)|$ as $s \ra 0$, i.e. 
   $r_1 = \tfrac{a_1m_1}{2\pi} |\log(s)|$.    
   When $n>1$ the flow 
   is a product of the flows in each factor $\C_i$, and the time taken to flow the circle $({\op{im}}\, u)\cap \C_i$ back to $\bdy E(\veca)
   \cap \C_i$  is approximately $\tfrac{a_im_i}{2\pi}|\log(s)|$ for $s \approx 0$.  Thus 
   $\rho_{\veca}(u(z)) \approx \min\limits_{1 \leq i \leq n} \tfrac{a_im_i}{2\pi} |\log(s)|$, and, if this minimum is attained for $i=i_0$,  the limiting orbit is an $m_{i_0}$-fold cover of the circle $\bdy E(\veca)\cap \C_{i_0}$.

Proposition~\ref{prop:Phi} also has a more global analogue which applies to any closed symplectic manifold $M$ and which will form the basis of our proof of Theorem~\ref{thm:equivalence}. However, we must first establish Theorem~\ref{thm:N_robust}, namely that the counts $N_{M,A}\lll \CC^\vecm_\vecD \po\rrr$ are robustly defined.

\subsection{Hidden constraints for cuspidal degenerations}\label{subsec:hc}

The basic strategy for establishing robustness of moduli spaces $\calM_{M,A}\lll \CC^\vecm\po \rrr$ will closely parallel the argument we used for $\calM_{M_\veca,A}(\orb^\veca_{c_1(A)-1})$ in \S\ref{sec:robust}, i.e. we seek to show that no bad degenerations can occur for a generic one-parameter family of almost complex structures.
For punctured curves in $M_\veca$ we considered the SFT compactification by pseudoholomorphic buildings, and we ultimately showed that this agrees with the uncompactified moduli space.
Similarly, for closed curves in $M$ with multidirectional tangency constraints we will consider a compatification by stable maps, and seek to show that this in fact agrees with the uncompactified moduli space.

As a warmup, recall that the multidirectional tangency constraint $\lll \CC_\vecD^\vecm \po \rrr$ reduces to a unidirectional tangency constraint $\lll \T^{(m)}_{D_1}\po\rrr$ in the case $\vecm = (m,1,\dots,1)$, and robustness of the moduli spaces
$\calM_{M,A}\lll \T_{D_1}^{(m)}\po \rrr$ was established in \cite[\S2]{McDuffSiegel_counting} for $M^{2n}$ semipositive (here $m = c_1(A)-1$).
The main subtlety in establishing robustness comes from the possibility of ghost degenerations, since strictly speaking a marked point on a constant curve component is tangent to any local divisor through its image to infinite order.
This issue is resolved by observing that the nearby nonconstant curve components satisfy tangency conditions which collectively ``remember'' the initial constraint $\lll \T_{D_1}^{(m)}\po \rrr$.

The naive extension of this argument is actually insufficient to rule out ghost degenerations for multidirectional tangency constraints, as we now explain.
Fix a closed symplectic manifold $M^{2n}$, a homology class $A \in H_2(M)$, and $\vecm = (m_1,\dots,m_n) \in \Z_{\geq 1}^n$ with $\sum_{s=1}^m m_s = c_1(A) + n -2$.
Fix also a point $\po \in M$, a set $\vecD = (\Ddiv_1,\dots,\Ddiv_n)$ of spanning local divisors at $\po$, and an almost complex structure $J \in \calJ(M,\vecD)$.\footnote{More generally we could take a sequence $J_1,J_2,J_3,\dots \in \calJ(M,\vecD)$ converging to some $J_\infty \in \calJ(M,\vecD)$, but we will suppress this to keep the notation simpler.}
Let $C_1,C_2,C_3,\dots \in \calM_{M,A}^J\lll \CC_{\vecD}^\vecm\rrr$ be a sequence of curves which converges to some stable map $C_\infty \in \ovl{\calM}_{M,A}^J\lll \po \rrr$.
Let $Q_0$ be the component of $C_\infty$ which carries the constraint $\lll \CC_{\vecD}^\vecm\rrr$,
and suppose that $Q_0$ is a ghost component.
Let $G$ be the maximal ghost tree containing $Q_0$, i.e. the set of all ghost components in $C_\infty$ which are connected to $Q_0$ through ghost components.
Let $Q_1,\dots,Q_k$ (for some $k \in \Z_{\geq 2}$) denote the nonconstant curve components of $C_\infty$ which are nodally adjacent to a curve component of $G$, and let $z_1,\dots,z_k$ denote the corresponding special points of $Q_1,\dots,Q_k$ respectively.
For $i = 1,\dots,k$, $z_i$ strictly carries a constraint $\lll \CC_\vecD^{\vecm_i}\po\rrr$ for some $\vecm_i = (m^i_1,\dots,m^i_n) \in \Z_{\geq 1}^n$.

\begin{lemma}[{\cite[Lem. 7.2]{CM1}}]\label{lem:CM_ineq}
In the above situation we have
\begin{align}\label{eq:CM_ineq}
\sum_{i=1}^k m_s^i \geq m_s
\end{align}
for $s = 1,\dots,n$.
\end{lemma}

\begin{example}\label{ex:Q_cusp_deg}
 Suppose that $C_\infty$ consists precisely of the curve components $Q_0,\dots,Q_k$, with $k \in \Z_{\geq 2}$.
Using $\eqref{eq:CM_ineq}$ and the index zero assumption $\sum_{s=1}^n m_s = c_1(A) + n-2$, we have
\begin{align*}
\sum_{i=1}^k \ind(Q_i) &= \sum_{i=1}^k \left( 2c_1([Q_i] + 2n-4 - 2\sum_{s=1}^n m_s^i) \right)
\\&= 2c_1(A) + k(2n-4) - 2\sum_{i=1}^k\sum_{s=1}^n m_s^i
\\&\leq 2c_1(A) + k(2n-4) - 2\sum_{s=1}^n m_s
\\ &= (2n-4)(k-1).
\end{align*}
Even if we assume $\ind(Q_i) \geq 0$ for $i = 1,\dots,k$ (e.g. if each $Q_i$ is simple and $J \in \calJ(M,\vecD)$ is generic), this does not immediately give any contradiction.
\end{example}
\begin{example}
  Let us further specialize the previous example by taking $\vecm = (m,1,\dots,1)$ (i.e. the case of local tangency constraints).
Since each $m_s^i$ is at least $1$, for $s = 2,\dots,n$ we have $\sum_{i=1}^k m_s^i \geq k \geq 2 > 1 = m_s$,
and hence 
\begin{align*}
\sum_{i=1}^k\sum_{s=1}^n m_s^i \geq m + k(n-1).
\end{align*}
Using $m = c_1(A)-1$, this gives
\begin{align*}
\sum_{i=1}^k \ind(Q_i) \leq 2 - 2k \leq -2
\end{align*}
which {\em is} a contradiction at least if $\ind(Q_1),\dots,\ind(Q_k) \geq 0$.
\end{example}

Although Lemma~\ref{lem:CM_ineq} is generally insufficient for ruling out bad degenerations, the following proposition puts stronger restrictions on degenerations of multidirectional tangency constraints.
\begin{prop}\label{prop:hid_constr}
In the same situation as Lemma~\ref{lem:CM_ineq}, we have
\begin{align}\label{eq:hc} 
\sum_{i=1}^k \min\limits_{1 \leq s \leq n} a_s m^i_s \geq \min\limits_{1 \leq s \leq n}a_sm_s
\end{align}
for any choice of $\veca = (a_1,\dots,a_n) \in \R_{> 0}^n$.
\end{prop}
\begin{rmk}
Observe that \eqref{eq:hc} imples \eqref{eq:CM_ineq}.
Indeed by choosing $\veca = (a_1,\dots,a_n)$ so that $a_j = 1$ and $a_s \gg 1$ for $s \neq j$, we get $\min\limits_{1 \leq s \leq n}a_s m_s^i = m_j^i$ and $\min\limits_{1 \leq s \leq n}a_sm_s = m_j$.
\end{rmk}

\begin{proof}[Proof of Proposition~\ref{prop:hid_constr}]
  Recall that $J$ is assumed to be integrable near $\po$. 
  Pick local complex coordinates $z_1,\dots,z_n$ for $(M,J)$ centered at $\po$ and defined in some open neighborhood $U \subset M$ such that 
  $\Ddiv_s = \{z_s = 0\}$ for $s = 1,\dots,n$.
  Let $\zeta: U \ra \C^n$ denote the corresponding holomorphic chart.

For $i = 1,\dots,k$, the restriction of $Q_i$ to a small neighborhood of $z_i$ has image contained in $U$, with only $z_i$ mapping to $\po$.
By choosing a complex coordinate near $z_i$ we view this as a holomorphic map $\mr{Q}_i:
\Dpunc \ra U \setminus \{\po\}$.
Recalling the diffeomorphism $\Phi_\veca: \R \times \bdy E(\veca) \ra \C^n \setm \{\vec{0}\}$ from Proposition~\ref{prop:Phi}, we consider the $J_{\bdy E(\veca)}$-holomorphic composition
\begin{align*}
\eta_i :=  \Phi_\veca^{-1} \circ \zeta \circ \mr{Q}_i: \Dpunc \ra \R \times \bdy E(\veca),
\end{align*}
By Proposition~\ref{prop:Phi}, $\eta_i$ is negatively asymptotic at the puncture $0 \in \Dpunc$ to the Reeb orbit $\orb^\veca_{-\vecm_i}$ in $\bdy E(\veca)$, which has action $\calA(\orb^\veca_{-\vecm_i}) = \min\limits_{1 \leq s \leq n} a_s m^i_s$. 
In particular, for any $\eps > 0$ we can find $R \ll 0$ so that the loop $\ga_i$ given by restricting $\eta_i$ to the preimage of $\{r = R\}$ satisfies
\begin{align*}
\left|\int (\ga_i)^* \alpha_\veca - \min\limits_{1 \leq s \leq n}a_s m_s^i\right| < \eps.
\end{align*}
Here $\alpha_\veca$ denotes the standard contact one-form on $\bdy E(\veca)$, i.e. the restriction of the Liouville one-form $\la_\std = \tfrac{1}{2}\sum_{s=1}^n(x_idy_i - y_idx_i)$ on $\C^n$.

For $j \in \Z_{\geq 1}$ sufficiently large, we can restrict the curve $C_j$ to the preimage of $U \setminus \{\po\}$, then postcompose with $\Phi_\veca^{-1}\circ \zeta$, and finally restrict to the preimage of $\{r \leq R\}$ to obtain a $J_{\bdy E(\veca)}$-holomorphic map $\mu_j: \Sigma_j \ra \R \times \bdy E(\veca)$. Here $\Sigma_j$ is a Riemann surface with $k$ boundary circles $c^j_1,\dots,c^j_k$ and one interior puncture, such that
for $i = 1,\dots,k$ we have
\begin{align*}
\left|\int (\ga^j_i)^* \alpha_\veca - \min\limits_{1 \leq s \leq n}a_s m_s^i\right| < 2\eps,
\end{align*}
where $\ga^j_i$ denotes the restriction of $\mu_j$ to $c^j_i$.
Since by Proposition~\ref{prop:Phi} $\mu_j$ is negatively asymptotic at its interior puncture to the Reeb orbit $\orb^\veca_{-\vecm}$ in $\bdy E(\veca)$ of action
$\min\limits_{1 \leq s \leq n} a_sm_s$, we can find $R' \ll R$ so that the loop
$\rho_j$ given by restricting $\mu_j$ to $\{r = R'\}$ satisfies 
\begin{align*}
\left|\int \rho_j^* \alpha_\veca - \min\limits_{1 \leq s \leq n}a_s m_s\right| < \eps.
\end{align*}

Now put $(\Sigma_j)^R_{R'} := \mu_j^{-1}(\{R' \leq r \leq R\})$.
This is a Riemann surface with boundary circles $c_1^j,\dots,c_k^j$ mapping to $\{r = R\}$ and an additional boundary circle $c^j_0$ which maps to $\{r = R'\}$.
By Stokes' theorem and nonnegativity of energy, we have
\begin{align*}
0 \leq \int_{(\Sigma_j)^R_{R'}} \mu_j^*d\alpha_\veca &= \sum_{i=1}^k \int(\ga^j_i)^*\alpha_\veca - \int \rho_j^*\alpha_\veca 
\\&\leq \sum_{i=1}^k\min\limits_{1 \leq s \leq n}a_s m_s^i - \min\limits_{1 \leq  s \leq n}a_sm_s + (k+1)\eps.
\end{align*}
By $\eps > 0$ was arbitrarily small, this gives \eqref{eq:hc}.
\end{proof}

\begin{rmk}\label{rmk:formal_curve_from_cusp_deg}
  We can view Proposition~\ref{prop:hid_constr} as saying that whenever a constraint $\lll \CC^{\vecm}\pt\rrr$ degenerates into constraints $\lll \CC^{\vecm_1}\pt\rrr,\dots,\lll \CC^{\vecm_k}\pt\rrr$ there must exist a formal rational curve in $\bdy E(\veca)$ with positive ends $\orb^\veca_{-\vecm_1},\dots,\orb^\veca_{-\vecm_k}$ and negative end $\orb^\veca_{-\vecm}$, for any choice of (rationally independent) $\veca \in \R_{>0}^n$.
Recall that Conditions \ref{cond:B} and \ref{cond:C} put restrictions on formal curves in $\bdy E(\veca)$, and we have seen that these conditions hold e.g. under the assumptions of Lemma~\ref{lem:suff_for_assump_B}.
\end{rmk}

Another way to infer hidden constraints is via iterated blowups. Since the numerics become rather complicated we just illustrate this idea with a simple example.

\begin{example}
We will show that the constraint $\lll \CC^{(3,2)}\pt\rrr$ cannot degenerate into $\lll \CC^{(2,1)}\pt\rrr$ and $\lll \CC^{(1,1)}\pt\rrr$.
More precisely, consider a sequence of curves $C_1,C_2,C_3,\dots \in \calM_{M,B}^J\lll \CC^{(3,2)}_{(\Ddiv_1,\Ddiv_2)}\po \rrr$, each of which strictly satisfies the constraint, and suppose by contradiction that these converge to a curve $C_\infty$ consisting of two irreducible components $C_\infty^1,C_\infty^2$ which strictly satisfy the constraints $\lll \CC^{(3,2)}_{(\Ddiv_1,\Ddiv_2)}\po \rrr$ and $\lll \CC^{(3,2)}_{(\Ddiv_1,\Ddiv_1)}\po \rrr$ respectively.
We assume also that all of these curves are disjoint from $\po$ away from the main constraints.

After blowing up at $\po$, the proper transforms $\wt{C}_1,\wt{C}_2,\wt{C}_3,\dots$ all lie in homology class $B - 2e$, where $e = [\exc] \in H_2(\bl^1 M)$ is the homology class of the exceptional divisor $\exc$. 
After possibly passing to a subsequence, these converge to a curve $\wt{C}_\infty$ which projects to $C_\infty$, and hence consists of the proper transforms $\wt{C}_\infty^1,\wt{C}_\infty^2$ of $C_\infty^1,C_\infty^2$ respectively.
A priori $\wt{C}_\infty$ could also have some additional components which are covers of $\exc$, but this is ruled out since $[\wt{C}_\infty^i] = [C_\infty^i] - e$ for $i = 1,2$ and hence $[\wt{C}_\infty^1] + [\wt{C}_\infty^2] = B-2e$.
But this is a contradiction, since $\wt{C}_\infty^1$ and $\wt{C}_\infty^2$ are disjoint near $\exc$.

Suppose by contradiction that $C$ is a curve in homology class $B$ which strictly satisfies $\lll \CC^{(3,2)}_{(\Ddiv_1,\Ddiv_2)}\po\rrr$, and which degenerates into curves $C_1$ and $C_2$ which strictly satisfy $\lll \CC^{(2,1)}_{(\Ddiv_1,\Ddiv_2)}\po\rrr$ and $\lll \CC^{(1,1)}_{(\Ddiv_1,\Ddiv_2)}\po\rrr$.
Blowing up at $\po$, the proper transform $\wt{C}$ of $C$ lies in homology class $B - 2e$, while the proper transforms of $\wt{C}_1$ and $\wt{C}_2$ and $C_1$ and $C_2$ lie in homology classes $B_1 - e$ and $B_2 - e$ respectively.
Then $\wt{C}$ degenerates into $\wt{C}_1\cup \wt{C}_2$ along with some number of copies of (covers of) $\exc$.
However this is not possible since $\wt{C}_1$ and $\wt{C}_2$ are locally disjoint.
\end{example}


\subsection{Counting curves with multidirectional tangency constraints}

As before, let $M^{2n}$ be a closed symplectic manifold, $A \in H_2(M)$ a homology class, and $\vecm = (m_1,\dots,m_n) \in \Z_{\geq 1}^n$ a tuple satisfying $\sum_{s=1}^n m_s = c_1(A)+n-2$.
As usual, $\vecD = (\Ddiv_1,\dots,\Ddiv_n)$ denotes a collection of smooth local symplectic divisors which span at a point $\po \in M$.

\begin{definition}
 We will say that $\calM_{M,A}\lll \CC^\vecm \pt \rrr$ is {\bf robust} if $\calM_{M,A}^J\lll \CC^\vecm_\vecD \po\rrr$ is finite and regular for any choice of $\po,\vecD$, and generic $J \in \calJ(M,\vecD)$, and moreover the (signed) count 
 $\#\calM_{M,A}^J\lll \CC^\vecm_\vecD \po\rrr$ is independent of these choices. 
\end{definition}
\NI We will say a robust moduli space $\calM_{M,A}\lll \CC^\vecm \pt \rrr$ is {\bf deformation invariant} if it remains robust under deformations of the symplectic form on $M$, and moreover the count $\# \calM_{M,A}\lll \CC^\vecm \pt \rrr$ is unchanged under such deformations.

Given a tuple $\veca \in \R_{>0}^n$, recall that we have the lattice path $\Da^\veca_1,\Da^\veca_2,\Da^\veca_3,\dots \in \Z_{\geq 1}^n$ from Definition~\ref{def:Da}.
\begin{thm}\label{thm:N_robust}
Let $M^{2n}$ be a semipositive closed symplectic manifold, $A \in H_2(M)$ a homology class, $\vecD = (\Ddiv_1,\dots\Ddiv_n)$ a collection of spanning local divisors at a point $\po \in M$, 
and $\vecm = (m_1,\dots,m_n) \in \Z_{\geq 1}^n$ a tuple satisfying $\sum_{s=1}^n m_s = c_1(A) + n-2$.
Assume that there exists a tuple $\veca \in \R_{>0}^n$ satisfying Assumptions \ref{assump:A}, \ref{assump:B}, and \ref{assump:C}, and such that $\Da^\veca_{c_1(A)-1} = \vecm$.
Then the moduli space $\calM_{M,A}\lll \CC^{\vecm}\pt\rrr$ is robust and deformation invariant.
\end{thm}

\begin{rmk}\label{rmk:N_robust_implies_thm_B}
Theorem~\ref{thm:N_robust} implies Theorem~\ref{thmlet:multidir_counts} as follows. 
Taking $\veca = (p,q^+,a_3,\dots,a_n)$ with $a_3,\dots,a_n > pq$, we have $\Delta^\veca_{p+q-1} = (p,q,1,\dots,1)$ whence Assumptions \ref{assump:A} and \ref{assump:B} hold by Lemma~\ref{lem:suff_for_assump_B}. 
\end{rmk}


\begin{proof}[Proof sketch of Theorem~\ref{thm:N_robust}]

Given $J \in \calJ(M,\vecD)$, we denote by $\ovl{\calM}^J_{M,A}\lll \CC_\vecD^\vecm \po \rrr$ the closure of $\calM^J_{M,A}\lll \CC_\vecD^\vecm \po \rrr$ in the compactified moduli space $\ovl{\calM}_{M,A}^J\lll \po \rrr$ (i.e. the space of maps passing through $\po$).
Similarly, given a one-parameter family $\{J_t \in \calJ(M,\vecD)\;|\; t \in [0,1]\}$, we define $\ovl{\calM}_{M,A}^{\{J_t\}} \lll \CC_\vecD^\vecm \po \rrr$ to be the closure of $\calM_{M,A}^{\{J_t\}} \lll \CC_\vecD^\vecm \po \rrr$ in $\ovl{\calM}_{M,A}^{\{J_t\}} \lll \po \rrr$.
The main task is to establish
$\ovl{\calM}_{M,A}^{\{J_t\}} \lll \CC_\vecD^\vecm \po \rrr = \calM_{M,A}^{\{J_t\}} \lll \CC_\vecD^\vecm \po \rrr$ for generic $\{J_t\}$.

The most interesting point is to rule out degenerations as in Example~\ref{ex:Q_cusp_deg}, i.e. with $C_\infty \in \ovl{\calM}_{M,A}^{\{J_t\}} \lll \CC_\vecD^\vecm \po \rrr$ composed of a ghost component $Q_0$ and nonconstant components $Q_1,\dots,Q_n$ which carry respective constraints $\vecm_1,\dots,\vecm_k$.
Given such a degeneration, by Proposition~\ref{prop:hid_constr} there exists a formal curve in $\bdy E(\veca)$ with positive ends $\orb^\veca_{-\vecm_1},\dots,\orb^\veca_{-\vecm_k}$ and negative end $\orb^\veca_{-\vecm}$, 
where by assumption $\veca \in \R_{>0}^n$ satisfies Assumptions \ref{assump:A} and \ref{assump:B} and we have $\Da^\veca_{c_1(A)-1} = \vecm$ (c.f. Remark~\ref{rmk:formal_curve_from_cusp_deg}).

For the purposes of index calculations, we could view $Q_1,\dots,Q_k$ as curve components $\wt{Q}_1,\dots,\wt{Q}_k$ in $\wh{M}_\veca$, where $\wt{Q}_i$ has negative end $\orb^\veca_{-\vecm_i}$ for $i =1,\dots,k$.
On a purely formal level, note that the index of $\calM_{M_\veca,[Q_i]}(\orb^\veca_{-\vecm_i})$ is at least that of $\calM_{M,[Q_i]}\lll \CC^{\vecm_i} \pt\rrr$ for $i = 1,\dots,k$.
The upshot is that any index lower bound for the closed curve $Q_i$ is a fortiori true for its punctured counterpart $\wt{Q}_i$.
Also, similar to the arguments in \S\ref{subsec:proofs_of_lemmas}, we can assume by genericity of $\{J_t\}$ and standard transversality techniques that the underlying simple curves of $Q_1,\dots,Q_k$ have nonnegative  indices.
But now the existence of the configuration $\wt{Q}_0,\dots,\wt{Q}_k$ is ruled out essentially by the argument in \S\ref{subsec:pf_main_rob_res}.

The rest of the proof also closely follows that of Theorem~\ref{thm:robust_counts}, after formally trading closed curves with multidirectional tangency constraints in $M$ for punctured curves in $\wh{M}_\veca$, and ghost trees in $M$ for formal curves in $\bdy E(\veca)$. We leave the details to the reader.
\end{proof}

\subsection{Multidirectional tangencies and ellipsoidal ends: global equivalence}\label{subsec:global_equiv}

The following is our main result equating closed curves with multidirectional tangency constraints and punctured curves with ellipsoidal ends.

\begin{thm}\label{thm:equivalence}
Let $M^{2n}$ be a semipositive closed symplectic manifold, $A \in H_2(M)$ a homology class, and $\veca \in \R_{>0}^n$ a rationally independent tuple satisfying Assumptions \ref{assump:A}, \ref{assump:B}, and \ref{assump:C}.
Put $\vecm := \Da^\veca_{c_1(A)-1} \in \Z_{\geq 1}^n$.
Then we have
\begin{align*}
\Tcount_{M,A}^\veca = N_{M,A}\lll \CC^{\vecm}\pt\rrr.
\end{align*}
\end{thm}
\begin{rmk}
  Similar to Remark~\ref{rmk:N_robust_implies_thm_B}, Theorem~\ref{thm:equivalence} implies Theorem~\ref{thmlet:cusp=ell} via Lemma~\ref{lem:suff_for_assump_B}. 
\end{rmk}

One can imagine proving Theorem~\ref{thmlet:cusp=ell} by surrounding the constraint $\lll \CC^\vecm \pt \rrr$ with an ellipsoid of shape $\veca$ and stretching the neck. 
This would a detailed understanding of curves with local multidirectional tangency constraints in $\wh{E}(\veca)$ and their gluings. Instead, we will give a more direct correspondence
by establishing an explicit bijection
\begin{align*}
\calM_{M_\veca,A}^{J_E}(\orb^\veca_{c_1(A)-1}) \cong \calM_{M,A}^{J_C} \lll \CC^{\vecm}_{\vecD} \po\rrr 
\end{align*}
for certain special choices of $J_E \in \calJ(M_\veca)$ and $J_C \in \calJ(M,\vecD)$.

In the following, by {\bf complex Darboux coordinates} we mean local coordinates which simultaneously trivialize a symplectic form and an almost complex structure.
For a K\"ahler manifold, such coordinates exist if and only if the K\"ahler metric is flat, e.g. $\mathbb{T}^{2n}$ works but not $\CP^n$.

\begin{prop}[{\cite[\S4]{SDEP}}]\label{prop:Q}
Let $M^{2n}$ be a closed symplectic manifold and $J \in \calJ(M)$ a tame almost complex structure such that there exist complex Darboux coordinates $z_1,\dots,z_n$ centered at $\po$.
Then, for any given $\veca \in \R_{>0}^n$ rationally independent, there exists an explicit diffeomorphism
\begin{align*}
Q: \wh{M}_\veca \ra M \setminus \{\po\} 
\end{align*}
such that $Q^*J$ lies in $\calJ(M_\veca)$.

Moreover, put $\vecD = (\Ddiv_1,\dots,\Ddiv_n)$, with $\Ddiv_i = \{z_i = 0\}$ for $i = 1,\dots,n$,
suppose that $u: \CP^1 \ra M$ is a $J$-holomorphic map which 
strictly satisfies $\lll \CC_{\vecD}^\vecm \po \rrr$ at a point $z_0 \in \CP^1$ for some $\vecm \in \Z_{\geq 1}^n$, and is otherwise disjoint from $\po$.
Then the $Q^*J$-holomorphic map
\begin{align*}
\Phi_\veca^{-1} \circ u|_{\CP^1 \setminus \{z_0\}}: \CP^1 \setminus \{z_0\} \ra \wh{M}_\veca
\end{align*}
is negatively asymptotic to the Reeb orbit $\orb^\veca_{-\vecm}$ in $\bdy E(\veca)$ at the puncture $z_0$.
\end{prop}
\NI Here the map $Q$ is given by the formula in \eqref{eq:Phiveca} near the negative end of 
$\wh{M}_\veca$ and is slowly deformed to the identity away from this end.

We now complete the proof of Theorem~\ref{thm:equivalence}.
\begin{proof}[Proof of Theorem~\ref{thm:equivalence}]

Pick $J \in \calJ(M)$ as in Proposition~\ref{prop:Q}, which we can assume is generic away from $\po \in M$. 
Then Proposition~\ref{prop:Q} sets up a bijection
\begin{align}\label{eq:bij_from_Q}
\bigcup\limits_{\vecm'}\calM_{M,A}^J \lll \CC_{\vecD}^{\vecm'} \po \rrr \cong \calM^{Q^*J}_{M_\veca,A}(\orb^\veca_{c_1(A)-1}),
\end{align}
where the union is over all $\vecm' \in \Z_{\geq 1}^n$ satisfying $\orb^\veca_{-\vecm'} = \orb^\veca_{c_1(A)-1}$.
By 
definition, $\vecm = \Da^\veca_{c_1(A)-1}$ is the 
unique
tuple $\vecm' = (m_1',\dots,m_n')$ which minimizes $\sum_{s=1}^n m_s'$ subject to $\orb^\veca_{-\vecm'} = \orb^\veca_{c_1(A)-1}$.
It follows from \eqref{eq:indmulti} that
 for any other $\vecm'$ satisfying $\orb^\veca_{-\vecm'} = \orb^\veca_{c_1(A)-1}$ we have 
\begin{align*}
\ind\, \calM_{M,A}^J\lll \CC_\vecD^{\vecm'}\po\rrr \;\,<\;\, \ind\, \calM_{M,A}^J\lll \CC_{\vecD}^\vecm\po\rrr = 0,
\end{align*}
and hence $\calM_{M,A}\lll \CC_\vecD^{\vecm'}\po\rrr = \nil$ by our genericity assumption on $J$.
It follows that \eqref{eq:bij_from_Q} is in fact a bijection
\begin{align*}
\calM_{M,A}^J \lll \CC_{\vecD}^{\vecm} \po \rrr \cong \calM^{Q^*J}_{M_\veca,A}(\orb^\veca_{c_1(A)-1}),
\end{align*}
and hence we have
\begin{align*}
\Tcount_{M,A}^\veca = \#\calM_{M_\veca,A}^{Q^*J}(\orb^\veca_{c_1(A)-1}) = \#\calM_{M,A}^J\lll \CC_{\vecD}^{\vecm}\po \rrr = N_{M,A}\lll \CC^{\vecm}\pt\rrr.
\end{align*}

\end{proof}

\subsection{Cusps, multidirectional tangencies, and singular symplectic curves}\label{subsec:cusps_multi_sing}



We now elaborate on the relationship between cusp singularities and multidirectional tangency constraints.
Let $C \subset \C^2$ be an algebraic curve. 
A singular point $\po \in C$ is a {\bf $\mathbf{(p,q)}$ cusp} if its link is the $(p,q)$ torus knot, i.e. if for $\eps > 0$ sufficiently small there is a diffeomorphism
\begin{align*}
(S^3_\eps,C \cap S^3_\eps) \cong (S^3, \{x^p + y^q = 0\} \cap S^3),
\end{align*}
where $S^3_\eps \subset \C^2$ denotes the sphere of radius $\eps$ centered at $\po$.
More generally, any singularity of an algebraic curve in $\C^2$ has a link which is an iterated torus link (see \cite{eisenbud1985three}).

The singular curves most relevant to the ellipsoidal superpotential are those having one main singularity which is a $(p,q)$ cusp, and possibly some additional singularities which are ordinary double points (i.e. modeled on $\{x^2  + y^2 = 0\}$).
\begin{definition}\label{def:sesqsui_symp}
Given a symplectic four-manifold $M^4$ and $p,q \in \Z_{\geq 1}$ with $\gcd(p,q)=1$, a {\bf $\mathbf{(p,q)}$-sesquicuspidal symplectic curve} is a subset $C \subset M$ such that
\begin{itemize}
    \item there is a point $\po \in C$ with an open neighborhood $U \subset M$ such that $(U,C \cap U)$ is symplectomorphic to $(U',C' \cap U')$, where $C'$ is an algebraic curve in $\C^2$ having a $(p,q)$ cusp at $\po' \in C'$ and $U' \subset \C^2$ is an open neighborhood of $\po'$
    \item $C \setminus \{\po\}$ is a positively immersed symplectic submanifold of $M$.
  \end{itemize}    
\end{definition}

\begin{rmk}
Every sesquicuspidal curve is in particular a singular symplectic curve in the sense of \cite[Def. 2.5]{golla2019symplectic}.  
\end{rmk}

\begin{lemma}\label{lem:cusp_and_multi_equiv}
Let $C \subset \C^2$ be an irreducible local algebraic curve near a point $\po \in \C^2$, and fix $p,q \in \Z_{\geq 1}$ relatively prime.
\begin{enumerate}[label=(\alph*)]
  \item Suppose that $C$ strictly satisfies the constraint $\lll \CC^{(p,q)}_{(\Ddiv_1,\Ddiv_2)} \po\rrr$ for some smooth local holomorphic divisors $\Ddiv_1,\Ddiv_2$ which span at $\po$. Then $C$ has a $(p,q)$ cusp at $\po$.
  \item Conversely, suppose that $C$ has a $(p,q)$ cusp at $\po$. Then we can find a smooth local divisors $(\Ddiv_1,\Ddiv_2)$ intersecting transversely at $\po$ so that $C$ strictly satisfies the constraint
$\lll \CC^{(p,q)}_{(\Ddiv_1,\Ddiv_2)} \po\rrr$.
\end{enumerate}
\end{lemma}
\begin{rmk}  
If $C$ has a $(p,q)$ cusp at $\po$, then the multiplicity of $C$ at $\po$ is $\min \{p,q\}$, i.e. we have $(\Ddiv \cdot C)_{\po} \geq \min \{p,q\}$ for any divisor $\Ddiv$ passing through $\po$.
If say $p > q$ and we have $(\Ddiv_1 \cdot C)_{\po} = p$, then we have $(\Ddiv_2 \cdot C)_{\po} = q$ for {\em any} choice of smooth local divisor $\Ddiv_2$ which intersects $\Ddiv_1$ transversely at $\po$. In other words, the choice of $\Ddiv_2$ is essentially irrelevant.
\end{rmk}

\begin{proof}
To prove (a), let $x,y$ be complex coordinates for $\C^2$ centered at the origin such that 
  $\Ddiv_1 = \{x = 0\}$ and $\Ddiv_2 = \{y = 0\}$.
Recall that $C$ has a Newton--Puiseux parametrization of the form 
\begin{align*}
  y = c_1 x^{\frac{r_1}{p_1}} + c_2 x^{\frac{r_2}{p_1p_2}} + c_3 x^{\frac{r_3}{p_1p_2p_3}} + \cdots,
\end{align*}
 where
\begin{itemize}
  \item $c_i \in \C^*$ for $i \in \Z_{\geq 1}$
  \item the Puiseux pairs $(p_i,r_i) \in \Z_{\geq 1}^2$ are relatively prime for all $i \in \Z_{\geq 1}$
  \item there exists $L \in \Z_{\geq 1}$ such that $p_i = 1$ for all $i > L$
  \item we have increasing exponents $\tfrac{r_1}{p_1} < \tfrac{r_2}{p_1p_2} < \tfrac{r_3}{p_1p_2p_3} < \cdots$
\end{itemize}
(here we follow the conventions of \cite{neumann2017topology}).
The link of $C$ at $\po$ is then an iterated torus knot, with cabling parameters $(p_i,s_i)$ determined by $s_1 = r_1$ and $s_{i} = r_{i} - r_{i-1}p_{i}+p_{i-1}p_{i}s_{i-1}$ for $i \geq 2$.
This gives a parametrization of $C$ of the form
\begin{align}\label{eq:t_param}
t \mapsto (t^{p_1\cdots p_L},c_1 t^{r_1p_2\cdots p_L} + c_2t^{r_2p_3\cdots p_L} + \cdots).
\end{align}
and in particular we have intersection multiplicities
$p = C \cdot \Ddiv_1 = p_1\cdots p_L$ and $q = C \cdot \Ddiv_2 = r_1p_2 \cdots p_L$.
Since $p$ and $q$ are relatively prime, we must have $p_2 = \cdots = p_L = 1$, i.e. there is just a single cabling parameter $(p_1,s_1) = (p_1,r_1) = (p,q)$. 

To prove (b), let $x,y$ be complex coordinates for $\C^2$ and consider the parameterization of $C$ as in \eqref{eq:t_param}.
Observe that if $C$ has a $(p,q)$ cusp at $\po$ then there is at most one topologically nontrivial cabling parameter $(p_i,s_i)$, i.e for some $k \in \Z_{\geq 1}$ we have $p_i = 1$ for all $i \neq k$.
Then the parameterization \eqref{eq:t_param} takes the form
\begin{align}\label{eq:param_part_b}
t \mapsto (t^{p_1},c_1 t^{r_1p_1} + \cdots + c_{k-1}t^{r_{k-1}p_1} + c_kt^{r_k} + c_{k+1}t^{r_{k+1}} + \cdots ).
\end{align}
Assuming $p > q$ without loss of generality, 
then the multiplicity of $C$ at $\po$ is $q$, so 
we must have $p_1 = q$.
Note that $s_i = r_i$ for $i = 1,\dots,k$,
so have $r_k = s_k = p$, and \eqref{eq:param_part_b} becomes
\begin{align}\label{eq:param_part_b_simp}
t \mapsto (t^{q},c_1 t^{r_1q} + \cdots + c_{k-1}t^{r_{k-1}q} + c_kt^{p} +  \cdots ).
\end{align}
Then the polynomial
\begin{align*}
F(x,y) := y - c_1x^{r_1} - \cdots - c_{k-1}x^{r_{k-1}},
\end{align*}
has lowest degree term $t^p$ when composed with  the parameterization \eqref{eq:param_part_b_simp},
which means that $\Ddiv_1 := \{F(x,y) = 0\}$ satisfies $(\Ddiv_1 \cdot C)_{\po} = p$.
\end{proof}

\sss

We are now ready to complete the proof of Theorem~\ref{thmlet:sing_symp_curve}.

\begin{proof}[Proof of Theorem~\ref{thmlet:sing_symp_curve}]
Suppose first that $N_{M,A}\lll \CC^{(p,q)}\pt\rrr$ is nonzero.
 which means that for any choice of $\vecD = (\Ddiv_1,\Ddiv_2)$ and generic $J \in \calJ(M,\vecD)$ the moduli space $\calM_{M,A}^J \lll \CC_\vecD^{(p,q)} \po \rrr$ is nonempty.
By genericity any curve $C \in \calM_{M,A}^J \lll \CC_\vecD^{(p,q)} \po \rrr$ satisfies the constraint $\lll \CC_\vecD^{(p,q)}\po\rrr$ strictly, 
and hence has a $(p,q)$ cusp by Lemma~\ref{lem:cusp_and_multi_equiv}(a).  
After a small perturbation we can further assume that all singularities of $C$ away from the main cusp are positive ordinary double points.
   
Conversely, suppose that there is a $(p,q)$-sesquicuspidal symplectic curve $C$ in $M$ lying in homology class $A$ which is positively immersed away from the cusp point $\po$. 
Let $J$ be a tame almost complex structure on $M$ which is integrable near $\po$ and preserves $C$.
By Lemma~\ref{lem:cusp_and_multi_equiv}(b) we can find $\vecD = (\Ddiv_1,\Ddiv_2)$ such that $C$ strictly satisfies $\lll \CC_{\vecD}^{(p,q)}\po\rrr$.
Finally, by Proposition~\ref{prop:sesqui_lower_bd} below we have $N_{M,A}\lll \CC^{(p,q)}\po \rrr \geq 1$.
\end{proof}

Note that by combining Theorem~\ref{thmlet:sing_symp_curve} and the symplectic deformation invariance part of Theorem~\ref{thm:N_robust}, we have: 

\begin{cor}\label{cor:uni_symp_def_invt}
Let $M^4$ be a closed symplectic four-manifold, and suppose that $C$ is an index zero $(p,q)$-sesquicuspidal rational symplectic curve in $M$.
Then for any symplectic deformation $M'$ of $M$ there exists an index zero $(p,q)$-unicuspidal rational symplectic curve $C'$ in $M'$ satisfying $[C'] = [C]$.
\end{cor}

Lastly, the following Proposition~\ref{prop:sesqui_lower_bd} allows us to give lower bounds for the counts $N_{M,A}\lll \CC^\vecm \pt \rrr$ when $M$ is a smooth complex projective surface by constructing sesquicuspidal algebraic curves, even if the natural K\"ahler form on $M$ is not flat (e.g. $M = \CP^2$).

\begin{prop}\label{prop:sesqui_lower_bd}
Let $M^4$ be a closed four-dimensional symplectic manifold and $A \in H_2(M)$ a homology class such that $p+q = c_1(A)$ for some $p,q \in \Z_{\geq 1}$ relatively prime.
Let $\vecD = (\Ddiv_1,\Ddiv_2)$ be spanning local divisors at $\po \in M$, and let $J \in \calJ(M,\vecD)$ be any almost complex structure.
Suppose that there exist, for some $k \in \Z_{\geq 1}$, curves $C_1,\dots,C_k \in \calM_{M,A}^J\lll \CC_\vecD^{(p,q)} \po \rrr$ which are immersed away from the constraint.
Then we have $N_{M,A}\lll \CC^\vecm \pt \rrr \geq k$.
\end{prop}

\begin{proof}
We do not assume that $J$ is generic, but by a version of automatic transversality (see \cite{barraud2000courbes}) the curves $C_1,\dots,C_k$ are regular.
Then for $J' \in \calM(M,\vecD)$ any sufficiently small generic perturbation of $J$, the curves $C_1,\dots,C_k$ deform to curves 
$C_1',\dots,C_k' \in \calM^{J'}_{M,A}\lll \CC_\vecD^\vecm \po\rrr$.
Again by automatic transversality, all curves in $\calM^{J'}_{M,A}\lll \CC_\vecD^\vecm \po\rrr$ are regular and count positively (c.f. \cite[\S5.2]{mcduff2021symplectic}).
Since $J'$ is generic, we have by definition $N_{M,A}\lll \CC^\vecm \pt \rrr = \#\calM^{J'}_{M,A}\lll \CC_\vecD^\vecm \po\rrr \geq k$.
\end{proof}

\subsection{Stabilization invariance II}\label{subsec:stab_inv_II}

We end this section by discussing the closed curve counterpart of Theorem~\ref{thm:ell_sup_stab}:
\begin{thm}\label{thm:N_stab}
Let $M^{2n}$ be a closed symplectic manifold, $A \in H_2(M)$ a homology class, and $\vecm = (m_1,\dots,m_n) \in \Z_{\geq 1}$ a tuple such that $\sum_{s=1}^n m_s = c_1(A)+n-2$.
For $Q^{2N}$ another closed symplectic manifold, we have
\begin{align}\label{eq:N_stab}
N_{M\times Q,A\times [\pt]}\lll \CC^{(\vecm,\mathbb{1})} \pt \rrr = N_{M,A}\lll \CC^\vecm \pt \rrr,
\end{align}
provided that the moduli spaces underlying both sides of \eqref{eq:N_stab} are robust and deformation invariant.
\end{thm}
\NI Here $(\vecm,\mathbb{1})$ denotes the tuple $(m_1,\dots,m_n,\underbrace{1,\dots,1}_N)$ given by padding $\vecm$ with ones.
Together with the equivalence from \S\ref{subsec:global_equiv}, this imples Theorem~\ref{thm:ell_sup_stab}. 

\sss

To prove Theorem~\ref{thm:N_stab}, observe that, by the deformation invariance assumption, we can assume without loss of generality that the symplectic form on $Q$ is integral.
Then by Donaldson's theorem we can find a full flag of smooth symplectic divisors:
\begin{align*}
\{\qo\} = Q_0 \subset Q_1 \subset \cdots \subset Q_N = Q,
\end{align*}
where $\dim Q_i = 2i$ for $i = 0,\dots,N$.
Here $\qo$ is simply a point in $Q$, and we pick also a preferred basepoint $\po \in M$.

Let $\vecD = (\Ddiv_1,\dots,\Ddiv_{n+N})$ be a collection of spanning smooth local symplectic divisors at the point $(\po,\qo) \in M \times Q$, such that
$\Ddiv_1,\dots,\Ddiv_n$ are tangent to $M \times \{\qo\}$, and $\Ddiv_{n+i}$ is tangent to $M \times Q_i$ for $i = 1,\dots,N$.
As a shorthand, we put $\vecD_i = (\Ddiv_1,\dots,\Ddiv_{n+i})$ for $i = 0,\dots,N$.

Choose an admissible almost complex structure $J_N \in \calJ(M\times Q,\vecD)$ which preserves $M \times Q_i$ for $i = 0,\dots,N$ and is otherwise generic.
Then for $i = 0,\dots,N-1$ we have a natural inclusion map
\begin{align}\label{eq:stab_II_inc}
\calM_{M \times Q_i,A \times [\pt]}^{J_i}\lll \CC_{\vecD_i}^{(\vecm,\mathbb{1})} (\po,\qo) \rrr \subset \calM_{M \times Q_{i+1},A \times [\pt]}^{J_{i+1}} \lll \CC_{\vecD_{i+1}}^{(\vecm,\mathbb{1})} (\po,\qo)\rrr.
\end{align}
\NI Here $(\vecm,\mathbb{1})$ is shorthand for $(m_1,\dots,m_n,\underbrace{1,\dots,1}_i)$ when appearing in a constraint on a space of dimension $2n+2i$.

\begin{lemma}\label{lem:stab_II_inc_bij}
For $i = 0,\dots,N-1$, the inclusion map \eqref{eq:stab_II_inc} is a regularity-preserving bijection.
\end{lemma}
\begin{proof}
Suppose by contradiction that there is some curve 
\begin{align*}
C \in \calM_{M \times Q_{i+1},A \times [\pt]}^{J_{i+1}}\lll \CC_{\vecD_{i+1}}^{(\vecm,\mathbb{1})} (\po,\qo) \rrr \setminus \calM_{M \times Q_{i},A \times [\pt]}^{J_{i}}\lll \CC_{\vecD_{i}}^{(\vecm,\mathbb{1})} (\po,\qo) \rrr.
\end{align*}
Viewing $C$ as a curve in $M \times Q_{i+1}$, note that by positivity of intersections $C$ must have nonnegative homological intersection number with the $J_{i+1}$-holomorphic divisor $M \times Q_i \subset M \times Q_{i+1}$, and in fact this is positive since $C$ passes through the point $(\po,\qo) \in M \times Q_i$.
On the other hand, the classes $[C] = A \times [\pt] \in H_2(M \times Q_{i+1})$ and $[M \times Q_i] \in H_{2n+2i}(M \times Q_{i+1})$ evidently have trivial homological intersection number, so this is a contradiction.

It  
follows 
 that any curve $C \in \calM_{M \times Q_{i+1},A \times [\pt]}^{J_{i+1}}\lll \CC_{\vecD_{i+1}}^{(\vecm,\mathbb{1})} (\po,\qo) \rrr$ must be entirely contained in $M \times Q_i$, which shows that the inclusion \eqref{eq:stab_II_inc} is a bijection.
The fact that a regular curve in $\calM_{M \times Q_{i},A \times [\pt]}^{J_{i}}\lll \CC_{\vecD_{i}}^{(\vecm,\mathbb{1})} (\po,\qo) \rrr$ is also regular in $\calM_{M \times Q_{i+1},A \times [\pt]}^{J_{i+1}}\lll \CC_{\vecD_{i+1}}^{(\vecm,\mathbb{1})} (\po,\qo) \rrr$ follows by a straightforward adaptation of the argument in \cite[\S A]{mcduff2021symplectic}.
\end{proof}

\begin{proof}[Proof of Theorem~\ref{thm:N_stab}]
By applying Lemma~\ref{lem:stab_II_inc_bij} $N$ times, we find that the natural inclusion map
\begin{align*}
  \calM_{M,A}^{J_0}\lll \CC^{\vecm}_{\vecD_0} \po  \rrr\subset \calM_{M \times Q,A\times [\pt]}^{J_N}\lll \CC^{(\vecm,\mathbb{1})}_{\vecD_N} (\po,\qo)\rrr
\end{align*}
is a regularity-preserving bijection, whence we have
\begin{align*}
N_{M\times Q,A \times [\pt]}\lll \CC^{(\vecm,\mathbb{1})}\pt \rrr &= \# \calM^{J_N}_{M \times Q,A \times [\pt]}\lll \CC_{\vecD_N}^{(\vecm,\mathbb{1})}(\po,\qo) \rrr \\
&   = \# \calM^{J_0}_{M,A}\lll \CC_{\vecD_0}^{\vecm}\po \rrr = N_{M,A}\lll \CC^\vecm \pt \rrr.
\end{align*}
This completes the proof.
\end{proof}

\section{From perfect exceptional classes to unicuspidal curves}

We first discuss resolution of singularities in \S\ref{subsec:res_sing} from the point of view of the box diagram. In \S\ref{subsec:role_of_div} we elaborate on the role of the local divisor in guiding the resolution, with a view towards enumerative problems.
In \S\ref{subsec:rel_GW} we prove a correspondence theorem relating multidirectional tangency counts with relative Gromov--Witten invariants. In \S\ref{subsec:deg_to_nongen_bl} to prove Theorem~\ref{thmlet:per_exc} by a degeneration argument involving generic and nongeneric almost complex structures on blowups. Finally, in \S\ref{subsec:F_1} we apply these techniques to classify rigid unicuspidal curves in the first Hirzebruch surface.

\subsection{Resolution of singularities and the box diagram}\label{subsec:res_sing}

We begin by discussing some aspects of embedded resolution of singularities for curve singularities. In brief, blowing up a $(p,q)$ cusp results in a $(p-q,q)$ cusp,
and by repeatedly blowing up we arrive at a smooth resolution $\wt{C}$. The numerics of the  resulting chain of divisors can be neatly encoded in a device called the ``box diagram'', which also manifests connections with ellipsoidal ends.

In more detail, let $M^4$ be a closed symplectic four-manifold and $J \in \calJ(M)$ a tame almost complex structure which is integrable near $\po \in M$.
Let $C \subset M$ be a $J$-holomorphic curve with a $(p,q)$ cusp singularity at $\po$ for some $p > q \geq 2$ relatively prime.
We denote by $\bl^1 M$ the (complex analytic) blowup of $M$ at $\po$, and by $C^1$ the proper transform of $C$.
Let $\Ddiv_1,\Ddiv_2$ be smooth local $J$-holomorphic divisors in $M$ which intersect transversely at $\po$ such that $C$ strictly satisfies the constraint $\lll \CC_{(\Ddiv_1,\Ddiv_2)}^{(p,q)}\po\rrr$ (these exist by Lemma~\ref{lem:cusp_and_multi_equiv}(b)). 
Since $C$ has multiplicity $q$ at $\po$, $C^1$ has contact order $p-q$ with the proper transform $\Ddiv_1^1 \subset \bl^1M$ of $\Ddiv_1$, while $C^1$ is disjoint from the proper transform of $\Ddiv_2$.
We take $\Ddiv_2^1 := \exc \subset \bl^1 M$ to be the exceptional divisor resulting from the blowup $\bl^1M \ra M$.
Then $C^1$ strictly satisfies the constraint $\lll \CC_{(\Ddiv_1^1,\Ddiv_2^1)}^{(p^1,q^1)}\pp_1 \rrr$, where we put $\pp_1 := \Ddiv_1^1 \cap \Ddiv_2^1$ and $(p^1,q^1) := (p-q,q)$.
After possibly swapping $\Ddiv_1^1$ and $\Ddiv_2^1$, we can assume that $p^1 \geq q^1$. 

\begin{rmk}
  Note that $\{x^q =  y^p \} \subset \C^2$ has the parametrization $t\mapsto (t^p,t^q)$, and its blowup at the origin has the parametrization
$$
t\mapsto 
\bigl([t^p,t^q],t^p,t^q\bigr) = \bigl([t^{p-q}:1],t^p,t^q\bigr) =  \bigl([x(t):1],x(t)y(t),y(t)\bigr),
$$
where $(x(t),y(t)) = (t^{p-q},t^q)$, i.e. the blowup has a $(p-q,q)$ cusp at $([0:1],0,0)$.
\end{rmk}

We proceed by blowing up $\bl^1 M$ at $\pp_1$, giving a proper transform $C^2 \subset \bl^2 M$ which strictly satisfies $\lll \CC_{\Ddiv_1^2,\Ddiv_2^2}^{(p^2,q^2)}\pp_2 \rrr$, with $\pp_2 := \Ddiv_1^2 \cap \Ddiv_2^2$ and $(p^2,q^2) := (p^1-q^1,q^1)$.
Continuing in this manner, we eventually arrive, after say $K \in \Z_{\geq 1}$ blowups, at the {\bf minimal resolution}
$C^{K} \subset \bl^K M$, which is smooth.
After some additional blowups, say $L$ overall, we arrive at the {\bf normal crossings resolution} $C^L \subset \bl^L M$, which further satisfies that the total transform of $C$ in $\bl^LM$ (i.e. the preimage of $C$ under the blowup map $\bl^LM \ra M$) is a normal crossings divisor.

For $j = 1,\dots,L$, let $\F_j^j \subset \bl^j M$ denote the exceptional divisor which is the preimage of $\pp_{j-1}$ under the blowup map $\bl^j M \ra \bl^{j-1}M$, and let $\F_j^k$ denote its proper transform in $\bl^k M$ for $k = j+1,\dots,L$.
We also put $\F_i := \F_i^L$ for $i = 1,\dots, L$.
Note that each blowup adds a $\Z$ summand to the second homology, and we have a natural identification 
\begin{align*}
H_2(\bl^L M) \cong H_2(M) \oplus \Z\langle e_1,\dots,e_L\rangle,
\end{align*}
 with $e_i \cdot e_i = -1$ and $A \cdot e_i = e_j \cdot e_i = 0$ for all  $1 \leq i < j\leq L$ and $A \in H_2(M)$.

\sss

In the study of symplectic embeddings of four-dimensional ellipsoids $E(q,p)$, a key role is played by the weight sequence $\weight(p,q)$ \cite[\S1]{McDuff-Schlenk_embedding}, which is a sequence of positive integers associated to $(p,q)$.
The following pictorial perspective will be useful:
\begin{definition}
  Given $p,q \in \Z_{\geq 1}$ relatively prime, the {\bf box diagram} $\bx(p,q)$ is the unique decomposition into squares of the rectangle $\rect(p,q)$ with horizontonal length $p$ and vertical length $q$, subject to the following rules:
\begin{enumerate}
  \item for $p = q = 1$, we have $\bx(1,1) = \rect(1,1)$ (i.e. we take the trivial decomposition)
  \item if $p > q$, then $\bx(p,q)$ consists of the square $\rect(q,q)$ flush with the left side of $\rect(p,q)$, along with the decomposition of the remainder according to $\bx(p-q,q)$
  \item if $q > p$, then $\bx(p,q)$ consists of the square $\rect(p,p)$ flush with the bottom side of $\rect(p,q)$, along with the decomposition of the remainder according to $\bx(p,q-p)$.
\end{enumerate}
\end{definition}

\begin{figure}
    \centering
   \includegraphics[width=.825\textwidth]{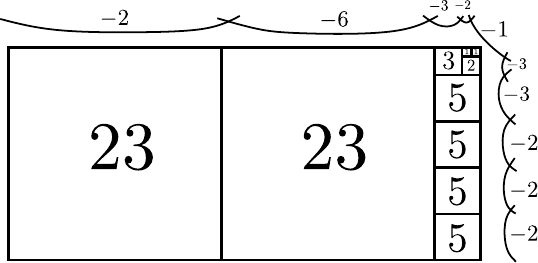} 

    \caption{Below: the box diagram $\bx(51,23)$. Above: the spheres $\F_1,\dots,\F_L$ (labeled by self-intersection numbers) arising in the resolution of the $(51,23)$ cusp singularity.}
    \label{fig:box}
\end{figure}

The squares in $\bx(p,q)$ are totally ordered by the rule that a square comes before any other square which lies to its right or above it, and we view the $i$th square as representing $\F_i$ for $i = 1,\dots,L$.
The {\bf weight sequence} of $(p,q)$ is then $\weight(p,q) := (m_1,\dots,m_L)$, where $m_i$ is the side length of the $i$th square in $\bx(p,q)$ for $i = 1,\dots,L$.
Note that each square in $\bx(p,q)$ except for the last one touches either the top side or the right side of $\rect(p,q)$, and we denote the corresponding spheres by $\F_1^\hor,\dots,\F_{L_\hor}^\hor$ and $\F_1^\ver,\dots,\F_{L_\ver}^\ver$ respectively.
Thus the preimage of $\po$ under the blowup map $\bl^L M \ra M$ is
$$\F_1 \cup \cdots \cup \F_L  = (\F_1^\hor \cup \cdots \cup \F_{L_\hor}^\hor) \cup \left(\F_1^\ver \cup \cdots \cup \F_{L_\ver}^\ver\right) \cup \F_L.$$

The following properties of the box diagram $\bx(p,q)$ are elementary to verify and show that it encodes the numerics of the resolution of a $(p,q)$ cusp.
\begin{lemma}(see also \cite[\S A]{McDuff-Schlenk_embedding} and \cite[Fig. 3.4]{mcduff2009symplectic})
\hfill
\begin{enumerate}[label=(\alph*)]
  \item For $i = 1,\dots,l$ we have $[\F_i] = e_i - c_{i+1}^i e_{i+1} - \cdots - c_L^i e_L$, where $c^i_j = 1$ if the $j$th square in $\bx(p,q)$ lies either immediately to the right of or immediately above the $i$th square, and $c^i_j = 0$ otherwise.
  \item For $1 \leq i < j \leq L_\hor$, $\F_i^\hor$ and $\F_j^\hor$ meet in a single point if $|j-i| = 1$, and $\F_i^\hor \cap \F_j^\hor = \nil$ otherwise. Similarly, for $1 \leq i < j \leq L_\ver$, $\F_i^\ver$ and $\F_j^\ver$ meet in a single point if $|j-i| = 1$, and $\F_i^\ver \cap \F_j^\ver = \nil$ otherwise.
  \item $\F_{L}$ meets each of $\F^\hor_{L_\hor},\F^\ver_{L_\ver}$ in a single point and is disjoint from $\F_1^\hor,\dots,\F^\hor_{L_\hor-1}$ and $\F_1^\ver,\dots,\F^\ver_{L_\ver-1}$
  \item The normal crossing resolution $C^L$ intersects $\F_L$ in one point and is disjoint from $\F_1,\dots,\F_{L-1}$.
\end{enumerate}
\end{lemma}

\NI In particular, the self-intersection number $[\F_i] \cdot [\F_i]$ is $-1-k$, where $k$ is the number of squares in $\bx(p,q)$ which lie immediately to the right of or immediately above the $i$th square, and we have $[\F_L] \cdot [\F_L] = -1$.
\begin{example}
  Figure~\ref{fig:box} shows the box diagram for $(51,23)$, which has weight sequence is $23,23,5,5,5,5,3,2,1,1$.
  The curves $\F_1^\hor,\dots,\F_4^\hor$ appear at the top, with respective self-intersection numbers $-2,-6,-3,-2$, and the curves $\F_1^\ver,\dots,\F_5^\ver$ appear at the right, with respective self-intersection numbers $-2,-2,-2,-3,-3$.
\end{example}

\begin{rmk}\label{rmk:wt_seq_cfe}
We can also write the weight sequence of $(p,q)$ without repetitions in the form $(w_1^{\times r_1},\dots,w_\ell^{\times r_\ell})$, with $w_1 > \cdots > w_\ell$. Then $(r_1,\dots,r_\ell)$ is precisely the continued fraction expansion coefficients of $p/q$, i.e. we have
\begin{align*}
p/q = [r_1,\dots,r_\ell] := 
r_1+\cfrac{1}{r_2 + \cfrac{1}{\ddots + \cfrac{1}{r_\ell}}}.
\end{align*}
\end{rmk}

\subsection{The role of the local divisor}\label{subsec:role_of_div}

In the above description of $\bl^L M$, the blowup points $\po,\pp_1,\dots,\pp_{L-1}$ all depend on the initial curve $C$, so a priori another curve $C'$ with a $(p,q)$ cusp at $\po$ would have a resolution living in a different blowup on $M$, thereby complicating our counting efforts.
The next lemma shows that in fact $\bl^L M$ depends only on 
a certain jet of the divisor $\Ddiv_1$.
As before we put $\weight(p,q) = (m_1,\dots,m_L) = (w_1^{\times r_1},\dots,w_\ell^{\times r_\ell})$ and assume $p \geq q$.
\begin{lemma} 
Let $\Ddiv,\Ddiv' \subset M$ be smooth local $J$-holomorphic divisors passing through $\po$ and having the same $r_1$-jet at $\po$.
Let $C,C' \subset M$ be $J$-holomorphic curves, each having a $(p,q)$ cusp at $\po$, and such that $(C \cdot \Ddiv)_{\po} = (C' \cdot \Ddiv')_{\po} = p$.
Then the sequence of blowups of $M$ which achieves the normal crossing resolution of $C$ is the same as that for $C'$.
\end{lemma}

\begin{proof}
We proceed by induction on $p+q$.
Here we allow the case $q = 1$, which corresponds to $C,C'$ being smooth but having contact order $p$ with $\Ddiv,\Ddiv'$ respectively.
For the base case we suppose $p = q = 1$, so that $C$ and $C'$ are both smooth and pass through $\po$ with multiplicity $1$.
In this case by convention the normal crossing resolution for both $C$ and $C'$ is obtained by blowing up $M$ once at $\po$.

For the inductive step, note that the first blowup for both $C$ and $C'$ occurs at $\po$. Let $\exc$ denote the resulting exceptional divisor, and let
$\wt{C}, \wt{C}', \wt{\Ddiv},  \wt{\Ddiv}'$ denote the respective proper transforms of $C, C', \Ddiv, \Ddiv'$.  
Since $C$ and $C'$ are tangent at $\po$, $\wt{C}$ and $ \wt{C}'$ intersect $\exc$ at the same point, say $\pp_1$. 

Suppose first that $p< 2q$, or equivalently $r_1 =1$, so that $\wt{C}$ and $ \wt{C}'$ both have a $(q,p - q)$ cusp at $\pp_1$,
and we have $\wt{C}\cdot \exc = \wt{C}'\cdot \exc  = q$.  
Then by the inductive hypothesis, with
$\wt{C}_1,\wt{C}_2$ playing the respective roles of $C,C'$ and $\exc$ playing the roles of both
$\Ddiv$ and $\Ddiv'$, 
 the remaining blowups for $\wt{C}$ and $\wt{C}'$ coincide.

Now suppose that $p>2q$, so that  $\wt{C}$ and $\wt{C}'$ each has a $(p-q,q)$ cusp at $\pp_1$,
and the weight sequence of $(p-q,q)$ is $(w_1^{\times (r_1-1)},\dots,w_\ell^{\times r_\ell})$. 
Since $C$ has multiplicity $q$ at $\po$, we have
\begin{align*}
 (\wt{C} \cdot \wt{\Ddiv})_{\pp_1} = (C \cdot \Ddiv)_{\pp_0} - q = p-q, 
\end{align*}
and similarly $\wt{C}' \cdot \wt{\Ddiv}' = p-q$.
Note also that $\wt{\Ddiv}$ and $\wt{\Ddiv}'$ have the same $(r_1-1)$-jet at $\pp_1$.
Therefore we may again apply the inductive hypothesis
to conclude that the remaining blowups for $\wt{C}$ and $\wt{C}'$ coincide.
\end{proof}

In light of the above lemma, the following notation is well-defined, i.e. independent of the choice of $C$.
As above we put $\weight(p,q) = (m_1,\dots,m_L)$, with $p > q$ relatively prime positive integers,
and let $M^4$ be a closed symplectic manifold and $J \in \calJ(M)$ a tame almost complex structure which is integrable near a point $\po \in M$.
\begin{notation}\label{not:res_p_q}
Put $\res_{(p,q)}^{\Ddiv,\po}(M,J) := \bl^LM$, where $\bl^LM$ is the iterated (complex analytic) blowup of $M$ which achieves the normal crossing resolution for a local $J$-holomorphic curve $C$ which has a $(p,q)$ cusp at $\po$ and satisfies $(C \cdot \Ddiv)_{\po} = p$.
\end{notation}
\NI We will sometimes use the shorthand $\res_{(p,q)}(M) := \res_{(p,q)}^{\Ddiv,\po}(M,J)$ when the data $\Ddiv,\po,J$ is implicit or immaterial.

Note that by construction $\res_{(p,q)}^{\Ddiv,\po}(M,J)$ inherits an almost complex structure $\wt{J}$ which is integrable near the spheres $\F_1,\dots,\F_L$ and preserves each of them.
In fact, we can also identify $\res_{(p,q)}^{\Ddiv,\po}(M,J)$ diffeomorphically with the corresponding symplectic blowup using small Darboux balls, and we thereby equip $\res_{(p,q)}^{\Ddiv,\po}(M,J)$ with a symplectic form which tames $\wt{J}$ (and is uniquely defined up to symplectic deformation).

\subsection{Relationship with relative Gromov--Witten theory}\label{subsec:rel_GW}

Following Notation~\ref{not:res_p_q}, put $\wt{M} := \res_{(p,q)}^{\Ddiv,\po}(M,J)$,
with almost complex structure $\wt{J}$.
Put also $\wt{A} = A - m_1e_1 - \cdots - m_Le_L \in H_2(\wt{M})$, where $\weight(p,q) = (m_1,\dots,m_L)$.
Note that we have $\wt{A} \cdot [\F_L] = 1$ and $\wt{A} \cdot [\F_L] = 0$ for $i =1,\dots,L-1$, so by positivity of intersections any $\wt{J}$-holomorphic curve in $\wt{M}$ in homology class $\wt{A}$ intersects $\F_L$ in one point and is disjoint from $\F_1,\dots,\F_{L-1}$.

Put $\Ddiv_1 = \Ddiv$, and let $\Ddiv_2$ be any smooth local $J$-holomorphic divisor in $M$ which passes through $\po$ and intersects $\Ddiv$ transversely.
By the above discussion, since any $J$-holomorphic curve in $M$ has a proper transform in $\wt{M}$ and conversely any $\wt{J}$-holomorphic curve in $\wt{M}$ can be projected to a curve in $M$, we have:
\begin{prop}\label{prop:res_of_sing_bij}
There is a natural bijective correspondence 
\begin{align*}
\calM_{M,A}^J \lll \CC^{(p,q)}_{(\Ddiv_1,\Ddiv_2)}\po\rrr \cong \calM_{\wt{M},\wt{A}}^{\wt{J}}.
\end{align*}
\end{prop}

\begin{cor}\label{cor:N_is_rel_GW}
 Let $M^4$ be a closed symplectic four-manifold and $A \in H_2(M)$ a homology class such that $c_1(A) = p+q$ for some $p,q\in \Z_{\geq 1}$ relatively prime. 
 Then the multidirectional tangency count $N_{M,A}\lll \CC^{(p,q)}\pt\rrr$ agrees with the genus zero symplectic Gromov--Witten invariant of $\wt{M}$ in homology class $\wt{A}$ relative to the norming crossing divisor $\F_1 \cup \cdots \cup \F_L$, with specified intersection pattern as above. 
\end{cor}

\begin{example}
  It is important to note that the relative Gromov--Witten invariant in Corollary~\ref{cor:N_is_rel_GW} is {\em not} generally equal to the absolute Gromov--Witten invariant $N_{\wt{M},\wt{A}}$ in the same homology class
   (although we see in \S\ref{subsec:deg_to_nongen_bl} that they do agree when $A$ is a perfect exceptional class, in which case both counts are equal to $1$).
As a simple example, $N_{\CP^2,3[L]}\lll\CC^{(8,1)}\pt\rrr$ coincides with the local tangency invariant $N_{\CP^2,3[L]}\lll \T^{(8)}\pt \rrr = 4$ from \cite{McDuffSiegel_counting}, whereas the corresponding absolute Gromov--Witten invariant is
$N_{\bl^8 \CP^2,3[L]-e_1-\cdots-e_8} = 12$.
Incidentally, the ``combining constraints'' formula in \cite[\S4.2]{McDuffSiegel_counting} 
describes precisely how the $12$ curves in $\calM_{\bl^8\CP^2,3[L]-e_1-\dots-e_8}^{\wt{J}'}$ degenerate as we deform a generic $\wt{J}' \in \calJ(\bl^8\CP^2)$ to the nongeneric blowup almost complex structure $\wt{J} \in \calJ(\bl^8\CP^2)$, with $4$ of them landing in $\calM_{\bl^8\CP^2,3[L]-e_1-\dots-e_8}^{\wt{J}}$ and the remaining $8$ limiting to reducible configurations in $\ovl{\calM}_{\bl^8\CP^2,3[L]-e_1-\dots-e_8}^{\wt{J}}$.
For a different approach to this question via curves with negative ellipsoidal ends see 
\cite[Prop.3.5.1]{Ghost}.
\end{example}

\begin{rmk}
Let $m_1^\hor,\dots,m_{L_\hor}^\hor$ and $m_1^\ver,\dots,m_{L_\ver}^\ver$ denote the self-intersection numbers of $\F_1^\hor,\dots,\F_{L_\hor}^\hor$ and $\F_1^\ver,\dots,\F_{L_\ver}^\ver$ respectively. As explained in \cite[\S2.1]{golla2019symplectic}, these give the  {\em negative} (a.k.a. Hirzebruch--Jung) continued fraction expansions of $\tfrac{p}{p-q}$ and $\tfrac{q}{q-r(p,q)}$ respectively, where $r(p,q)$ is the remainder of the division of $p$ by $q$. 
Namely, we have $\tfrac{p}{p-q} = [m_1^\hor,\dots,m_{L_\hor}^\hor]^-$ and $\tfrac{q}{q-r(p,q)} = [m_1^\ver,\dots,m_{L_\ver}^\ver]^-$, 
where in general we put
\begin{align*}
[c_1,\dots,c_k]^- := c_1-\cfrac{1}{c_2 - \cfrac{1}{\ddots - \cfrac{1}{c_k}}}.
\end{align*}
If $p/q = [c_1,\dots,c_k]^-$, then the cyclic quotient surface singularity\footnote{This is modeled on the quotient of $\C^2$ by the action $(z_1,z_2) \mapsto (e^{2\pi i/p}z_1,e^{2\pi i q/p}z_2)$ of the group of $p$th roots of unity.} $\tfrac{1}{p}(1,q)$ has a resolution which introduces a chain of spheres with self-intersection numbers $-c_1,\dots,-c_k$.
\end{rmk}
\begin{example}\label{ex:cyc_quo_res}
Let $\ell \subset \R^2$ denote the line passing through $(q,0)$ and $(0,p)$, and let $\Omega$ denote the set of points in the first quadrant $\R_{\geq 0}^2$ which lie on or above $\ell$.
Consider the noncompact toric surface $X_{(q,p)}$ with moment map polytope $\Omega$.
Then $X_{(q,p)}$ is a weighted blow up of $\C^2$ at the origin, and it has two cyclic quotient singularities $\tfrac{1}{p}(1,p-q)$ and $\tfrac{1}{q}(1,q-r(p,q))$.
We can resolve both of these by toric blowups, giving a smooth surface $\wt{X}_{(q,p)}$ having two chains of spheres with self-intersection numbers $(-m_1^\hor,\dots,-m_{L_\hor}^\hor)$ and $(-m_1^\ver,\dots,-m_{L_\ver}^\ver)$ respectively.
These chains are joined by the $(-1)$-sphere in $\wt{X}_{(q,p)}$ which is the proper transform of the toric divisor in $X_{(q,p)}$ lying over the slant edge in $\Omega$.
\end{example}

Observe that in Example~\ref{ex:cyc_quo_res} the resolved surface $\wt{X}_{(q,p)}$ bears striking resemblance to the result of embedded resolution of singularities for a curve with a $(p,q)$ cusp at the origin.
This offers the following perspective on Theorem~\ref{thmlet:cusp=ell}.
Put $\Ddiv_1^\std = \{z_1 = 0\}$ and $\Ddiv_2^\std = \{z_2 = 0\}$.
Given a local curve $C \subset \C^2$ which strictly satisfies $\lll \CC_{(\Ddiv_1^\std,\Ddiv_2^\std)}^{(p,q)}(0,0) \rrr$,
by iteratively blowing up we arrive at the normal crossing resolution $\wt{C} \subset \bl^L \C^2 \cong \wt{X}_{(q,p)}$,
which intersects the $(-1)$-sphere $\F_L$ in one point and is disjoint from the other spheres $\F_1,\dots,\F_{L-1}$.
Then after blowing down $\F_1,\dots,\F_{L-1}$ we get a smooth curve $C'$ in $X_{(q,p)}$ which is disjoint from the two orbifold points.
At least heuristically, this is akin to a curve in the negative symplectic completion of $\C^2 \setminus \intE(q,p)$ with negative end asymptotic to one of the Reeb orbits in the family from Remark~\ref{rmk:orbit_family}. 
This picture easily extends to any closed symplectic manifold $M^4$, noting that all relevant blowups can taken in a small neighborhood of the image of an ellipsoid embedding $E(\eps q,\eps p) \hooksymp M$ for $\eps > 0$ sufficiently small.

\subsection{Degenerating to the nongeneric blowup}\label{subsec:deg_to_nongen_bl}

The goal of this subsection is to prove Theorem~\ref{thmlet:per_exc}. As before, let $M^4$ be a closed symplectic manifold, and let $J \in \calJ(M)$ be a tame almost complex structure which is integrable near a point $\po \in M$.
Let $p > q$ be relatively prime positive integers, and put $\weight(p,q) = (m_1,\dots,m_L)$.
Let $\Ddiv \subset M$ be a smooth local $J$-holomorphic divisor passing through $\po$, and put
$\res_{(p,q)}(M) := \res_{(p,q)}^{\Ddiv,\po}(M,J)$.

Recall that we have the identification
\begin{align*}
H_2(\res_{(p,q)}(M)) \cong H_2(M) \oplus \mathbb{Z} \langle e_1,\dots,e_L\rangle.
\end{align*}
In particular, note that $H_2(\res_{(p,q)}(M))$ depends only on $H_2(M)$ and $(p,q)$ and not on the data $J,\Ddiv,\po$.
A general element in $H_2(\res_{(p,q)}(M))$ takes the form $A = B - k_1e_1 - \cdots - k_Le_L$ for some $k_1,\dots,k_L \in \Z$, with first Chern number $c_1(A) = c_1(B) - k_1 - \cdots - k_L$ and self-intersection number $A \cdot A = B \cdot B - k_1^2 - \cdots - k_L^2$.

\begin{definition}\label{def:exc_sphere}
A homology class $B \in H_2(M)$ is {\bf exceptional} if it satisfies $c_1(B) = 1$ and $B \cdot B = -1$, and $B$ is represented by a symplectically embedded two-sphere in $M$.
\end{definition}

\NI A priori the last condition in Definition~\ref{def:exc_sphere} is nontrivial to check, but for blowups of $\CP^2$ there is a purely combinatorial characterization in terms of elementary Cremona transformations (see \cite[Prop. 1.2.12]{McDuff-Schlenk_embedding}).

\begin{definition}
A homology class $A \in H_2(M)$ is {\bf $\mathbf{(p,q)}$-perfect exceptional}\footnote{This usage differs slightly from other sources.} if 
$\wt{A} := A - m_1e_1 - \cdots - m_Le_L \in H_2(\res_{(p,q)}(M))$ is an exceptional class.
\end{definition}
\NI Note that a $(p,q)$-perfect exceptional class $A$ satisfies $c_1(A) = p + q$.

A standard fact about exceptional homology classes $B \in H_2(M)$ is that the corresponding Gromov--Witten invariant $N_{M,B}$ is equal to $1$.
\begin{lemma}\label{lem:exc_GW_1}
Let $M^4$ be a closed symplectic four-manifold, and suppose that $B \in H_2(M)$ is an exceptional class. Then we have $N_{M,B} = 1$, and in particular $\ovl{\calM}_{M,B}^J \neq \nil$ for any $J \in \calJ(M)$. 
\end{lemma}
\NI The basic idea is that using $B \cdot B = -1$ and positivity of intersections there cannot be two irreducible $J$-holomorphic spheres in the same exceptional class. Meanwhile, using the fact that $B$ has a symplectically embedded representative we can engineer a choice of almost complex structure $J$ which preserves it.
Also, any irreducible pseudoholomorphic sphere in an exceptional class $B$ is necessarily nonsingular by the adjunction formula.
Note that we have $\ovl{\calM}_{M,B}^J = \calM_{M,B}^J$ for generic $J \in \calJ(M)$ (but not necessarily for arbitrary $J$).

\sss

Observe that the almost complex structure $\wt{J}$ on $\res_{(p,q)}(M)$ is {\em not} generic, since it preserves the spheres $\F_1,\dots,\F_{L-1}$ which all have negative index.
Therefore a priori an exceptional class $\wt{A} \in H_2(\res_{(p,q)}(M))$ might be represented by a degenerate configuration, in which case it does not necessarily blow down to a nice curve in $M$.
The next lemma, which is the final  ingredient to prove Theorem~\ref{thmlet:per_exc}, shows that this cannot occur provided that $J$ is generic away from $\po$.
\begin{lemma}\label{lem:per_exc_generic_J}
  Let $M^4$ be a closed symplectic four-manifold, and let $A \in H_2(M)$ be an $(p,q)$-perfect exceptional class for some $p > q$ relatively prime.
Let $\Ddiv$ be smooth local symplectic divisor passing through $\po \in M$, and suppose that $J \in \calJ(M)$ is integrable near $\po$, preserves $\Ddiv$, and is otherwise generic.
Put $\wt{M} := \res_{(p,q)}^{\Ddiv,\po}(M,J)$ with its induced almost complex structure $\wt{J} \in \calJ(\wt{M})$, and put $\wt{A} := A - m_1 e_1 - \cdots - m_Le_L$, where $\weight(p,q) = (m_1,\dots,m_L)$.
Then we have $\ovl{\calM}_{\wt{M},\wt{A}}^{\wt{J}} = \calM_{\wt{M},\wt{A}}^{\wt{J}}$.
\end{lemma}

Deferring the proof for the moment, we complete the proof of Theorem~\ref{thmlet:per_exc}.
\begin{proof}[Proof of Theorem~\ref{thmlet:per_exc}]
Pick any smooth local divisor divisors $\vecD = (\Ddiv_1,\Ddiv_2)$ which span at a point $\po$, and generic $J \in \calJ(M,\vecD)$.
We assume $p \geq q$ without loss of generality, and 
put $\wt{M} := \res_{(p,q)}^{\Ddiv_1,\po}(M,J)$ with its induced almost complex structure $\wt{J}$.

Suppose first that there is an index zero rational $(p,q)$-unicuspidal symplectic curve $C$ in $M$ in homology class $A$. By Proposition~\ref{prop:res_of_sing_bij} and Theorem~\ref{thmlet:sing_symp_curve} we have 
$$\#\calM^{\wt{J}}_{\wt{M},\wt{A}} = \# \calM_{M,A}^J \lll \CC^{(p,q)}_{\vecD}\po\rrr = N_{M,A}\lll \CC^{(p,q)}\pt\rrr > 0.$$
In particular, $\wt{A}$ is represented by a symplectically embedded sphere in $\wt{M}$.
The fact that $C$ has index zero translates into $c_1(\wt{A}) = 1$, while $C$ being $(p,q)$-unicuspidal translates into $\wt{A} \cdot \wt{A} = -1$. 

Conversely, suppose that $A \in H_2(M)$ is a $(p,q)$-perfect exceptional, i.e. 
$\wt{A} = A - m_1e_1 - \cdots - m_Le_L \in H_2(\wt{M})$ is an exceptional class.
By Lemma~\ref{lem:per_exc_generic_J} we have $\ovl{\calM}_{\wt{M},\wt{A}}^{\wt{J}} = \calM_{\wt{M},\wt{A}}^{\wt{J}}$, so by Lemma~\ref{lem:exc_GW_1} we have
\begin{align*}
\#\calM_{\wt{M},\wt{A}}^{\wt{J}} =  N_{\wt{M},\wt{A}} = 1.
\end{align*}
By Proposition~\ref{prop:res_of_sing_bij} and genericity of $J$ we thus have
\begin{align*}
N_{M,A}\lll \CC^{(p,q)}\pt \rrr = \# \calM_{M,A}^J \lll \CC_{\vecD}^{(p,q)}\po\rrr = \# \calM_{\wt{M},\wt{A}}^{\wt{J}} = 1.
\end{align*}
Finally, by Theorem~\ref{thmlet:sing_symp_curve} there exists a rational $(p,q)$-sesquicuspidal symplectic curve in $M$ lying in homology class $A$, and this is necessarily unicuspidal by the adjunction formula.

\end{proof}

\begin{rmk}
If $M^4$ is a smooth complex projective surface, it is not a priori clear whether the analogue of Lemma~\ref{lem:per_exc_generic_J} holds for its prefered integrable almost complex structure, which is not necessarily generic. However, if it does, then the argument in the proof of Theorem~\ref{thmlet:per_exc} shows that any perfect exceptional class corresponds to an {\em algebraic} unicuspidal curve.
\end{rmk}

\begin{proof}[Proof of Lemma~\ref{lem:per_exc_generic_J}]
As in \S\ref{subsec:res_sing}, let $\F_1,\dots,\F_L$
denote the $\wt{J}$-holomorphic spheres in $\wt{M}$ which project to $\po$ under the blowup map $\wt{M} \ra M$, with
$[\F_i] = e_i -c^i_{i+1}e_{i+1} - \cdots - c_L^i e_L$ with $c_j^i \in \{0,1\}$ for $1 \leq i < j \leq L$.

Given $C \in \ovl{\calM}_{\wt{M},A}^{\wt{J}}$, it suffices to show that $C$ is irreducible. 
Suppose by contradiction that $C$ has more than one component.
Since $\wt{J}$ is generic outside of a small neighborhood of $\F_1 \cup \cdots \cup \F_L$, by index considerations we can assume that at least one component of $C$ covers one of $\F_1,\dots,\F_L$.
Put $C = C_1 \cup C_2$, where no component of $C_1$ is contained in $\F_1 \cup \cdots \cup \F_L$ and each component of $C_2$ is contained in $\F_1 \cup \cdots \cup \F_L$, with $C_1,C_2$ both nonempty.

\begin{claim}
  We have $[C_2] \cdot \wt{A} \geq 0$.
\end{claim}
\begin{proof}
  It suffices to show that $[\F_i] \cdot \wt{A} \geq 0$ for $i = 1,\dots,L$.
Recall that we have $[\F_i] \cdot \wt{A} = 0$ for $i = 1,\dots,L-1$ and $[\F_L] \cdot \wt{A} = 1$.
\end{proof}

Let $\wt{J}' \in \calJ(\wt{M})$ be a sufficiently small generic perturbation of $\wt{J}$ (note that $\F_1,\dots,\F_L$ are {\em not} $\wt{J}'$-holomorphic).
Then we can deform $C_1$ to a $\wt{J}'$-holomorphic curve $C_1'$ with $[C_1'] = [C_1] \in H_2(\wt{M})$.
Indeed, since $\wt{J}$ is generic outside of a small neighborhood of $\F_1\cup\cdots\cup \F_m$, we can assume that the underlying simple curve of each component $C_1$ is regular, and hence unobstructed under small perturbations of $\wt{J}$.

Moreover, since $\wt{A} \in H_2(\wt{M})$ is an exceptional class, by Lemma~\ref{lem:exc_GW_1} and genericity of $\wt{J}'$ there exists a $\wt{J}'$-holomorphic sphere $C' \in \calM_{\wt{M},\wt{A}}^{\wt{J}'}$.
\begin{claim}
$C_1'$ must intersect $C'$ in isolated points, i.e. $C_1'$ does not have any component with the same image as $C'$.  
\end{claim}
\begin{proof}
We will argue using Lemma~\ref{lem:ce_i} below that no component of $C_1'$ can lie in a homology class which is a nonzero multiple of $\wt{A}$.
Firstly, note that no component of $C_1'$ can lie in a homology class of the form $ce_i$ for some $c \in \Z_{\geq 1}$ and $i \in \{1,\dots,L\}$. Indeed, by the way $C_1'$ was constructed we would then have a component $Q$ of $C_1$ of the same form, but
then we have 
\begin{align*}
[Q] \cdot [\F_i] = ce_i \cdot (e_i - c_{i+1}^ie_{i+1} - \cdots - c_L^ie_i) = -c,
\end{align*}
which violates positivity of intersections.

Now suppose by contradiction that $C_1'$ has a component $Q'$ such that $[Q'] = k\wt{A}$ for some nonzero $k \in \Z$, and let $B$ denote the total homology class of the remaining components of $C_1'$, so that we have
$[C_1'] = k\wt{A} + B$.
Note that the natural projection map $\pi: H_2(\wt{M}) \ra H_2(M)$ satisfies $\pi(\wt{A}) = A$ and $\pi([C_2]) = 0$ and hence $\pi([C_1']) = \pi([C_1]) = A$, so we have
\begin{align*}
c_1(A) = c_1(\pi([C_1'])) = kc_1(A) + c_1(\pi(B)).
\end{align*}
By Lemma~\ref{lem:ce_i} we have $c_1(A),c_1(\pi(B)),c_1(kA) \geq 1$, and hence $k \geq 1$, which is a contradiction.
\end{proof}

Finally, by positivity of intersections we have $[C_1] \cdot [C_2] \geq 0$ and $[C_1'] \cdot [C'] \geq 0$, and hence 
\begin{align*}
0 \leq [C_1'] \cdot [C'] = (\wt{A} - [C_2]) \cdot \wt{A} = -1 - [C_2] \cdot \wt{A}  \leq -1,
\end{align*}
which is a contradiction.
\end{proof}
\begin{lemma}\label{lem:ce_i}
  Any $\wt{J}'$-holomorphic curve in $\wt{M}$ represents a homology class which is either of the form (i) $ce_i$ for some $c \in \Z_{\geq 1}$ and $i \in \{1,\dots,L\}$, or else it is of the form (ii) $A' - k_1e_1 - \cdots - k_Le_L$, where $k_1,\dots,k_L \in \Z_{\geq 0}$ and $A' \in H_2(M)$ satisfies $c_1(A') \geq 1$.
\end{lemma}
\begin{proof}
Let $C_0$ be a $\wt{J}'$-holomorphic curve in $\wt{M}$, and put $[C_0] = A' - k_1e_1 - \cdots - k_Le_L$ for some $A' \in H_2(M)$ and $k_1,\dots,k_L \in \Z$. Evidently it suffices to prove the lemma in the case that $C_0$ is simple, whence $\ind(C_0) \geq 0$.
 Since $\wt{J}'$ is generic and $e_1,\dots,e_n \in H_2(\wt{M})$ are exceptional classes, there is a curve component $\exc_i \in \calM_{\wt{M},e_i}^{\wt{J}'}$ for $i = 1,\dots,L$.
If $C_0$ does not cover any of $\exc_1,\dots,\exc_L$ then by positivity of intersections we have
$0 \leq [C_0] \cdot [\exc_i] = k_i$ for $i=1,\dots,L$,
 so we have
\begin{align*}
0 \leq \ind(C_0) = -2 + 2c_1(A') - 2k_1 - \cdots - 2k_L,
\end{align*}
and hence $c_1(A') \geq 1 + k_1 + \cdots + k_L \geq 1$.
\end{proof}

\subsection{The case of the first Hirzebruch surface}\label{subsec:F_1}

Let $F_1$ denote the first Hirzebruch surface, i.e. the one-point blowup of $\CP^2$.
We can also view this as the toric K\"ahler manifold whose moment polygon has vertices $(0,0),(3,0),(1,2),(0,2)$. This choice of normalization makes the symplectic form monotone, although in light of Corollary~\ref{cor:uni_symp_def_invt} we will only be concerned with the symplectic form up to symplectic deformation.

Let $\ell = [L] \in H_2(F_1)$ denote the line class and $e = [\exc] \in H_2(F_1)$ the exceptional class, so that any $A \in H_2(F_1)$ takes the form $d\ell - me$ for some $d,m \in \Z$.
By Theorem~\ref{thmlet:per_exc}, index zero rational unicuspidal symplectic curves in $F_1$ are in bijective correspondence with perfect exceptional classes in $H_2(F_1)$.
Putting
\begin{align*}
\perf(F_1) := \{(p,q,d,m) \;|\; A = d\ell - me \;\text{is a}\;\text{$(p,q)$-perfect exceptional class}\},
\end{align*}
we can reformulate this as: 
\begin{cor}\label{cor:F_1_uni}
There is an index zero $(p,q)$-unicuspidal symplectic curve in $F_1$ in homology class $A = d\ell - me$ if and only if we have $(p,q,d,m) \in \perf(F_1)$.
\end{cor}
Put also
\begin{align*}
\perfbar(F_1) := \{(p,q)\;|\; (p,q,d,m) \in \perf(F_1) \text{ for some } d,m\}.
\end{align*}

A comprehensive description of $\perf(F_1)$ is given in \cite{magill2022staircase}, so together with Corollary~\ref{cor:F_1_uni} this describes all index zero $(p,q)$-unicuspidal rational symplectic curves in $F_1$.
Here we make just a few remarks about $\perf(F_1)$, referring the reader to loc. cit. for more details.
\begin{thm}[\cite{magillmcd2021,magill2022staircase}]\hfill
  \begin{itemize}
  \item The forgetful map $\perf(F_1) \ra \perfbar(F_1)$ is injective, i.e. for each $(p,q)$ there is at most one $(p,q,d,m) \in \perf(F)$. 
  More precisely, for any $(p,q,d,m) \in \perf(F_1)$ we must have
  $d = \tfrac{1}{8}(3p+3q+\eps t)$ and $m = \tfrac{1}{8}(p+q+3\eps t)$, where $t = \sqrt{p^2-6pq + q^2 + 8}$ and $\eps \in \{-1,1\}$ (and one can check that $d,m$ are integers for at most one value of $\eps$).
  \item If $(p,q) \in \perfbar(F_1)$, then we also have $S(p,q) := (6p-q,p) \in \perfbar(F_1)$.
  \item If $(p,q) \in \perfbar(F_1)$ with $p/q > 6$, then we also have $R(p,q) := (6p-35q,p-6q) \in \perfbar(F_1)$.
\end{itemize}
\end{thm}

\begin{rmk}
As we show in \cite{CuspsStairs}, the symmetry $S(p,q)$ has a natural geometric interpretation in terms of the generalized Orevkov twist, which is a birational transformation $F_1 \dashrightarrow F_1$. We currently do not of any geometric interpretation of the symmetry $R(p,q)$. It is not currently clear whether every index zero rational $(p,q)$-cuspidal symplectic curve in $F_1$ is realized by an algebraic curve.  
\end{rmk}

\bibliographystyle{math}
\bibliography{biblio}

\newcommand{\etalchar}[1]{$^{#1}$}
\begin{thebibliography}{FdBLMHN}

\bibitem[Bar]{barraud2000courbes}
Jean-Fran{\c{c}}ois Barraud.
\newblock {Courbes pseudo-holomorphes {\'e}quisingulieres en dimension 4}.
\newblock {\em Bulletin de la Soci{\'e}t{\'e} Math{\'e}matique de France} {\bf
  128}(2000), 179--206.

\bibitem[BEH{\etalchar{+}}]{BEHWZ}
Fr{\'e}d{\'e}ric Bourgeois, Yakov Eliashberg, Helmut Hofer, Krzysztof Wysocki,
  and Eduard Zehnder.
\newblock {Compactness results in symplectic field theory}.
\newblock {\em Geom. Topol.} {\bf 7}(2003), 799--888.

\bibitem[CH]{CaH}
Lucia Caporaso and Joe Harris.
\newblock {Counting plane curves of any genus}.
\newblock {\em Invent. Math.} {\bf 131}(1998), 345--392.

\bibitem[CM1]{CM1}
Kai Cieliebak and Klaus Mohnke.
\newblock {Symplectic hypersurfaces and transversality in {G}romov-{W}itten
  theory}.
\newblock {\em J. Symplectic Geom.} {\bf 5}(2007), 281--356.

\bibitem[CM2]{CM2}
Kai Cieliebak and Klaus Mohnke.
\newblock {Punctured holomorphic curves and {L}agrangian embeddings}.
\newblock {\em Invent. Math.} {\bf 212}(2018), 213--295.

\bibitem[CGHMP]{cristofaro2020infinite}
Dan Cristofaro-Gardiner, Tara~S Holm, Alessia Mandini, and Ana~Rita Pires.
\newblock {On infinite staircases in toric symplectic four-manifolds}.
\newblock {\em arXiv:2004.13062} (2020).

\bibitem[CGH]{CGH}
Daniel Cristofaro-Gardiner and Richard Hind.
\newblock {Symplectic embeddings of products}.
\newblock {\em Comment. Math. Helv.} {\bf 93}(2018), 1--32.

\bibitem[CGHM]{Ghost}
Daniel Cristofaro-Gardiner, Richard Hind, and Dusa McDuff.
\newblock {The ghost stairs stabilize to sharp symplectic embedding
  obstructions}.
\newblock {\em J. Topol.} {\bf 11}(2018), 309--378.

\bibitem[EN]{eisenbud1985three}
David Eisenbud and Walter~D Neumann.
\newblock {\em Three-dimensional link theory and invariants of plane curve
  singularities}.
\newblock Princeton University Press, 1985.

\bibitem[EG]{etnyre2020symplectic}
John~B Etnyre and Marco Golla.
\newblock {Symplectic hats}.
\newblock {\em arXiv:2001.08978} (2020).

\bibitem[FdBLMHN]{fernandez2006classification}
Javier Fern{\'a}ndez~de Bobadilla, Ignacio Luengo, Alejandro
  Melle~Hern{\'a}ndez, and Andras N{\'e}methi.
\newblock {Classification of rational unicuspidal projective curves whose
  singularities have one Puiseux pair}.
\newblock pages 31--45. Springer, 2006.

\bibitem[GS]{golla2019symplectic}
Marco Golla and Laura Starkston.
\newblock {The symplectic isotopy problem for rational cuspidal curves}.
\newblock {\em Compositio Mathematica} {\bf 158}(2022), 1595--1682.

\bibitem[GH]{Gutt-Hu}
Jean Gutt and Michael Hutchings.
\newblock {Symplectic capacities from positive {$S^1$}-equivariant symplectic
  homology}.
\newblock {\em Algebr. Geom. Topol.} {\bf 18}(2018), 3537--3600.

\bibitem[HK]{HK}
Richard Hind and Ely Kerman.
\newblock {New obstructions to symplectic embeddings}.
\newblock {\em Invent. Math.} {\bf 196}(2014), 383--452.

\bibitem[Hut1]{Hutchings_quantitative_ECH}
Michael Hutchings.
\newblock {Quantitative embedded contact homology}.
\newblock {\em J. Differential Geom.} {\bf 88}(2011), 231--266.

\bibitem[Hut2]{Hlect}
Michael Hutchings.
\newblock {Lecture notes on embedded contact homology}.
\newblock pages 389--484. Springer, 2014.

\bibitem[HT]{HuT}
Michael Hutchings and Clifford~Henry Taubes.
\newblock {Gluing pseudoholomorphic curves along branched covered cylinders.
  {I}}.
\newblock {\em J. Symplectic Geom.} {\bf 5}(2007), 43--137.

\bibitem[KV]{KV}
Joachim Kock and Israel Vainsencher.
\newblock {\em An invitation to quantum cohomology: Kontsevich's formula for
  rational plane curves}.
\newblock Springer Science \& Business Media, 2007.

\bibitem[KM]{kontsevich1994gromov}
Maxim Kontsevich and Yu~Manin.
\newblock {Gromov-Witten classes, quantum cohomology, and enumerative
  geometry}.
\newblock {\em Communications in Mathematical Physics} {\bf 164}(1994),
  525--562.

\bibitem[MM]{magillmcd2021}
Nicki Magill and Dusa McDuff.
\newblock {Staircase symmetries in Hirzebruch surfaces}.
\newblock {\em Algebr. Geom. Topol.} (To appear).

\bibitem[MMW]{magill2022staircase}
Nicki Magill, Dusa McDuff, and Morgan Weiler.
\newblock {Staircase patterns in {H}irzebruch surfaces}.
\newblock {\em arXiv:2203.06453} (2022).

\bibitem[McD1]{mcduff2009symplectic}
Dusa McDuff.
\newblock {Symplectic embeddings of 4-dimensional ellipsoids}.
\newblock {\em Journal of Topology} {\bf 2}(2009), 1--22.

\bibitem[McD2]{Mint}
Dusa McDuff.
\newblock {A remark on the stabilized symplectic embedding problem for
  ellipsoids}.
\newblock {\em Eur. J. Math.} {\bf 4}(2018), 356--371.

\bibitem[MS1]{JHOL}
Dusa McDuff and Dietmar Salamon.
\newblock {\em J-holomorphic curves and symplectic topology}.
\newblock American Mathematical Soc., 2012.

\bibitem[MS2]{McDuff-Schlenk_embedding}
Dusa McDuff and Felix Schlenk.
\newblock {The embedding capacity of 4-dimensional symplectic ellipsoids}.
\newblock {\em Ann. of Math. (2)} {\bf 175}(2012), 1191--1282.

\bibitem[MS3]{McDuffSiegel_counting}
Dusa McDuff and Kyler Siegel.
\newblock {Counting curves with local tangency constraints}.
\newblock {\em Journal of Topology} {\bf 14}(2021), 1176--1242.

\bibitem[MS4]{CuspsStairs}
Dusa McDuff and Kyler Siegel.
\newblock {Rational cuspidal curves and symplectic staircases}.
\newblock (In preparation).

\bibitem[MS5]{mcduff2021symplectic}
Dusa McDuff and Kyler Siegel.
\newblock {Symplectic capacities, unperturbed curves, and convex toric
  domains}.
\newblock {\em Geometry and topology} (To appear).

\bibitem[MS6]{SDEP}
Grigory Mikhalkin and Kyler Siegel.
\newblock {Ellipsoidal superpotentials and stationary descendants}.
\newblock {\em arXiv:2307.13252} (2023).

\bibitem[Neu]{neumann2017topology}
Walter~D Neumann.
\newblock {Topology of hypersurface singularities}.
\newblock {\em arXiv:1706.04386} (2017).

\bibitem[Ore]{orevkov2002rational}
S~Yu Orevkov.
\newblock {On rational cuspidal curves: I. {S}harp estimate for degree via
  multiplicities}.
\newblock {\em Mathematische Annalen} {\bf 324}(2002), 657--673.

\bibitem[Per]{pereira2022equivariant}
Miguel Pereira.
\newblock {Equivariant symplectic homology, linearized contact homology and the
  {L}agrangian capacity}.
\newblock {\em arXiv:2205.13381} (2022).

\bibitem[Sch]{Schlenk_old_and_new}
Felix Schlenk.
\newblock {Symplectic embedding problems, old and new}.
\newblock {\em Bull. Amer. Math. Soc. (N.S.)} {\bf 55}(2018), 139--182.

\bibitem[Ton]{tonk}
Dmitry Tonkonog.
\newblock {String topology with gravitational descendants, and periods of
  {L}andau-{G}inzburg potentials}.
\newblock {\em arXiv:1801.06921} (2018).

\bibitem[Wen1]{Wendl_aut}
Chris Wendl.
\newblock {Automatic transversality and orbifolds of punctured holomorphic
  curves in dimension four}.
\newblock {\em Comment. Math. Helv.} {\bf 85}(2010), 347--407.

\bibitem[Wen2]{wendl_SFT_notes}
Chris Wendl.
\newblock {Lectures on Symplectic Field Theory}.
\newblock
  \url{https://www.mathematik.hu-berlin.de/~wendl/Sommer2020/SFT/lecturenotes.pdf},
  2020.

\end{thebibliography}

\end{document}